\documentclass[a4paper,10pt]{amsart}
\usepackage{amsmath, amscd, amssymb, mathrsfs, 
  mathtools, svninfo, enumerate, fullpage, diagrams}
\usepackage[pagebackref]{hyperref}
\renewcommand{\ge}{\geqslant}
\renewcommand{\le}{\leqslant}

\svnInfo $Id: klz2.tex 377 2016-02-02 11:05:48Z davidloeffler $

\title{Rankin--Eisenstein classes in Coleman families}

\author{David Loeffler}
\address[Loeffler]{Mathematics Institute\\
Zeeman Building, University of Warwick\\
Coventry CV4 7AL, UK}
\email{d.a.loeffler@warwick.ac.uk}

\author{Sarah Livia Zerbes}
\address[Zerbes]{Department of Mathematics \\
University College London\\
Gower Street, London WC1E 6BT, UK}
\email{s.zerbes@ucl.ac.uk}

\thanks{The authors' research was supported by the following grants: Royal Society University Research Fellowship (Loeffler); Leverhulme Trust Research Fellowship ``Euler systems and Iwasawa theory'' (Zerbes). Parts of this paper were written while the authors were guests at MSRI in Berkeley, California, supported by NSF grant 0932078 000.}

\theoremstyle{plain}
    \newtheorem{theorem}{Theorem}[subsection]
    \newtheorem{lemma}[theorem]{Lemma}
    \newtheorem{proposition}[theorem]{Proposition}
    \newtheorem{corollary}[theorem]{Corollary}
    \newtheorem{lettertheorem}{Theorem}
\theoremstyle{definition}
    \newtheorem{definition}[theorem]{Definition}
    \newtheorem{notation}[theorem]{Notation}
    
    \newtheorem{assumption}[theorem]{Assumption}
\theoremstyle{remark}
    \newtheorem{remark}[theorem]{Remark}
    \newtheorem{note}[theorem]{Note}

\newtheorem*{remark*}{Remark}
\newtheorem*{convention}{Convention}

\DeclareMathOperator{\TSym}{TSym}
\DeclareMathOperator{\Sym}{Sym}
\DeclareMathOperator{\Hom}{Hom}

\DeclareMathOperator{\Fil}{Fil}
\DeclareMathOperator{\GL}{GL}

\DeclareMathOperator{\pr}{pr}
\DeclareMathOperator{\loc}{loc}
\DeclareMathOperator{\myPr}{Pr}\renewcommand{\Pr}{\myPr}

\DeclareMathOperator{\mom}{mom}
\DeclareMathOperator{\norm}{norm}

\DeclareMathOperator{\Gal}{Gal}
\DeclareMathOperator{\Res}{Res}

\DeclareMathOperator{\id}{id}

\newcommand{\EI}{\mathcal{EI}}
\newcommand{\RI}{\mathcal{RI}}

\newcommand{\dR}{\mathrm{dR}}
\newcommand{\rig}{\mathrm{rig}}

\newcommand{\et}{\text{\textup{\'et}}}

\newcommand{\la}{\mathrm{la}}
\newcommand{\alg}{\mathrm{alg}}

\newcommand{\Iw}{\mathrm{Iw}}
\newcommand{\cris}{\mathrm{cris}}

\newcommand{\frP}{\mathfrak{P}}

\newcommand{\sC}{\mathscr{C}}
\newcommand{\sR}{\mathscr{R}}

\newcommand{\sF}{\mathscr{F}}
\newcommand{\sH}{\mathscr{H}}

\newcommand{\cD}{\mathcal{D}}
\newcommand{\cE}{\mathcal{E}}
\newcommand{\cF}{\mathcal{F}}
\newcommand{\cG}{\mathcal{G}}
\newcommand{\cI}{\mathcal{I}}
\newcommand{\cO}{\mathcal{O}}

\newcommand{\cL}{\mathcal{L}}

\newcommand{\cW}{\mathcal{W}}

\renewcommand{\AA}{\mathbf{A}}

\newcommand{\CC}{\mathbf{C}}

\newcommand{\QQ}{\mathbf{Q}}
\newcommand{\RR}{\mathbf{R}}
\newcommand{\ZZ}{\mathbf{Z}}
\newcommand{\DD}{\mathbf{D}}

\DeclareFontFamily{U}{wncy}{}
\DeclareFontShape{U}{wncy}{m}{n}{<->wncyr10}{}
\DeclareSymbolFont{mcy}{U}{wncy}{m}{n}
\DeclareMathSymbol{\Sha}{\mathord}{mcy}{"58}

\newcommand{\Qp}{{\QQ_p}}
\newcommand{\Zp}{\ZZ_p}

\newcommand{\cEI}{{}_c\cE\cI}
\newcommand{\cRI}{{}_c\mathcal{RI}}
\newcommand{\cBF}{{}_c \mathcal{BF}}
\newcommand{\BF}{\mathcal{BF}}
\newcommand{\Eis}{\mathrm{Eis}}

\newcommand{\QsQ}{\QQ^S\!/\QQ}

\newcommand{\into}{\hookrightarrow}
\newcommand{\onto}{\twoheadrightarrow}
\newcommand{\htimes}{\mathop{\hat\otimes}}
\newcommand{\tbt}[4]{\begin{pmatrix}#1 & #2 \\ #3 & #4\end{pmatrix}}
\newcommand{\stbt}[4]{\left(\begin{smallmatrix}#1 & #2 \\ #3 & #4\end{smallmatrix}\right)}

\numberwithin{equation}{subsection}

\begin{document}

\begin{abstract}
 We show that the Euler system associated to Rankin--Selberg convolutions of modular forms, introduced in our earlier works with Lei and Kings, varies analytically as the modular forms vary in $p$-adic Coleman families. We prove an explicit reciprocity law for these families, and use this to prove cases of the Bloch--Kato conjecture for Rankin--Selberg convolutions.
\end{abstract}

\maketitle

\begin{center}
\emph{Dedicated to the memory of Robert F.~Coleman}
\end{center}
%\versioninfo.

% set this to 1 rather than 3 before submission!
\setcounter{tocdepth}{1}

\tableofcontents

\section{Introduction}

 Let $p > 2$ be a prime. The purpose of this paper is to study the $p$-adic interpolation of \emph{\'etale Rankin--Eisenstein classes}, which are Galois cohomology classes attached to pairs of modular forms $f, g$ of weights $\ge 2$, forming a ``cohomological avatar'' of the Rankin--Selberg $L$-function $L(f, g, s)$.
 
 In a previous work with Guido Kings \cite{KLZ1b}, we showed that these Rankin--Eisenstein classes for \emph{ordinary} modular forms $f, g$ interpolate in 3-parameter $p$-adic families, with $f$ and $g$ varying in Hida families and a third variable for twists by characters. We also proved an ``explicit reciprocity law'' relating certain specialisations of these families to critical values of Rankin--Selberg $L$-functions, with applications to the Birch--Swinnerton-Dyer conjecture for Artin twists of $p$-ordinary elliptic curves, extending earlier works of Bertolini--Darmon--Rotger \cite{BDR-BeilinsonFlach, BDR-BeilinsonFlach2}.
 
 In this paper, we generalise these results to non-ordinary modular forms $f, g$, replacing the Hida families by Coleman families:
 
 \begin{lettertheorem}
  Let $f, g$ be eigenforms of weights $\ge 2$ and levels $N_f, N_g$ coprime to $p$ whose Hecke polynomials at $p$ have distinct roots; and let $f_\alpha, g_\alpha$ be non-critical $p$-stabilisations of $f, g$. Let $\cF, \cG$ be Coleman families through $f_\alpha, g_\alpha$ (over some sufficiently small affinoid discs $V_1, V_2$ in weight space).
  
  Then there exist classes
  \[ \cBF^{[\cF, \cG]}_{m} \in H^1\left(\QQ(\mu_m), D^{\la}(\Gamma) \htimes M_{V_1}(\cF)^* \htimes M_{V_2}(\cG)^*\right)\]
  for each $m \ge 1$ coprime to $p$ and $c > 1$ coprime to $6pN_f N_g$, such that the specialisations of the classes $\cBF^{[\cF, \cG]}_{m}$ are the Rankin--Eisenstein classes for all specialisations of $\cF$ and $\cG$, and all characters of $\Gamma$ for which these classes are defined.
 \end{lettertheorem}
 
 Here $M_{V_1}(\cF)^*$ and $M_{V_2}(\cG)^*$ are families of Galois representations over $\cO(V_1)$ and $\cO(V_2)$ attached to $\cF$ and $\cG$, and $D^{\la}(\Gamma)$ is the algebra of distributions on the cyclotomic Galois group $\Gamma$. A slightly modified version of this theorem holds for weight 1 forms as well. For a precise statement, see Theorem \ref{thm:3varelt} below. 
  
 The proof of Theorem \ref{thm:3varelt} reveals some new phenomena which may be of independent interest; the Galois modules in which these classes lie are, in a natural way, \'etale counterparts of the modules of ``nearly overconvergent modular forms'' introduced by Urban \cite{Urban-nearly-overconvergent}.
 
 \begin{lettertheorem}
  The image of the class $\cBF^{[\cF, \cG]}_{1}$ under an appropriately defined Perrin-Riou ``big logarithm'' map is Urban's 3-variable $p$-adic Rankin--Selberg $L$-function for $\cF$ and $\cG$.
 \end{lettertheorem}
 
 See Theorem \ref{thm:explicitrecip} for a precise statement. In order to define the Perrin-Riou logarithm in this context, one needs to work with triangulations of $(\varphi, \Gamma)$-modules over the Robba ring; we use here results of Ruochuan Liu \cite{Liu-triangulation}, showing that the $(\varphi, \Gamma)$-modules of the Galois representations $M_{V_1}(\cF)^*$ and $M_{V_2}(\cG)^*$ admit canonical triangulations.
 
 Specialising this result at a point corresponding to a critical value of the Rankin--Selberg $L$-function, and applying the Euler system machine of Kolyvagin and Rubin, we obtain a case of the Bloch--Kato conjecture for Rankin convolutions:
 
 \begin{lettertheorem}[Theorem \ref{thm:BKSel} and Corollary \ref{cor:ptduality}]
  Let $f, g$ be eigenforms of levels coprime to $p$ and weights $r, r'$ respectively, with $1 \le r' < r$; and let $s$ be an integer such that $r' \le s \le r-1$ (equivalently, such that $L(f, g, s)$ is a critical value of the Rankin--Selberg $L$-function). Suppose $L(f, g, s) \ne 0$. Then, under certain technical hypotheses, the Bloch--Kato Selmer groups $H^1_{\mathrm{f}}(\QQ, M(f) \otimes M(g)(s))$ and  $H^1_{\mathrm{f}}(\QQ, M(f)^* \otimes M(g)^*(1-s))$ are both zero, where $M(f)$ and $M(g)$ are the $p$-adic representations attached to $f$ and $g$.
 \end{lettertheorem}
 
 One particularly interesting case is when $f = f_E$ is the modular form attached to an elliptic curve $E$, and $g$ is a weight 1 form corresponding to a 2-dimensional odd irreducible Artin representation $\rho$. In this case, the Bloch--Kato Selmer group $H^1_{\mathrm{f}}(\QQ, M(f) \otimes M(g)(1))$ is essentially the $\rho$-isotypical part of the $p$-Selmer group of $E$ over the splitting field of $\rho$, so we obtain new cases of the finiteness of Selmer (and hence Tate--Shafarevich) groups. See Theorem \ref{thm:sha} for the precise statement. 
 
 \begin{remark*}
  Since this paper was originally submitted, it has come to light that there are some unresolved technical issues in the paper \cite{Urban-nearly-overconvergent} upon which Theorem B, and hence Theorem C, relies. We hope that these issues will be resolved in the near future; as a temporary expedient, we have given in \S \ref{sect:appendix} below an alternate proof of a weaker form of Theorem B which avoids these problems, and thus suffices to give an unconditional proof of Theorem C.
 \end{remark*}
 
 This paper could not have existed without the tremendous legacy of mathematical ideas left by the late Robert Coleman. We use Coleman's work in three vital ways: firstly, Coleman was the first to construct the $p$-adic families of modular forms along which we interpolate; secondly, the Perrin-Riou big logarithm map is a generalisation of Coleman power series in classical Iwasawa theory (introduced in Coleman's Cambridge Part III dissertation); and finally, the results of \cite{KLZ1a} giving the link to values of $p$-adic $L$-functions, which are the main input to Theorem B, are proved using Coleman's $p$-adic integration theory. We are happy to dedicate this paper to the memory of Robert Coleman, and we hope that his work continues to inspire other mathematicians as it has inspired us.
  
 \subsection*{Acknowledgements}
 
 During the preparation of this paper, we benefitted from conversations with a number of people, notably Fabrizio Andreatta, Pierre Colmez, Hansheng Diao, Henri Darmon, Adrian Iovita, Guido Kings, Ruochuan Liu, Jay Pottharst, Karl Rubin and Chris Skinner. We would also like to thank the anonymous referee for several helpful comments and corrections.
 
 Large parts of the paper were written while the authors were visiting the Mathematical Sciences Research Institute in Berkeley, California, for the programme ``New Geometric Methods in Automorphic Forms'', and it is again a pleasure to thank MSRI for their support and the organisers of the programme for inviting us to participate.

%%%%%%%%%%%%%%%%%%%%%%%%%%%%%%%%%%%%%
%%%%%%%%%%%%%%%%%%%%%%%%%%%%%%%%%%%%%

\section{Analytic preliminaries}
 \label{sect:analyticprelim}
 The aim of this section is to extend some of the results of Appendix A.2 of \cite{LLZ14}, by giving a criterion for a collection of cohomology classes to be interpolated by a distribution-valued cohomology class.

 \subsection{Continuous cohomology}
  \label{sect:banach}

   We first collect some properties of Galois cohomology of profinite groups acting on ``large'' topological $\Zp$-modules (not necessarily finitely generated over $\Zp$). A very rich theory is available for groups $G$ satisfying some mild finiteness hypotheses (see e.g.~\cite[\S 1.1]{Pottharst-analytic}); but we will need to consider the Galois groups of infinite $p$-adic Lie extensions, which do not have good finiteness properties, so we shall proceed on a somewhat ad hoc basis, concentrating on $H^0$ and $H^1$.

  \begin{definition}
   \begin{enumerate}[(i)]
    \item If $G$ is a profinite group, a \emph{topological $G$-module} is an abelian topological group $M$ endowed with an action of $G$ which is (jointly) continuous as a map $G \times M \to M$.
    \item For $G$ and $M$ as in (i), we define the cohomology groups $H^*(G, M)$ as the cohomology of the usual complex of \emph{continuous} cochains $C^\bullet(G, M)$.
    \item We equip the groups $C^i(G, M) = \operatorname{Maps}(G^i, M)$ with the compact-open topology (equivalently, the topology of uniform convergence).
   \end{enumerate}
  \end{definition}

  With these definitions, the groups $C^*(G, -)$ define a functor from topological $G$-modules to complexes of \emph{topological} groups (i.e.~the topology is functorial in $M$, and the differentials $C^i(G, M) \to C^{i+1}(G, M)$ are continuous). Hence the cocycles $Z^i(G, M)$ are closed in $C^i(G, M)$. However, the cochains $B^i(G, M)$ need not be closed in general, so the quotient topology on the cohomology groups $H^i(G, M)$ may fail to be Hausdorff; and the subspace and quotient topologies on $B^i(G, M)$ may not agree. Our next goal is to show that these pathologies can be avoided for $i = 1$ and some special classes of modules $M$.

  Let $A$ be a Noetherian Banach algebra over $\Qp$. Then any finitely-generated $A$-module has a unique Banach space structure making it into a Banach $A$-module \cite[Proposition 3.7.3/3]{BGR}.

  \begin{proposition}
   \label{prop:boundariesclosed}
   Let $M$ be a finitely-generated free $A$-module, equipped with a continuous $A$-linear action of a profinite group $G$. Then:
   \begin{enumerate}
    \item the space $B^1(G, M)$ is closed in $Z^1(G, M)$;
    \item the subspace topology induced by $B^1(G, M) \into Z^1(G, M)$ coincides with the quotient topology induced by $M \onto B^1(G, M)$;
    \item the quotient map $M \onto B^1(G, M)$ has a continuous section (not necessarily $A$-linear or $G$-equivariant).
   \end{enumerate}
  \end{proposition}

  \begin{proof}
   We begin by noting that $Z^1(G, M)$ is, by definition, a closed subspace of the space $C^1(G, M)$ of continuous functions from $G$ to $M$, and since $M$ is Banach, the topology of $Z^1(G, M)$ is the Banach topology induced by the supremum norm on $C^1(G, M)$. However, if $M \cong A^{\oplus d}$ then we have
   \[ C^1(G, M) = C^1(G, \Qp) \htimes_{\Qp} M = C^1(G, \Qp)^{\oplus d} \htimes_{\Qp} A \]
   as a topological $A$-module. Since $C^1(G, \Qp)^{\oplus d}$ is orthonormalizable as a $\Qp$-Banach space (every $\Qp$-Banach space has this property), it follows that $C^1(G, M)$ is orthonormalizable as an $A$-Banach module, as orthonormalizability is preserved by base extension. However, $B^1(G, M)$ is manifestly finitely-generated as an $A$-module, and any finitely-generated submodule of an orthonormalizable $A$-Banach module is closed \cite[Lemma 2.8]{buzzard-eigenvarieties}. This proves (1).

   Parts (2) and (3) now follow from the open image theorem \cite[Proposition I.1.3]{colmez98}, which shows that any continuous surjective map between $\Qp$-Banach spaces has a continuous section (and, in particular, a continuous bijection between $\Qp$-Banach spaces must be a homeomorphism).
  \end{proof}

  \begin{remark}
   It seems likely that this result is true for any finitely-generated $A$-module $M$ with $G$-action (without assuming that $M$ be free) but we do not know how to prove this.
  \end{remark}

  \begin{definition}
   If $X$ and $Y$ are two $\Qp$-Banach spaces, let $\cL_w(X, Y)$ denote the space of continuous linear maps $X \to Y$ equipped with the weak topology (the topology of pointwise convergence).
  \end{definition}

  Now if $M$ is a $\Qp$-Banach space with a continuous action of a profinite group $G$, then $\cL_w(X, M)$ also acquires a continuous $G$-action by composition, for any Banach space $X$.

  \begin{proposition}
   Suppose the differential $d: M \to B^1(G, M)$ has a continuous section. Then the differential
   \[ \cL_w(X, M) \to B^1(G, \cL_w(X, M)) \]
   also has a continuous section, for any Banach space $X$.
  \end{proposition}

  \begin{proof}
   Let $\phi: B^1(G, M) \to M$ be a section. We use this to define $\tilde\phi: B^1(G, \cL_w(X, M)) \to \cL_w(X, M)$ as follows. Given $\sigma \in B^1(G, \cL_w(X, M))$, we may compose with an arbitrary $x \in X$ to obtain an element $\sigma_x \in B^1(G, M)$, and $\phi(\sigma_x)$ is then an element of $M$. This defines a map from $B^1(G, \cL_w(X, M))$ to the space of linear maps $X \to M$; however, for any $\mu \in B^1(G, \cL_w(X, M))$ we may write $\mu = dL$ for some continuous $L$, and we can then describe the image of $\mu(g)$ as the map obtained by composing $L$ with $M \rTo^d B^1(G, M) \rTo^\phi M$, which is thus continuous. This defines a continuous map $\tilde\phi$ such that the diagram
   \begin{diagram}
    B^1(G, \cL_w(X, M)) &\rTo^{\tilde\phi} & \cL_w(X, M) \\
    \dTo & & \dTo\\
    B^1(G, M)&\rTo^\phi & M
   \end{diagram}
   commutes for every $x \in X$. However, in order to show that the top horizontal arrow is continuous, it suffices (by the definition of the weak topology) to show that the diagonal composition is continuous for every $x$. Since the left vertical arrow is obviously continuous, and $\phi$ is continuous by assumption, this completes the proof.
  \end{proof}

  \begin{proposition}
   \label{prop:infres}
   If $M$ is a topological $G$-module, $H \trianglelefteq G$ is a closed subgroup, and there exists a continuous section $B^1(H, M) \to M$, then there is an exact sequence
   \[ 0 \to H^1(G/H, M^H) \to H^1(G, M) \to H^1(H, M)^{G/H} \to H^2(G/H, M^H). \]
  \end{proposition}

  \begin{proof}
   The injectivity of the first map, and the exactness at $H^1(G, M)$, are easily seen by a direct cocycle computation (which is valid for arbitrary topological $G$-modules).

   Exactness at $H^1(H, M)^{G/H}$ is much more subtle. Let $\sigma: H \to M$ be a continuous cocycle whose class $[\sigma] \in H^1(H, M)$ is $G$-invariant. Then, for any $g \in G$, the element $\sigma^g - \sigma$ lies in $B^1(H, M)$, where $\sigma^g$ is the cocycle $h \mapsto g \sigma(g^{-1} h g)$. This defines a continuous map $G \to B^1(H, M)$.

   By hypothesis, the differential $M \to B^1(H, M)$ has a continuous section. Composing this with the above map, we obtain a continuous map $\phi: G \to M$ such that $g\sigma(g^{-1}hg) - \sigma(h) = (h-1) \phi(g)$ for all $h\in H$ and $g \in G$. We may now argue as in the usual proof of the exactness of the inflation-restriction exact sequence for discrete modules \cite[Proposition 1.6.5]{NSW} to define a continuous 1-cochain $\tilde\sigma: G \to M$ such that $\tilde\sigma |_{H} = \sigma$ and $d\tilde\sigma \in Z^2(G/H, M^H)$, which gives exactness at $H^1(H, M)^{G/H}$.
  \end{proof}

  \begin{remark}
   The hypotheses of this proposition are satisfied, in particular, for any module of the form $M = \cL_w(X, N)$ where $X$ is any Banach space, $N$ is finitely-generated and free over a Noetherian Banach algebra $A$, and the group $H$ acts $A$-linearly on $N$ and trivially on $X$. This covers all the cases we shall need below.
  \end{remark}

%%%%%%%%%%%%%%%%%%%%%%%%%%%%%%%%%%%%%%%%%%%%%%%%%%

 \subsection{Distributions}

  For $\lambda \in \RR_{\ge 0}$, we define the Banach space $C_\lambda(\Zp, \Qp)$ of order $\lambda$ functions on $\Zp$ as in \cite{colmez-fonctions}. This has a Banach basis consisting of the functions $p^{\lfloor \lambda\ell(n)\rfloor} \binom{x}{n}$ for $n \ge 0$, where $\ell(n)$ denotes the smallest integer $L \ge 0$ such that $p^L > n$. We define $D_\lambda(\Zp, \Qp)$ as the continuous dual of $C_\lambda(\Zp, \Qp)$; for $f \in C_\lambda(\Zp, \Qp)$ and $\mu\in D_\lambda(\Zp, \Qp)$ we shall sometimes write $\int f\, \mathrm{d}\mu$ for the evaluation $\mu(f)$. The space $D_\lambda(\Zp, \Qp)$ has a standard norm defined by
  \[ \|\mu\|_{\lambda} = \sup_{n \ge 0} p^{-\lfloor \lambda\ell(n)\rfloor} \left\| \int_{x \in U} \binom{x}{n}\, \mathrm{d}\mu\right\|. \]

  \begin{proposition}
   \label{prop:normbound}
   For any integer $h \ge \lfloor \lambda \rfloor$, the standard norm on $D_\lambda(\Zp, \Qp)$ is equivalent to the norm defined by
   \[ \sup_{n \ge 0} \sup_{a \in \Zp} p^{-\lfloor \lambda n \rfloor}\left\| \int_{x \in a + p^n \Zp}\left( \frac{x-a}{p^n}\right)^h \, \mathrm{d}\mu\right\|.\]
  \end{proposition}

  \begin{proof}
   See \cite{colmez98}, Lemma II.2.5.
  \end{proof}

  As well as the Banach topology induced by the above norms (the so-called \emph{strong topology}), the space $D_\lambda(\Zp, \Qp)$ also has a \emph{weak topology}\footnote{This notation is somewhat misleading; it would be better to describe this as the \emph{weak-star topology}, and to reserve the term \emph{weak topology} for the topology on $D_\lambda(\Zp, \Qp)$ induced by its own continuous dual (for the strong topology), in line with the usual terminology in classical functional analysis. However, the above abuse of notation has become standard in the nonarchimedean theory, perhaps because the continuous duals of spaces such as $D_\lambda(\Zp, \Qp)$ are too pathological to be of much interest.}, which can be defined as the weakest topology making the evaluation maps $\mu \mapsto \int f\, \mathrm{d}\mu$ continuous for all $f \in C_\lambda(\Zp, \Qp)$.

  \begin{remark}
   The weak topology is much more useful for our purposes than the strong topology, since the natural map $\Zp \into D_0(\Zp, \Qp)$ given by mapping $a \in \Zp$ to the linear functional $f \mapsto f(a)$ is \emph{not} continuous in the strong topology, while it is obviously continuous in the weak topology.
  \end{remark}

  More generally, if $M$ is a $\Qp$-Banach space, we define $D_\lambda(\Zp, M) = \Hom_{\mathrm{cts}}(C_\lambda(\Zp, \Qp), M)$; as before, this has a strong topology induced by the operator norm (which we write as $\|-\|_\lambda$), and a weak topology given by pointwise convergence on $C_\lambda(\Zp, \Qp)$.
  
  \begin{proposition}
   \label{prop:banachsteinhaus}
   Let $X$ be a compact Hausdorff space, and $M$ a Banach space, and let $\sigma: X \to D_\lambda(\Zp, M)$ be a continuous map (with respect to the weak topology on $D_\lambda(\Zp, M)$). Then $\sup\{ \|\sigma(x)\|_{\lambda}: x \in X\} < \infty$.
  \end{proposition}
  
  \begin{proof}
   For each $f \in C_\lambda(\Zp, \Qp)$, the map $X \to M$ given by $x \mapsto \sigma(x)(f)$ is continuous, and hence bounded. By the Banach--Steinhaus theorem, this implies that the collection of linear maps $\{ \sigma(x) : x \in X\}$ is bounded in the uniform norm.
  \end{proof}
  
  \begin{definition}
   For $h\geq 0$, denote by $LP^{[0, h]}(\Zp, \Qp)$ the space of locally polynomial functions on $\Zp$ of degree $\le h$. If $M$ is a $\Qp$-vector space, write $D_{\alg}^{[0, h]}(\Zp,M)$ for the $\Qp$-linear homomorphisms of $LP^{[0, h]}(\Zp, \Qp)$ into $M$. 
  \end{definition}
  
  \begin{remark}
   An element $\mu\in LP^{[0, h]}(\Zp, \Qp)$ is uniquely determined by a collection of values
   $\int_{a+p^n\Zp}x^i\mu(x)$ for $i\in [0,h]$, $a\in\Zp$, $n\in\mathbf{N}$, 
   satisfying the compatibility relations
   \[ \int_{a+p^n\Zp}x^i\mu(x)=\sum_{k=0}^{p-1}\int_{a+kp^n+p^{n+1}\Zp}x^i\mu(x).\]
  \end{remark}

  \begin{lemma}
   \label{lemma:wkconvergence}
   Let $(\mu_n)_{n \ge 1}$ be a sequence of elements of $D_\lambda(\Zp, M)$ which is uniformly bounded (i.e.~there is a constant $C$ such that $\|\mu_n\|_{\lambda} \le C$ for all $n$), let $\mu \in D_\lambda(\Zp, M)$, and let $h \ge \lfloor \lambda \rfloor$ be an integer. If we have $\int f \, \mathrm{d}\mu_n \to \int f \, \mathrm{d}\mu$ as $n \to \infty$ for all $f\in LP^{[0, h]}(\Zp, \Qp)$, then $\mu_n \to \mu$ in the weak topology of $D_\lambda(\Zp, M)$.
  \end{lemma}

  \begin{proof}
   This is immediate from the density of $LP^{[0, h]}(\Zp, \Qp)$ in $C_\lambda(\Zp, \Qp)$.
  \end{proof}

  Finally, if $U$ is an open subset of $\Zp$, we define $D_\lambda(U, M)$ as the subspace of $D_\lambda(\Zp, M)$ consisting of distributions supported in $U$; this is closed (in both weak and strong topology).

 \subsection{Cohomology of distribution modules}

  We now apply the theory of the preceding sections in the context of representations of Galois groups. Our arguments are closely based on those used by Colmez \cite{colmez98} for local Galois representations, but also incorporating some ideas from Appendix A.2 of \cite{LLZ14}.

  We consider either of the two following settings: either $K$ is a finite extension of $\Qp$ and $G = \Gal(\overline{K} / K)$; or $K$ is a finite extension of $\QQ$ and $G = \Gal(K^S / K)$, where $K^S$ is the maximal extension of $K$ unramified outside some finite set of places $S$ including all infinite places and all places above $p$. In both cases we write $H^*(K, -)$ for $H^*(G, -)$; this notation is a little abusive in the global setting, but this should not cause any major confusion.

  We set $K_\infty = K(\mu_{p^\infty})$, and $H = \Gal(\overline{K} / K_\infty)$ (resp. $\Gal(K^S / K_\infty)$ in the global case). Thus $H$ is closed in $G$ and the cyclotomic character identifies $\Gamma = G / H$ with an open subset of $\Zp^\times$.

  \begin{remark}
   More generally, one may take for $K_\infty$ any abelian $p$-adic Lie extension of $K$ of dimension 1; see forthcoming work of Francesc Castella and Ming-Lun Hsieh for an application of this theory in the context of anticyclotomic extensions of imaginary quadratic fields.
  \end{remark}

  As in section \ref{sect:banach} above, we let $A$ be a a Noetherian $\Qp$-Banach algebra, and $M$ a finite free $A$-module with a continuous $A$-linear action of $H$; and we fix a choice of norm $\|\cdot\|_M$ on $M$ making it into a Banach $A$-module. We shall be concerned with the continuous cohomology $H^1(K_\infty, D_{\lambda}(\Gamma, M))$, where $D_{\lambda}(\Gamma, M)$ is equipped with the weak topology. Note that this cohomology group is endowed with a supremum seminorm, since every continuous cocycle $H \to D_\lambda(\Gamma, M)$ is bounded by Proposition \ref{prop:banachsteinhaus}.

  \begin{proposition}
   \label{prop:boundfordists}
   Let $\lambda \in \RR_{\ge 0}$. Then $H^1(K_\infty, D_\lambda(\Gamma, M))$ injects into $H^1(K_\infty, D_{\alg}^{[0, h]}(\Gamma, M))$ for any integer $h \ge \lfloor \lambda \rfloor$.

   An element $\mu \in H^1(K_\infty, D_{\alg}^{[0, h]}(K_\infty, M))$ is in the image of this injection if and only if the sequence
   \[ p^{-\lfloor \lambda n\rfloor} \sup_{\gamma \in \Gamma} \left\|  \int_{\gamma \Gamma_n} \left(\frac{\chi(x) - \chi(\gamma)}{p^n}\right)^h\, \mathrm{d}\mu\right\| \tag{$\star$} \]
   is bounded as $n \to \infty$, where $\left\|\cdot\right\|$ is the norm on $H^1(K_\infty, M)$ induced by the norm of $M$. Moreover, if this condition holds, we have
   \[ \|\mu\|_\lambda \le D \sup_{n \ge 0} p^{-\lfloor \lambda n\rfloor} \sup_{\gamma \in \Gamma} \left\|  \int_{\gamma \Gamma_n} \left(\frac{\chi(x) - \chi(\gamma)}{p^n}\right)^h\, \mathrm{d}\mu\right\|, \]
   where $\|\mu\|_\lambda$ is the supremum seminorm on $H^1(K_\infty, D_\lambda(\Gamma, M))$ and  $D$ is a constant independent of $K$ and $M$.
  \end{proposition}

  \begin{proof}
   For the injectivity, see Proposition II.2.1 of \cite{colmez98}, where this result is proved for arbitrary Banach representations $M$ such that $B^1(K_\infty, M)$ is closed in $Z^1(K_\infty, M)$; Proposition \ref{prop:boundariesclosed} shows that this is automatic under our present hypotheses on $M$. (The argument in \emph{op.cit.} is given for $K$ local, but it applies identically in the global case too.)

   To describe the image of this map, we follow the argument of Proposition II.2.3 of \emph{op.cit.} in which the result is shown for $A = \Qp$ and $K$ local. Exactly as in \emph{op.cit.}, given any class in $H^1(K_\infty, D_{\alg}^{[0, h]}(\Gamma, M))$ satisfying $(\star)$, then we may represent it by a cocycle $g \mapsto \mu(g)$ in $Z^1(K_\infty, D_{\alg}^{[0, h]}(\Gamma, M))$ which also satisfies $(\star)$ in the supremum norm. For each $h \in H$, we see that $\mu(h)$ lies in the image of $D_\lambda(\Gamma, M) \into D^{[0, h]}_{\alg}(\Gamma, M)$. Thus $\mu$ defines a cocycle on $H$ with values in $D_\lambda(\Gamma, M)$. Moreover, the values $\|\mu(h)\|_\lambda$ for $h \in H$ are bounded above by a constant multiple of the supremum of the sequence in $(\star)$, by Proposition \ref{prop:normbound}.

   It remains to check that the cocycle $g \mapsto \mu(g)$ is continuous (for the weak topology of $D_\lambda(\Gamma, M)$). This is asserted without proof \emph{loc.cit.}, and we are grateful to Pierre Colmez for explaining the argument. Since $H$ is a compact Hausdorff space, it suffices to show that for every convergent sequence $g_n \to g$, the sequence $\mu_n := \mu(g_n)$ converges to $\mu(g)$ in $D_\lambda(\Gamma, M)$. However, by construction we know that $\int f\, \mathrm{d}\mu_n$ converges to $\int f\, \mathrm{d}\mu$ for each $f \in LP^{[0,h]}(\Gamma, \Qp)$. Since the $\mu_n$ are uniformly bounded, Lemma \ref{lemma:wkconvergence} shows that they converge weakly to $\mu(g)$ as required.
  \end{proof}

  We now consider a special case of this statement. We impose the stronger assumption that $M$ is a continuous representation of the larger group $G = \Gal(\overline{K}/K)$ (resp. $\Gal(K^S/K)$ in the global case), rather than just of $H$. We equip $D_\lambda(\Gamma, M)$ with an action of $G$ by
  \[ \int_{x \in \Gamma} f(x)\, \mathrm{d}g(\mu) =g \left( \int_{x \in \Gamma} f([g]^{-1} x) \,\mathrm{d}\mu\right)\]
  where $[g]$ is the image of $g$ in $\Gamma$.

  \begin{proposition}
   \label{prop:unbounded-iwasawa}
   Let $\lambda \in \RR_{\ge 0}$, $h \ge \lfloor \lambda \rfloor$ an integer, and suppose we are given elements $x_{n, j} \in H^1(K_\infty, M)^{\Gamma_n = \chi^j}$, for all $n \ge 0$ and $0 \le j \le h$, satisfying the following conditions:
   \begin{itemize}
    \item For all $n \ge 0$, we have $\sum_{\gamma \in \Gamma_n / \Gamma_{n + 1}} \chi(\gamma)^{-j} \gamma \cdot x_{n+1, j} = x_{n, j}$.
    \item There is a constant $C$ such that
    \[ \left\| p^{-hn} \sum_{j = 0}^h (-1)^j \binom{h}{j} x_{n, j}\right\| \le Cp^{\lfloor \lambda n \rfloor} \]
    for all $n$.
   \end{itemize}
   Then there is a unique element $\mu \in H^1\left(K_\infty, D_\lambda(\Gamma, M)\right)^\Gamma$ satisfying
   \[ x_{n, j} = \int_{\Gamma_n} \chi^j \mu \]
   for all $n \ge 0$ and $0 \le j \le h$; and there is a constant $D$ independent of $K$ and of $M$ such that
   \[ \| \mu\|_\lambda \le C D, \]
   where $\|\mu\|_\lambda$ is the seminorm on $H^1(K_\infty, D_\lambda(\Gamma, M))$ induced by the norm of $D_\lambda(\Gamma, M)$.
  \end{proposition}

  \begin{proof}
   We claim first that there is a unique $\mu^{\mathrm{alg}} \in H^1(K_\infty, D_{\alg}^{[0, h]}(\Gamma, M))^\Gamma$ such that
   \[ x_{n, j} = \int_{\Gamma_n} \chi^j \mu^{\mathrm{alg}}. \]
   This follows from the fact that the functions $\phi_{n, j}(x) \coloneqq x^j \mathbf{1}_{1 + p^n \Zp}(x)$ for $n \ge 0$ and $0 \le j \le h$, and their translates under $\Gamma$, span the space $LP^{[0, h]}(\Gamma, \Qp)$.

   By Proposition \ref{prop:boundfordists}, the existence of the constant $C$ implies that $\mu^{\mathrm{alg}}$ is the image of a class $\mu \in H^1(K_\infty, D_\lambda(\Gamma, M))$, which must itself be $\Gamma$-invariant since the injection $ H^1(K_\infty, D_\lambda(\Gamma, M)) \into  H^1(K_\infty, D_{\alg}^{[0, h]}(\Gamma, M))$ commutes with the action of $\Gamma$. This proposition also shows that $\|\mu\|_\lambda$ is bounded above by $CD$.
  \end{proof}

  Using the inflation-restriction exact sequence (and the fact that $\Gamma$ has cohomological dimension 1) we see that $\mu$ lifts to a class in $H^1(K, D_{\lambda}(\Gamma, M))$. This lift is not necessarily unique, but it is unique modulo $H^1(\Gamma, D_\lambda(\Gamma, M^{G_{K_\infty}}))$ (and thus genuinely unique if $M^{G_{K_\infty}} = 0$).

%%%%%%%%%%%%%%%%%%%%%%%%%%%%%%%%%%%%%%%

  \subsection{Iwasawa cohomology}
  
  We now show that there is an interpretation of the module $H^1(K, D_\lambda(\Gamma, M))$ in terms of Iwasawa cohomology. Since the group $G$ has excellent finiteness properties (unlike its subgroup $H$), we have the general finite-generation and base-change results of \cite{Pottharst-analytic} at our disposal.
  
  We now assume that $A$ is a reduced affinoid algebra over $\Qp$. By a theorem of Chenevier (see \cite[Lemma 3.18]{Chenevier-application}) we may find a Banach-algebra norm on $A$, with associated unit ball $A^\circ = \{ a \in A : \|a\| \le 1\}$, and a compatible Banach $A$-module norm on $M$ with unit ball $M^\circ \subset M$, such that $G$ preserves $M^\circ$ and $M^\circ$ is locally free as an $A^\circ$-module.
  
  \begin{definition}
   We set
   \[ H^1_{\Iw}(K_\infty, M) = \left( \varprojlim_n H^1(K_n, M^\circ) \right)[1/p]. \]
  \end{definition}
  
  This is evidently independent of the choice of lattice $M^\circ$. 
  
  \begin{proposition}
   The module $H^1_{\Iw}(K_\infty, M)$ is finitely-generated over $D_0(\Gamma, A)$, and there are isomorphisms
   \begin{align*}
    H^1(K, D_0(\Gamma, M)) &\cong H^1_{\Iw}(K_\infty, M),\\
    H^1(K, D^{\mathrm{la}}(\Gamma, M)) &\cong D^{\mathrm{la}}(\Gamma, A) \otimes_{D_0(\Gamma, A)} H^1_{\Iw}(K_\infty, M).
   \end{align*}
  \end{proposition}
   
  \begin{proof}
   Let $A^\circ$ be as above. Then the ring $B^\circ = D_0(\Gamma, A^\circ) \cong A^\circ[[X]]$ is Noetherian, and it is complete and separated with respect to the ideal $I = (p, [\gamma] - 1)$, where $\gamma$ is a topological generator of $\Gamma / \Gamma_{\mathrm{tors}}$; moreover, $D_0(\Gamma, M^\circ) = B^\circ \otimes_{A^\circ} M^\circ$ is a flat $B^\circ$-module. Hence \cite[Theorem 1.1]{Pottharst-analytic} applies. By part (4) of the theorem, we see that $H^1(K, D_0(\Gamma, M^\circ))$ is finitely-generated over $D_0(\Gamma, A^\circ)$. Moreover, part (3) of the theorem shows that
   \[  
    H^1(K, D_0(\Gamma, M^\circ)) =  \varprojlim_m H^1(K_n,  D_0(\Gamma, M^\circ) / I^m),
   \]
    and every power $I^m$ contains the kernel of $D_0(\Gamma, A^\circ) \to A[\Gamma / \Gamma_n]$ for all sufficiently large $n$, so we also have an isomorphism
   \[ H^1(K, D_0(\Gamma, M^\circ)) = \varprojlim_n H^1(K_n, M^\circ \otimes_{A^\circ} A^\circ[\Gamma / \Gamma_n]) = H^1_{\Iw}(K_\infty, M^\circ), 
   \]
   where the last equality follows by Shapiro's lemma. Inverting $p$ we obtain the corresponding results with $A$-coefficients. Finally, we obtain the statement with locally analytic distributions by applying Theorem 1.9 of \emph{op.cit.} (in the case $n = \infty$). 
  \end{proof}
  
  \begin{corollary}
   In the above setting, for any $\lambda \in \RR_{\ge 0}$ there is a map
   \[ 
    H^1(K, D_{\lambda}(\Gamma, M)) \to 
    D^{\mathrm{la}}(\Gamma, A) \otimes_{D_0(\Gamma, A)} H^1_{\Iw}(K_\infty, M)
   \]
   compatible with the natural maps to $H^1(K, M(\chi^{-1}))$ for each character $\chi: M \to A^\times$.
  \end{corollary}
  
  \begin{proof}
   This follows from the fact that there is a continuous homomorphism $D_\lambda(\Gamma, A) \to D^{\mathrm{la}}(\Gamma, A)$, which gives (by the functoriality of continuous cohomology) a map
   \[  H^1(K, D_{\lambda}(\Gamma, M)) \to H^1(K, D^{\mathrm{la}}(\Gamma, M)).\]
   We now compose this with the second map from the previous proposition.
  \end{proof}
  
  \begin{proposition}
   \label{prop:iwacoho-unramified}
   If $K$ is a global field, then for every prime $v \ne p$, the inflation map
   \[ H^1(K_v^{\mathrm{nr}}, D^{\mathrm{la}}(\Gamma, M^{I_v})) \to H^1(K_v, D^{\mathrm{la}}(\Gamma, M))\]
   is an isomorphism.
  \end{proposition}
  
  \begin{proof}
   The corresponding statement for Iwasawa cohomology is well-known; and the result now follows by tensoring with $D^{\mathrm{la}}(\Gamma, A)$.
  \end{proof}
  
  A very slightly finer statement is possible if we consider coefficients in a field:
  
  \begin{proposition}
   \label{prop:colmeztautology}
   Suppose $V$ is a finite-dimensional $p$-adic representation of $G$. Then
   \[ H^1(K, D_\lambda(\Gamma, V)) = D_\lambda(\Gamma, \Qp) \otimes_{D_0(\Gamma, \Qp)} H^1_{\Iw}(K_\infty, V).\]
  \end{proposition}

  \begin{proof}
   In the local case, this surprisingly nontrivial result is Proposition II.3.1 of \cite{colmez98}. The proof relies on local Tate duality at one point, so we shall explain briefly how this can be removed in order to obtain the result in the global case as well.

   Firstly, from the finite generation of $H^2_{\Iw}(K_\infty, V)$ as a $\Lambda(\Gamma)$-module, there exists a $k$ such that $H^2_{\Iw}(K_\infty, V(k))^{\Gamma} = 0$. We may suppose (by twisting) that we have, in fact, $H^2_{\Iw}(K_\infty, V)^{\Gamma} = 0$.

   Let $\nu_n = (\gamma - 1)^n$ where $\gamma$ is a topological generator of $\Gamma$, and let $T$ be a lattice in $V$. Then the submodules $H^2_{\Iw}(K_\infty, T)[\nu_n]$ are an ascending sequence of $\Lambda(\Gamma)$-submodules of the finitely-generated module $H^2_{\Iw}(K_\infty, T)$. Since $\Lambda(\Gamma)$ is Noetherian and $H^2_{\Iw}(K_\infty, T)$ is finitely-generated, we conclude that this sequence of modules must eventually stabilize. But all the modules in this sequence are finite, since $H^2_{\Iw}(K_\infty, V)^{\Gamma}$ vanishes by assumption; this implies that there is a uniform power of $p$ (independent of $n$) which annihilates $H^2_{\Iw}(K_\infty, T)[\nu_n]$ for all $n \ge 1$. (Compare the proof of \cite[Proposition A.2.10]{LLZ14}, which is a similar argument with $\nu_n = (\gamma - 1)^n$ replaced by $\gamma^{p^n}-1$.) With this in hand we may proceed as in \cite{colmez98}.
  \end{proof}

  \begin{remark}
   We do not know if this result is valid for general $p$-adic Banach algebras (or even for affinoid algebras). It is also significant that the map is \emph{not} an isometry with respect to the natural norms on either side; there is a denominator arising from the torsion in $H^2_{\Iw}(K_\infty, T)$, which is difficult to control a priori (and, in particular, could potentially vary as we change the field $K$ in an Euler system argument). We are grateful to Ming-Lun Hsieh for pointing this out. We shall instead control denominators by means of the proposition that follows, in which the denominator depends on an $H^0$ rather than an $H^2$.
  \end{remark}

  \begin{proposition}
   \label{prop:ming-lun-lemma}
   Suppose that $V$ is a finite-dimensional $\Qp$-linear representation of $G$ such that $H^0(K_\infty, V) = 0$, and let $D'$ be a constant annihilating the finite group $H^0(K_\infty, V/T)$, for $T$ a $G$-invariant $\Zp$-lattice in $V$.

   Let $x_{n, j}$ be a collection of elements, and $C$ a constant, satisfying the hypotheses of Proposition \ref{prop:unbounded-iwasawa}; and let $\mu \in H^1(K, D_\lambda(\Gamma, V))$ be the resulting distribution. Then for every character $\kappa$ of $\Gamma$, we have
   \[ \left\| \int_{\Gamma} \kappa\, \mathrm{d}\mu\right\| \le C D D' \|\kappa\|_{\lambda}\]
   where on the left-hand side $\|\cdot\|$ denotes the norm on $H^1(K, V(\kappa^{-1}))$ for which the unit ball is the image of $H^1(K, T(\kappa^{-1}))$ (and $D$ is as in Proposition \ref{prop:boundfordists}).
  \end{proposition}

  \begin{proof}
   We know that $\|\mu\|_\lambda \le CD$ as elements of $H^1(K_\infty, D_\lambda(\Gamma, V))^\Gamma$. So $\|\int_\Gamma \kappa\, \mathrm{d}\mu\| \le CD \|\kappa\|_\lambda$ as elements of $H^1(K_\infty, V(\kappa^{-1}))^\Gamma$.

   By the definition of the supremum seminorm, this is equivalent to stating that the class $CD \|\kappa\|_\lambda \cdot \int_\Gamma \kappa\, \mathrm{d}\mu$ is the image of a class in $H^1(K_\infty, T(\kappa^{-1}))$. This class is not uniquely determined, and hence not necessarily $\Gamma$-invariant; but the constant $D'$ was chosen to annihilate the kernel of $H^1(K_\infty, T(\kappa^{-1})) \to H^1(K_\infty, V(\kappa^{-1}))$, so $CDD' \|\kappa\|_\lambda \cdot \int_\Gamma \kappa\, \mathrm{d}\mu$ lifts to a $\Gamma$-invariant class.

   Since $H^0(K_\infty, T) = 0$, we conclude that $H^1(K, T(\kappa^{-1})) \to H^1(K_\infty, T(\kappa^{-1}))^\Gamma$ is an isomorphism; thus $CDD' \|\kappa\|_\lambda \cdot \int_\Gamma \kappa\, \mathrm{d}\mu$ is in the image of the map $H^1(K, T(\kappa^{-1})) \to H^1(K, V(\kappa^{-1}))$ as required.
  \end{proof}

%%%%%%%%%%%%%%%%%%%%%%%%%%%%%%%%%%%%%
%%%%%%%%%%%%%%%%%%%%%%%%%%%%%%%%%%%%%

\section{Cyclotomic compatibility congruences}

%%%%%%%%%%%%%%%%%%%%%%%%%%%%%%%%%%%%%

 In this section, we establish that the Beilinson--Flach cohomology classes constructed in \cite{LLZ14} and \cite{KLZ1b} satisfy the criteria of the previous section, allowing us to interpolate them by finite-order distributions.

 \subsection{Modular curves: notation and conventions}
 \label{sect:Galoisrep}

  For $N \ge 4$, we write $Y_1(N)$ for the modular curve over $\ZZ[1/N]$ parametrising elliptic curves with a point of order $N$. Note that the cusp $\infty$ is not defined over $\QQ$ in this model, but rather over $\QQ(\mu_N)$.
  
  More generally, for $M, N$ integers with $M + N \ge 5$, we write $Y(M, N)$ for the modular curve over $\ZZ[1/MN]$ parametrising elliptic curves together with two sections $(e_1, e_2)$ which define an embedding of group schemes $\ZZ/M\ZZ \times \ZZ / N\ZZ \into E$ (so that $Y_1(N) = Y(1, N)$). We shall only consider $Y(M, N)$ in the case $M \mid N$, in which case the Weil pairing defines a canonical map from $Y(M, N)$ to the scheme $\mu_M^\circ$ of primitive $M$-th roots of unity, whose fibres are geometrically connected.

  If $A$ is an integer prime to $MN$, we shall sometimes also consider the curve $Y(M, N(A))$ over $\ZZ[1/AMN]$, parametrising elliptic curves with points $e_1, e_2$ as above together with a cyclic subgroup of order $A$.

  If $Y$ is one of the curves $Y(M, N)$ or $Y(M, N(A))$, we write $\sH_{\Zp}$ the relative Tate module of the universal elliptic curve over $Y$, which is an \'etale $\Zp$-sheaf on $Y[1/p]$. If the prime $p$ is clear from context we shall sometimes drop the subscript and write $\sH$ for $\sH_{\Zp}$. We write $\sH_{\Qp}$ for the associated $\Qp$-sheaf. We write $\TSym^k \sH_{\Zp}$ for the sheaf of degree $k$ symmetric tensors over $\sH_{\Zp}$; note that this is \emph{not} isomorphic to the $k$-th symmetric power, although these coincide after inverting $p$.
  
  \begin{remark}
   In this paper we will frequently consider \'etale cohomology of modular curves $Y(M, N(A))$, or products of pairs of such curves. All the coefficient sheaves we consider will be inverse systems of finite \'etale sheaves of $p$-power order, and we shall always work over bases on which $p$ is invertible. To lighten the notation,  the convention that if $p$ is \emph{not} invertible on $Y$, then $H^*_{\et}(Y, -)$ is a shorthand for $H^*_{\et}(Y[1/p], -)$.
  \end{remark}

 \subsection{Iwasawa sheaves}

  We recall some definitions and notation from \cite{KLZ1b}. Let $M, N \ge 1$ be integers with $M \mid N$ and $M + N \ge 5$. Then, associated to the \'etale sheaf of abelian groups $\sH_{\Zp}$ on $Y(M, N)[1/p]$, we have a sheaf of Iwasawa algebras $\Lambda(\sH_{\Zp})$ (c.f. Section 2.3 in {\em op.cit.}). For $c>1$ coprime to $6MNp$, let
  \[\cEI_{1,N}\in H^1_{\et}(Y(M, N), \Lambda(\sH_{\Zp})(1))\]
  be the Eisenstein--Iwasawa class, as defined in \cite[\S 4.3]{KLZ1b}. We now recall the definition of the Rankin--Iwasawa class on the product $Y(M, N)^2$, which is the image of $\cEI_{1,N}$ via a three-step procedure.
  
  Firstly, let us write $\Lambda(\sH_{\Zp})^{[j]}=\Lambda(\sH_{\Zp}) \otimes \TSym^j\left(\sH_{\Zp}\right)$ for $j \ge 0$. Then we have a morphism of \'etale sheaves on  $Y(M, N)[1/p]$, the \emph{Clebsch--Gordan map},
  \[ 
   CG^{[j]}: \Lambda(\sH_{\Zp})\rTo 
   \left(\Lambda(\sH_{\Zp})^{[j]}\hat\otimes \Lambda(\sH_{\Zp})^{[j]}\right)(-j)
  \]
   as defined in \cite[Definition 5.1.1]{KLZ1b}.
  
  Secondly, let $Y(M, N)^2$ denote the fibre product $Y(M, N) \times_{\mu_M^\circ} Y(M, N)$, where $\mu_M^\circ$ is the group of primitive $M$-th roots of unity as above. We denote by $\Lambda^{[j, j]}$ the exterior tensor product $\Lambda(\sH_{\Zp})^{[j]}\boxtimes \Lambda(\sH_{\Zp})^{[j]}$ on $Y(M, N)^2$. Pushforward along the diagonal embedding $\Delta: Y(M, N) \into Y(M, N)^2$ gives a map
  \[ 
   \Delta_*:
   H^1_{\et}\left(Y(M,N), \Lambda(\sH_{\Zp})^{[j]}\hat\otimes \Lambda(\sH_{\Zp})^{[j]}(1-j)\right)
   \rTo H^3_{\et}\left(Y(M,N)^2,\Lambda^{[j,j]}(2-j)\right).
  \]
  
  Thirdly, for $a\in \ZZ / M\ZZ$, denote by $u_a$ the automorphism of $Y(M,N)^2$ which is the identity on the first $Y(M, N)$ factor and is given by $(E,e_1,e_2)\mapsto \left(E,e_1+a\frac{N}{M}e_2,e_2\right)$ on the second factor. 
  
  \begin{definition}
   For integers $M, N \ge 1$ with $M \mid N$ and $M + N \ge 5$, $j \ge 0$, $a \in \ZZ / m \ZZ$, $p$ a prime $> 2$, and $c > 1$ coprime to $6 M N p$, define the \emph{Rankin--Iwasawa} class
   \[ 
    \cRI^{[j]}_{M, N,a} = \left( (u_a)_*\circ \Delta_*\circ CG^{[j]}\right)({}_c\EI_{1,N})\in H^3_{\et}\left(Y(M, N)^2,\Lambda^{[j,j]}(2-j)\right).
   \]
  \end{definition}
  
  The primary purpose of introducing the Rankin--Iwasawa class is that it is easy to prove norm-compatibility relations for it. Our actual interest is in a second, related class, defined by pushing forward $\cRI^{[j]}_{M, N, a}$ via a degeneracy map.
  
  \begin{definition}
   \label{def:BFelt}
   For integers $m \ge 1$ and $N \ge 4$, $j \ge 0$, $a \in \ZZ / m\ZZ$, and $c > 1$ coprime to $6mNp$, define the \emph{Beilinson--Flach} class
   \[ 
   \cBF^{[j]}_{m, N, a} \in H^3\left(Y_1(N)^2 \times \mu_m^\circ, \Lambda^{[j,j]}(2-j)\right) 
   \]
   to be the image of ${}_c\RI^{[j]}_{m,mN,a}$ under the map $(t_m\times t_m)_*$, where 
   \[ t_m:Y(m,mN) \rTo Y_1(N)\times \mu_m^\circ \]
   is the map given in terms of moduli spaces as
   \[ (E,e_1,e_2)\mapsto \left( \left(E/\langle e_1\rangle,e_2 \bmod \langle e_1\rangle\right), \langle e_1,Ne_2\rangle_{E[m]}\right).\]
  \end{definition}
  
  \begin{remark}
   Note that $t_m$ corresponds to $z \mapsto z/m$ on the upper half-plane.
  \end{remark}

  Finally, recall that there are natural maps
  \[ 1 \otimes \mom^j: \Lambda(\sH_{\Zp}) \to \Lambda(\sH_{\Zp})^{[j]}\]
  which, for a geometric generic point $\eta$, are given by the maps $\Lambda(\sH_{\eta}) \to \Lambda(\sH_{\eta}) \otimes \TSym^j \sH_{\eta}$, $[x] \mapsto [x] \otimes x^{\otimes j}$.

%%%%%%%%%%%%%%%%%%%%%%%%%%%%%%%%%%%%%%%%

 \subsection{Compatibility congruences}

  We now come to the key technical result required for the rest of this paper. Let $h \ge 1$. For each $r \ge 1$, we would like to prove a congruence modulo $p^{hr}$ relating the classes
  \[ \Res_{p^r}^{p^{hr}} \left(\cBF^{[j]}_{p^r, N, a}\right) \]
  for $0 \le j \le h$. Here $\Res_{p^r}^{p^{hr}}$ denotes the pullback along the natural map 
  \[ Y_1(N) \times \mu_{p^{hr}}^\circ \to Y_1(N) \times \mu_{p^{r}}^\circ, \]
  which corresponds classically to restriction of cocycles in Galois cohomology.

  \begin{definition}
   For an arbitrary $m$, let $Z(m,mN) \subseteq Y(m, m N)^2$ denote the preimage of the diagonal subvariety of $Y_1(N)$ under the natural projection map $Y(m, m N)^2 \to Y_1(N)^2$ (i.e.~the map corresponding to the identity on the upper half-plane, \emph{not} the map $t_m$). 
  \end{definition} 
  
  \begin{note} 
   The subvariety $Z(m,mN)$ is preserved by the action of $\Gamma_1(N) \times \Gamma_1(N)$, and in particular by the action of the element $u_a =\left(1, \stbt 1 a 0 1 \right)$ for any $a \in \ZZ / m \ZZ$. Since $u_a$ is an automorphism, and its inverse is $u_{-a}$, we have $(u_a)_* = (u_{-a})^*$.
  \end{note} 

  There is a canonical section of the sheaf $(\sH_{\Zp} \boxtimes \sH_{\Zp})(-1)$ over the subvariety $Z(m,mN)$, given by the Weil pairing (since along $Z(m,mN)$ the two universal elliptic curves coincide). We call this element $\mathcal{CG}$ (for ``Clebsch--Gordan''), since the Clebsch--Gordan map $CG^{[j]}$ is given by cup-product with the $j$-th divided power $\mathcal{CG}^{[j]}$ of this element. For $t \ge 1$, we write $\mathcal{CG}_t$ for the image of $\mathcal{CG}$ modulo $p^t$. Note that we have
  \[ u_a^*\left( \mathcal{CG} \right) = \mathcal{CG} \]
  for any $a \in \ZZ / m\ZZ$, since $\mathcal{CG}$ is independent of the level structure.

  Let $i$ be the inclusion of $Z(m,mN)$ into $Y(m, m N)^2$, so the diagonal embedding factors as
  \[ Y(m,mN)\rTo^\Delta Z(m,mN)\rTo^i Y(m,mN)^2.\]
  By construction, the element $\cRI^{[j]}_{m, m N, a}\in H^3_{\et}(Y(m,mN)^2,\Lambda^{[j,j]}(2-j))$ is given by 
  \begin{align}
   \cRI^{[j]}_{m, m N, a} &=i_* \circ u_{-a}^* \circ\Delta_*\circ CG^{[j]}\left(\cEI_{ 1,mN}\right)\notag \\
   &=   i_* \circ u_{-a}^* \left( \Delta_*(\cEI_{1,mN}) \cup \mathcal{CG}^{[j]} \right)\notag \\
   &= i_* \left(  (u_{a} \circ \Delta)_*(\cEI_{1,mN}) \cup \mathcal{CG}^{[j]} \right). \label{cRIdef}
  \end{align}

  We now take integers $r \ge 1$ and $h \ge 1$ as above, and we assume $p \nmid m$. We also assume that the following condition is satisfied:
  
  \begin{assumption}
   \label{assumptionN}
   We have $p^{(h-1)r} \mid N$, so there is a canonical section $Y_{hr}$ of $\sH_{hr}$ over $Y(mp^r, mp^r N)$.
  \end{assumption} 

  Under this assumption, the moment map modulo $p^{hr}$ is given by cup-product with the element $Y_{hr}$, so we obtain the following somewhat messy formula:
  
  \begin{proposition}
   For any $a \in \ZZ / mp^{hr} \ZZ$, we have the following equality modulo $p^{hr}$: 
   \begin{multline*}
    \sum_{j = 0}^h a^{h-j} (h-j)! (1 \otimes \mom^{h-j})^{\boxtimes 2} \Res_{mp^r}^{mp^{hr}} \left( \cRI_{mp^r, mp^rN, a}^{[j]} \right) \otimes \zeta_{p^{hr}}^{\otimes j} = 
    \\ i_*\left( (u_{a} \circ \Delta)_* \left(\cEI_{1, m p^r N}\right) \cup \left( a \cdot Y_{hr} \boxtimes Y_{hr} + \mathcal{CG}_{hr} \otimes \zeta_{p^{hr}} \right)^{[h]}\right).
   \end{multline*}
  \end{proposition}
  
  \begin{proof}
   This is a straightforward exercise from the definition of multiplication in the algebra $\TSym^\bullet$. (The factor of $(h-j)!$ appears because $(Y \boxtimes Y)^{[h-j]} = (h - j)! Y^{[h-j]} \boxtimes Y^{[h-j]}$).
  \end{proof}
    
  We can now prove the main theorem of this section:

  \begin{theorem}
   \label{thm:congruencesfinal}
   Suppose that $p \mid N$. Then for any $a \in \ZZ / mp^{hr} \ZZ$ and any $m$ coprime to $p$, we have
   \begin{multline*}
    \sum_{j = 0}^h a^{h-j} (h-j)! \Res_{mp^r}^{mp^{\infty}} (1 \otimes \mom^{h-j})^{\boxtimes 2}\left(\cBF^{[j]}_{mp^{r}, N, a}\right) \otimes \zeta_{p^{hr}}^{\otimes j}
    \\ \in p^{hr} H_{\et}^3\left(Y_1(N)^2 \times\mu_{mp^\infty}^\circ, \Lambda^{[h, h]}(\sH_{\Zp})(2)\right).
   \end{multline*}
  \end{theorem}
  
  \begin{proof}
   It follows from \cite[Theorem 5.3.1]{KLZ1b} that if $N'$ is any multiple of $N$ with the same prime divisors as $N$, then $\cBF^{[j]}_{p^{r}, N, a}$ is the image of $\cBF^{[j]}_{p^{r}, N', a}$ under pushforward along the natural degeneracy map $Y_1(N')\rightarrow Y_1(N)$. We can therefore assume without loss of generality that $N$ satisfies Assumption \ref{assumptionN}. 
   
   We may factor the map $(t_{m p^r} \times t_{m p^r})_*$ as the composite of a map on the coefficient sheaves, which is a morphism 
   \[ (t_{mp^r} \times t_{mp^r})_\sharp: \sH \boxtimes \sH \to t_{m p^r}^*(\sH) \boxtimes t_{m p^r}^*(\sH) \] 
   of sheaves on $Y(mp^r, mp^r N)^2$, followed by the pushforward via $t_{mp^r} \times t_{mp^r}$ on the underlying modular curve.
   
   We claim that when restricted to the image of $u_a \circ \Delta: Y(mp^r, mp^r N) \to Z(mp^r, mp^rN)$, the section $a \cdot Y_{r} \boxtimes Y_{r} + \mathcal{CG}_{r} \otimes \zeta_{p^{r}}$ of $\sH_r \boxtimes \sH_r$ is in the kernel of $(t_{mp^r} \times t_{mp^r})_\sharp$. 
   
   This follows from the fact that the map $(t_{mp^r} \times t_{mp^r})_\sharp$ is given by quotienting out by the first component of the level structure in each factor: on the fibre at a point $(E_1, P_1, Q_1) \times (E_2, P_2, Q_2)$ of $Y(mp^r, mp^r N)^2$, the fibre of $\sH \boxtimes \sH$ is the Tate module of $E_1 \times E_2$, and the map $(t_{mp^r} \times t_{mp^r})_\sharp$ is the quotient map $E_1 \times E_2 \to E_1 / \langle P_1 \rangle \times E_2 / \langle P_2 \rangle$. A point in the image of $u_a \circ \Delta$ is given by $(E, P, Q) \times (E, P + aNQ, Q)$ for some point $(E, P, Q)$ of $Y(mp^r, mp^r N)$, and the section $\mathcal{CG}_r \otimes \zeta_{p^r}$ is given by $N Q \boxtimes P - P \boxtimes NQ$. Thus we have
   \[ a \cdot Y_{r} \boxtimes Y_{r} + \mathcal{CG}_{r} \otimes \zeta_{p^{r}} = a NQ \boxtimes NQ + (NQ \boxtimes P - P \boxtimes NQ) = NQ \boxtimes (P + a N Q) - P \boxtimes NQ, \]
   which is annihilated by $(t_{mp^r} \times t_{mp^r})_\sharp$ as claimed.
   
   Since this element is annihilated by $(t_{mp^r} \times t_{mp^r})_\sharp$ modulo $p^r$, its $h$-th tensor power is annihilated by the same map modulo $p^{hr}$. This gives the congruence stated above. 
  \end{proof}
  
  \begin{remark}
   \label{remark:weasel}
   We shall in fact use a slight refinement of this theorem. Let $\cE$ be the universal elliptic curve over $Y_1(N)$, and let $D' = C - \{0\} \subset \cE[p]$, where $C$ is the universal level $p$ subgroup. Then there is a subsheaf $\sH_{\Zp} \langle D' \rangle$ of $\sH_{\Zp}$, which is the preimage of $D'$ under reduction modulo $p$, and a corresponding sheaf of Iwasawa modules $\Lambda(\sH_{\Zp}\langle D'\rangle)$.
   
   The Beilinson--Flach elements for $p \mid N$ are, by construction, the images of elements of the group
   \[ H^3_{\et}\left(Y_1(N) \times \mu_m^\circ, (\Lambda(\sH_{\Zp} \langle D' \rangle) \otimes \TSym^j \sH_{\Zp})^{\boxtimes 2}(2) \right); \]
   and exactly the same argument as above shows that we have a congruence modulo $p^{hr}$ in this group. We will need this below, in order to interpolate our elements in Coleman families.
  \end{remark}
  
%%%%%%%%%%%%%%%%%%%%%%%%%%%%

\subsection{Galois representations: notation and conventions}
 \label{sect:galrep}

  In this section, we shall fix notations for Galois representations attached to modular forms. Let $f$ be a normalised cuspidal Hecke eigenform of some weight $k+2 \ge 2$ and level $N_f \ge 4$, and let $L$ be a number field containing the $q$-expansion coefficients of $f$.

  \begin{definition}
   For each prime $\frP \mid p$ of $L$, we write $M_{L_{\frP}}(f)$ for the maximal subspace of
   \[ H^1_{\et, c}\left(Y_1(N_f)_{\overline{\QQ}}, \Sym^k \sH_{\Qp}^\vee\right) \otimes_\Qp L_{\frP} \]
   on which the Hecke operator $T_\ell$, for every prime $\ell$, acts as multiplication by $a_\ell(f)$. Dually, we write $M_{L_{\frP}}(f)^*$ for the maximal \emph{quotient} of the space  
   \[ 
    H^1_{\et}\left(Y_1(N_f)_{\overline{\QQ}}, \TSym^k(\sH_{\Qp})(1) \right) \otimes_\Qp L_{\frP} 
   \]
   on which the dual Hecke operators $T_\ell'$ act as $a_\ell(f)$.
  \end{definition}

  Both spaces $M_{L_{\frP}}(f)$ and $M_{L_{\frP}}(f)^*$ are 2-dimensional $L_{\frP}$-vector spaces with continuous actions of $\Gal(\overline{\QQ} / \QQ)$, unramified outside $S$, where $S$ is the finite set of primes dividing $p N_f$. The twist by 1 implies that the Poincar\'e duality pairing
  \[ M_{L_{\frP}}(f) \times M_{L_{\frP}}(f)^* \to L_{\frP} \]
  is well-defined (and perfect), justifying the notation. If $f$ is new and $f^*$ is the eigenform conjugate to $f$, then the natural map $M_{L_{\frP}}(f^*)(1) \to M_{L_{\frP}}(f)^*$ is an isomorphism of $L_{\frP}$-vector spaces, although we shall rarely use this.

  If $f$, $g$ are two eigenforms (of some levels $N_f, N_g$ and weights $k+2, k' + 2 \ge 2$) with coefficients in $L$, we write $M_{L_{\frP}}(f \otimes g)$ for the tensor product $M_{L_{\frP}}(f) \otimes_{L_{\frP}} M_{L_\frP}(g)$, and similarly for the dual $M_{L_{\frP}}(f \otimes g)^*$. Via the K\"unneth formula, we may regard $M_{L_{\frP}}(f \otimes g)^*$ as a quotient of $H^2_{\et}(Y_1(N)^2_{\overline{\QQ}},\TSym^{[k,k']}(\sH_{\Qp})(2))\otimes_{\Qp} L_{\frP}$, for any $N \ge 4$ divisible by $N_f$ and $N_g$, where $\TSym^{[k,k']}(\sH_{\Qp})$ denotes the \'etale $\Qp$-sheaf $\TSym^k \sH_{\Qp} \boxtimes \TSym^{k'} \sH_{\Qp}$.

 \subsection{Consequences for pairs of newforms}

  We now use the congruences of Theorem \ref{thm:congruencesfinal}, together with the $p$-adic analytic machinery of Section \ref{sect:analyticprelim}, in order to define ``unbounded Iwasawa cohomology classes'' interpolating the Beilinson--Flach elements for a given pair $(f, g)$ of eigenforms.

  \begin{remark}
   We shall prove a considerably stronger result below (incorporating variation in Coleman families) which will mostly supersede Theorem \ref{thm:cycloBFelts}: see Theorem \ref{thm:3varelt}. However, the proof of the stronger result is much more involved, so for the reader's convenience we have given this more direct argument.
  \end{remark}

  Let us choose two normalised cuspidal eigenforms $f$, $g$, of weights $k + 2, k' + 2$ and levels $N_f, N_g$ respectively, with $k, k' \ge 0$. Let $L$ be a number field containing the coefficients of $f$ and $g$, and $\frP$ a prime of $L$ above $p$, so that the Galois representation $M_{L_\frP}(f \otimes g)^*$ of \S \ref{sect:galrep} is defined.
  Assume that $0\leq j\leq \min\{k,k'\}$, and let $N$ be an integer divisible by $N_f$ and $N_g$ and having the same prime factors as $N_f N_g$. Let $m\geq 1$. Recall from \cite[Definition 3.3.1]{KLZ1b} that we have an \'etale Eisenstein class
  \[ 
   \Eis^{[k,k',j]}_{\et,1,mN}\in H^3_{\et}\left(Y_1(mN)^2,\TSym^{[k,k']}\sH_{\Qp}(2-j)\right),
  \]
  which can be constructed using Beilinson's Eisenstein symbol (and in particular is the image of a class in motivic cohomology). By abuse of notation, we also denote by $\Eis^{[k,k',j]}_{\et,1,mN}$ the pull-back of this class to $Y(m,mN)^2$. 
  
  \begin{definition}
   For $a \in \ZZ/m\ZZ$, define $\BF^{[f, g, j]}_{m,a}$ to be the image of $(u_a)_*\Eis^{[k,k',j]}_{\et,1,mN}$ under the following composition of maps:
   \begin{align*}
    H^3_{\et}\left(Y(m, mN)^2,\TSym^{[k,k']}\sH_{\Qp}(2-j)\right) \rTo^{(t_m \times t_m)_*}&  H^3_{\et}\left(Y_1(N)^2\times\mu_m^\circ, \TSym^{[k,k']}\sH_{\Qp}(2-j)\right)\\
     \rTo& H^1\left(\QQ(\mu_m),H^2_{\et}(Y_1(N)^2_{\overline{\QQ}},\TSym^{[k,k']}\sH_{\Qp}(2-j)\right)\\
    \rTo& H^1\left(\QQ(\mu_m),M_{L_{\frP}}(f \otimes g)^*(-j)\right).
   \end{align*}
   This is independent of the choice of $N$. For $c > 1$ coprime to $6mpN_f N_g$, we define 
   \[ 
    \cBF^{[f, g, j]}_{m,a} \coloneqq
    \Big(c^2-c^{-(k+k'-2j)} \varepsilon_f(c)^{-1}\varepsilon_g(c)^{-1}\sigma_c^2\Big)\,
    \BF^{[f, g, j]}_{m,a}.
   \]
  \end{definition}
  
  \begin{remark}
   Note that for $m = 1$ the class $\BF^{[f, g, j]}_{m,a}$ is the Eisenstein class $\operatorname{AJ}_{f,g,\et}\left(\Eis^{[k, k', j]}_{\et,1,N}\right)$ of \cite[\S 5.4]{KLZ1a}.
  \end{remark}
  
  Let us recall the connection between these classes and the Iwasawa-theoretic classes of the previous sections. Recall that we have maps
  \[ \mom^{k-j} \cdot 1: \Lambda(\sH) \otimes \TSym^j(\sH) \to \TSym^k(\sH) \]
  for each $k \ge j$.
  
  \begin{proposition}[{\cite[Proposition 5.2.3 (3)]{KLZ1b}}]
   The class $\cBF^{[f, g, j]}_{m,a}$ coincides with the image of
   \[ \left[ (\mom^{k-j} \cdot 1) \boxtimes (\mom^{k'-j} \cdot 1)\right] \left( \cBF^{[j]}_{m, N, a}\right) \]
   under projection to the $(f, g)$-eigenspace.
  \end{proposition}
  
  We now consider ``$p$-stabilised'' versions of these objects. If $p \nmid N_f$, we choose a root $\alpha_f \in L$ of the Hecke polynomial of $f$ (after extending $L$ if necessary); and we let $f_\alpha$ be the corresponding $p$-stabilisation of $f$, so $f_\alpha$ is a normalised eigenform of level $N_{f_\alpha}=p N_f$, with $U_p$-eigenvalue $\alpha_f$ and the same $T_\ell$-eigenvalues as $f$ for all $\ell \ne p$. If $p \mid N_f$, then we assume that $a_p(f) \ne 0$, and we set $\alpha_f = a_p(f)$ and (for consistency) $f_\alpha = f$ and $N_{f_\alpha}=N_f$. We define $\alpha_g$ and $g_\alpha$ similarly.
     
  If $p \nmid N_f N_g$, then the class $\cBF^{[f_\alpha,g_\alpha,j]}_{m, a}$ for $m$ coprime to $p$ is related to the Eisenstein class for the forms $f, g$ as follows. There is a correspondence
  $\Pr^{\alpha_f}: Y_1(p N_f) \to Y_1(N_f)$ given by $\pr_1 - \frac{\beta}{p^{k + 1}} \pr_2$, and $(\Pr^{\alpha_f})_*$ gives an isomorphism
  \[ M_{L_{\frP}}(f_\alpha)^* \to M_{L_{\frP}}(f)^*, \]
  and similarly for $g$.
  
  \begin{proposition}
   For $p \nmid m N_f N_g$, we have
   \[ (\Pr^{\alpha_f} \times \Pr^{\alpha_g})_* \left(\BF^{[f_\alpha,g_\alpha,j]}_{m, a}\right) = \left(1 - \frac{\alpha_f \beta_g}{p^{1 + j} \sigma_p}\right)  \left(1 - \frac{\beta_f \alpha_g}{p^{1 + j}\sigma_p}\right) \left(1 - \frac{\beta_f \beta_g}{p^{1 + j}\sigma_p}\right)\cdot\cBF^{[f, g, j]}_{m, a}.\]
   If $p \mid N_f$ but $p \nmid m N_g$, then we have
   \[ (\mathrm{id} \times \Pr^{\alpha_g})_* \left(\BF^{[f, g_\alpha, j]}_{m, a}\right) = \left(1 - \frac{\alpha_f \beta_g}{p^{1 + j}\sigma_p}\right) \cdot\cBF^{[f, g, j]}_{m, a}.\]
  \end{proposition}
  
  \begin{proof}
   This is a restatement of Lemma 5.6.4 and Remark 5.6.5 of \cite{KLZ1b}.
  \end{proof}
  
  We shall now interpolate the $\cBF^{[f_\alpha, g_\alpha, j]}_{m, a}$ for varying $m$ and $j$, under the following assumption:
   
  \begin{assumption}\label{ass:nottwists}
   The automorphic representations $\pi_f$ and $\pi_g$ corresponding to $f$ and $g$ are not twists of each other.
  \end{assumption} 
   
  \begin{note}
   Assumption \ref{ass:nottwists} is automatically  satisfied if $k \ne k'$.
  \end{note}
   
   Let $m$ be coprime to $p$ and $r\geq 1$. Then Assumption \ref{ass:nottwists} implies that $H^0(\QQ(\mu_{mp^\infty}),M_{L_{\frP}}(f\otimes g))=0$, so the restriction map induces an isomorphism
   \[  H^1\left(\QQ(\mu_{mp^r}),M_{L_{\frP}}(f\otimes g)^*(-j)\right)\cong H^1\left(\QQ(\mu_{mp^\infty}),M_{L_{\frP}}(f\otimes g)^*\right)^{\Gamma_r=\chi^j}.\]   
   
   \begin{convention}
    By abuse of notation, we write ${}_c\BF^{[f_\alpha,g_\alpha,j]}_{mp^r,a}$ for the image of the Beilinson-Flach element in $H^1(\QQ({\mu_{mp^\infty}}),M_{L_{\frP}}(f_\alpha\otimes g_\alpha)^*)^{\Gamma_r=\chi^{j}}$. 
   \end{convention}
   
   These elements satisfy the following compatibility:
      
  \begin{lemma}\label{lem:normcompatible}
   Let $m\geq 1$ be coprime to $p$, and let $r\geq 0$. Then 
   \[ \sum_{\Gamma_r/\Gamma_{r+1}}\chi(\gamma)^{-j}\gamma\cdot\, {}_c\BF^{[f_\alpha,g_\alpha,j]}_{mp^{r+1},a}=
   \begin{cases}
    (\alpha_f\alpha_g)\, \cBF^{[f_\alpha,g_\alpha,j]}_{mp^{r},a} & \text{if $r>0$}\\
    (\alpha_f\alpha_g-p^j\sigma_p)\, {}_c\BF^{[f_\alpha,g_\alpha,j]}_{mp^{r},a} & \text{if $r=0$}
   \end{cases}\]
  \end{lemma}
  
  \begin{proof}
   This follows from the second norm relation for the Rankin-Iwasawa classes (c.f. \cite[Theorem 5.4.4]{KLZ1b}).
  \end{proof}

  We impose the following ``small slope'' assumption:
  \begin{equation}
   \label{eq:smallslope}
   v_p(\alpha_f \alpha_g) < 1 + \min(k, k').
  \end{equation}

  \begin{theorem}
   \label{thm:cycloBFelts}
   If the small slope assumption \eqref{eq:smallslope} holds, then for any integers $m \ge 1$ coprime to $p$ and $a \in (\ZZ / mp^\infty \ZZ)^\times$, there exists a unique element
   \[
    \cBF^{[f_\alpha, g_\alpha]}_{m, a} \in D_\lambda(\Gamma, \Qp) \otimes_{D_0(\Gamma, \Qp)} H^1_{\Iw}(\QQ(\mu_{m p^\infty}), M_{L_{\frP}}(f_\alpha \otimes g_\alpha)^*),
   \]
   where $\lambda = v_p(\alpha_f \alpha_g)$, such that for every $r \ge 0$ and $0 \le j \le \min(k, k')$, the image of $\cBF^{[f_\alpha, g_\alpha]}_{m, a}$ in $H^1(\QQ(\mu_{m p^r}), M_{L_{\frP}}(f_\alpha\otimes g_\alpha)^*(-j))$ is given by
   \[
    \left.\begin{cases}
     (\alpha_f \alpha_g)^{-r} & \text{if $r > 0$} \\
     1 - \frac{p^{j} \sigma_p}{\alpha_f \alpha_g} & \text{if $r = 0$}
    \end{cases}\right\} \times
    \frac{\cBF^{[f_\alpha, g_\alpha, j]}_{mp^r, a}}{(-a)^j j! \binom{k}{j} \binom{k'}{j}}.
   \]
  \end{theorem}

  \begin{remark}
   Compare Theorem 6.8.4 of \cite{LLZ14}, which is the case $k = k' = 0$. 
  \end{remark}

  \begin{proof}
   This amounts to reorganizing the output of Theorem \ref{thm:congruencesfinal} and Proposition \ref{prop:unbounded-iwasawa}. Let $h = \min(k, k')$. Consider the composition of maps 
   \begin{align*}
    H_{\et}^3\left(Y_1(N)^2 \times \mu_{mp^\infty}, \Lambda^{[h, h]}(\sH_{\Zp})(2-h)\right) & \rTo^{\, \otimes e_h\,} H_{\et}^3\left(Y_1(N)^2 \times \mu_{mp^\infty}, \Lambda^{[h, h]}(\sH_{\Zp})(2)\right)\\
     & \rTo H^1\left(\QQ(\mu_{mp^\infty}), H^2_{\et}(Y_1(N)^2_{\overline{\QQ}},\Lambda^{[h, h]}(\sH_{\Zp})(2))\right)\\
    & \rTo H^1\left(\QQ(\mu_{mp^\infty}), H^2_{\et}(Y_1(N)^2_{\overline{\QQ}}, \TSym^{[k,k']}(\sH_{\Zp})(2))\right)\\
    & \rTo H^1\left(\QQ(\mu_{mp^\infty}),M_{L_{\frP}}(f_\alpha\otimes g_\alpha)^*\right)
   \end{align*}
   where $e_h$ is the canonical basis of $\Zp(h)$ over $\QQ(\mu_{p^\infty})$, and the third map is given by $(\mom^{k-h} \cdot \id) \boxtimes (\mom^{k'-h} \cdot \id)$.    
   An unpleasant manipulation of factorials shows that the image  of the expression in Theorem \ref{thm:congruencesfinal} under this composition of maps is equal to
   \begin{equation}
    \label{eq:messyeq} 
    \frac{k! (k')!}{(k-h)!(k'-h)! h!} \sum_{j = 0}^{h} (-1)^j \binom{h}{j} y_{r, j}, 
   \end{equation} 
   where we write $y_{r, j}$ for the quantity 
   \[ 
    \left[(-a)^j j! \binom{k}{j} \binom{k'}{j}\right]^{-1} \cBF^{[f, g, j]}_{mp^r, a} \in H^1\left(\QQ(\mu_{mp^{\infty}}),M_{L_{\frP}}(f\otimes g)^*\right)^{\Gamma_r=\chi^j}.
   \]     
   The image of $H^2_{\et}(Y_1(N)^2_{\overline{\QQ}}, \TSym^{[k,k']}(\sH_{\Zp})(2)) \otimes \cO_{\frP}$ in $M_{L_{\frP}}(f \otimes g)^*$ is a $\cO_{\frP}$-lattice, and hence it defines a norm $\|\cdot\|$ on $M_{L_{\frP}}(f \otimes g)^*$. So Theorem \ref{thm:congruencesfinal} gives the norm bound
   \[ \left\| \sum_{j = 0}^{h} (-1)^j \binom{h}{j} y_{r, j}\right\| = O(p^{-hr}),\]
   where the implied constant in the $O()$ term depends on $k,k',h$ but not on $r$. Combining this fact with Lemma \ref{lem:normcompatible}, we deduce that the quantities 
   \[ x_{r, j} = (\alpha_f \alpha_g)^{-r} y_{r, j}\in H^1\left(\QQ(\mu_{mp^\infty}), M_{L_{\frP}}(f_\alpha\otimes g_\alpha)^*\right)^{\Gamma_r=\chi^j}\] 
   satisfy the hypotheses of Proposition \ref{prop:unbounded-iwasawa}, so there exists an element 
   \[ {}_c\BF^{[f_\alpha, g_\alpha]}_{m, a}\in H^1\left(\QQ(\mu_{mp^\infty}),D_\lambda(\Gamma,  M_{L_{\frP}}(f_\alpha \otimes g_\alpha)^*)\right)^\Gamma\]
   interpolating the $x_{r,j}$. Using again that $H^0(\QQ(\mu_{mp^\infty}), M_{L_{\frP}}(f_\alpha \otimes g_\alpha)^*)=0$ by Assumption \ref{ass:nottwists}, this element lifts uniquely to 
   \[ D_\lambda(\Gamma, \Qp) \otimes_{D_0(\Gamma, \Qp)} H^1_{\Iw}(\QQ(\mu_{m p^\infty}), M_{L_{\frP}}(f_\alpha \otimes g_\alpha)^*)\] 
   and has the required interpolation properties, which finishes the proof. 
  \end{proof}

  We now note, for future use, the following vital property of the classes $\cBF^{[f_\alpha, g_\alpha]}_{m, a}$. Denote by 
  \begin{multline*} \cL_{M_{L_\frP}(f_\alpha \otimes g_\alpha)^*}:  D^{\la}(\Gamma, \Qp) \otimes_{D_0(\Gamma, \Qp)} H^1_{\Iw}(\QQ_{p, \infty}, M_{L_{\frP}}(f_\alpha \otimes g_\alpha)^*) \\ \rTo D^{\la}(\Gamma, \Qp) \otimes_\Qp \DD_{\cris}\left(M_{L_\frP}(f_\alpha \otimes g_\alpha)^*\right)\end{multline*}
  Perrin-Riou's regulator map (c.f. \cite{PerrinRiou-fonctionsL} and \cite[Appendix B]{LZ}). 

  \begin{proposition}
   \label{prop:vanishing}
   If the stronger inequality
   \[ v_p(\alpha_f \alpha_g) < \frac{1 + \min(k, k')}{2} \]
   holds, then the projection of $\cL_{M_{L_{\frP}}(f_\alpha \otimes g_\alpha)^*}\left(\cBF^{[f_\alpha, g_\alpha]}_{m, a}\right)$ to the $\varphi = (\alpha_f \alpha_g)^{-1}$-eigenspace of $\QQ(\mu_m) \otimes_{\QQ} \DD_{\mathrm{cris}}(M_{L_{\frP}}(f_\alpha \otimes g_\alpha)^*)$ is zero.
  \end{proposition}

  \begin{proof}
   Let $W$ be this eigenspace. It is well known that the projection of $\cL_{M(f \otimes g)^*}$ to $W$ gives a map
   \[ H^1_{\Iw}(\QQ_{p, \infty}, M(_{L_{\frP}}(f_\alpha \otimes g_\alpha)^*) \to D_\lambda(\Gamma, \Qp) \otimes W, \]
   where $\lambda =  v_p(\alpha_f \alpha_g)$ as before. So it gives a map
   \[
    D_{\lambda}(\Gamma, \Qp) \otimes_{D_0(\Gamma, \Qp)} H^1_{\Iw}(\QQ_{p, \infty}, M_{L_{\frP}}(f_\alpha \otimes g_\alpha)^*) \to D_{2\lambda}(\Gamma, E) \otimes W.
   \]

   However, for any character of $\Gamma$ of the form $z \mapsto z^j \chi(z)$, with $0 \le j \le \min(k, k')$ and $\chi$ of finite order, the image of $\cBF^{[f_\alpha, g_\alpha]}_{m, a}$ in $H^1(\QQ(\mu_m) \otimes \Qp, M_{L_{\frP}}(f_\alpha \otimes g_\alpha)^*(-j-\chi))$ lies in the Bloch--Kato $H^1_\mathrm{g}$ subspace, by construction (c.f. \cite[Proposition 3.3.2]{KLZ1b}). If $\chi$ is non-trivial (so that the interpolation factors relating $\cL_{M_{L_{\frP}}(f_\alpha \otimes g_\alpha)^*}$ to the dual exponential map are invertible, see \cite[Theorem B.5]{LZ}), then this implies that $\cL_{M_{L_{\frP}}(f_\alpha \otimes g_\alpha)^*}(\cBF^{[f_\alpha, g_\alpha]}_{m, a})(j + \chi) = 0$.

   So the projection of $\cL_{M_{L_{\frP}}(f_\alpha \otimes g_\alpha)^*}\left(\cBF^{[f_\alpha, g_\alpha]}_{m, a}\right)$ to $W$ is an element of $D_{2\lambda}(\Gamma, \Qp) \otimes W$ which vanishes at all but finitely many characters of the form $j + \chi$ with $j \in \{0, \dots, \min(k, k')\}$ and $\chi$ of finite order. Since $2\lambda < 1 + \min(k, k')$, this projection must be zero as required.
  \end{proof}

  \begin{remark}
   We shall in fact show below that the result of Proposition \ref{prop:vanishing} is actually true whenever $\alpha_f \alpha_g$ satisfies the weaker assumption \eqref{eq:smallslope} (i.e.~whenever the class $\cBF^{f_\alpha, g_\alpha}_{m, a}$ is defined), by deforming Proposition \ref{prop:vanishing} along a Coleman family.

   This vanishing property is natural in the context of Conjecture 8.2.6 of \cite{LLZ14}, which predicts the existence of an element in $\bigwedge^2 H^1_{\Iw}(\QQ(\mu_{mp^\infty}), M_{L_{\frP}}(f \otimes g)^*)$ from which the Beilinson--Flach elements (for all choices of $\alpha_f$ and $\alpha_g$) can be obtained by pairing with the map $\cL_{M_{L_{\frP}}(f \otimes g)^*}$ and projecting to a $\varphi$-eigenspace. Clearly, pairing an element of $\bigwedge^2$ with the same linear functional twice will give zero.
  \end{remark}

%%%%%%%%%%%%%%%%%%%%%%%%%%%%%%
%%%%%%%%%%%%%%%%%%%%%%%%%%%%%%

\section{Overconvergent \'etale cohomology and Coleman families}

%%%%%%%%%%%%%%%%%%%%%%%%%%%%%%

  We now recall the construction of $p$-adic families of Galois representations attached to modular forms via ``big'' \'etale sheaves on modular curves. We follow the account of \cite[\S 3]{andreattaiovitastevens}, but with somewhat altered conventions (for reasons which will become clear later). We also use some results of Hansen \cite{Hansen-Iwasawa} (from whom we have also borrowed the terminology ``overconvergent \'etale cohomology'').

 \subsection{Setup and notation}

  \begin{definition}
   We write $\cW$ for the rigid-analytic space over $\Qp$ parametrizing continuous characters of the group $\Zp^\times$. For an integer $m \ge 0$, we shall write $\cW_m$ for the wide open subspace parametrizing ``$m$-accessible'' weights, which are those satisfying $v_p(\kappa(t)^{p-1} - 1) > \frac{1}{p^m(p-1)}$ for all $t \in \Zp^\times$.
  \end{definition}

  \begin{remark}
   Note that $\cW$ is isomorphic to a disjoint union of $p-1$ open unit discs, and the bounded-by-1 rigid-analytic functions on $\cW$ are canonically $\Lambda(\Zp^\times)$; while $\cW_m$ is the union of the corresponding open subdiscs of radius $p^{-1/p^m(p-1)}$ with centres in $\Zp^\times$. Thus $\cW_0$ (which is the space denoted by $\cW^*$ in \cite{andreattaiovitastevens}) contains every $\Qp$-point of $\cW$, and in particular every weight of the form $z \mapsto z^j$, $j \in \ZZ$.
  \end{remark}

  Now let us fix some coefficient field $E$ (a finite extension of $\Qp$) with ring of integers $\cO_E$.

  \begin{definition}
   We let $U$ denote a wide open disc defined over $E$, contained in $\cW_m$ for some $m \ge 0$; and $\Lambda_U$ the $\cO_E$-algebra of rigid functions on $U$ bounded by 1 (so $\Lambda_U \cong \cO_E[[u]]$). We write $\kappa_U$ for the universal character $\Zp^\times \into \Lambda(\Zp^\times)^\times \to \Lambda_U^\times$.
  \end{definition}

  The ring $\Lambda_U$ is endowed with two topologies: the $p$-adic topology (which we shall not use) and the $m_U$-adic topology, which is the topology induced by the ideals $m_U^n$, where $m_U$ is the maximal ideal of $\Lambda_U$.

  \begin{definition}
   For $m \ge 0$, we write $LA_m(\Zp, \Lambda_U)$ for the space of functions $\Zp \to \Lambda_U$ such that for all $a \in \ZZ / p^m \ZZ$, the function $z \mapsto f(a + p^m z)$ is given by a power series $\sum_{n \ge 0} b_n z^n$ with $b_n \to 0$ in the $m_U$-adic topology of $\Lambda_U$.
  \end{definition}

  \begin{lemma}
   \label{lemma:convergencekappa}
   If $U \subseteq \cW_m$, then the function $z \mapsto \kappa_U(1 + pz)$ is in $LA_m(\Zp, \Lambda_U)$.
  \end{lemma}

  \begin{proof}
   This is a standard computation, but we have not been able to find a reference, so we shall give a brief sketch of the proof. Let us write $X_m$ for the affinoid rigid-analytic space over $\Qp$ defined by $\{ x: |x - a| \le p^{-m}\text{ for some } a \in \Zp\} \subseteq \mathbf{A}^1_\rig$. Then $LA_m(\Zp, \Lambda_U)$ is precisely the space of functions $\Zp \to \Lambda_U$ which extend to rigid-analytic $\Lambda_U$-valued functions on $X_m$.

   Firstly, the map $x \mapsto \frac{\log(1 + px)}{\log(1 + p)}$ is a bijection from $\Zp$ to $\Zp$ which extends to a rigid-analytic isomorphism from $X_m$ to itself for every $m$; so it suffices to show that $x \mapsto \kappa_U( (1 + p)^x )$ extends to a $\Lambda_U$-valued rigid-analytic function on $X_m$ whenever $U \subseteq \cW_m$. It suffices to consider the universal case $U = \cW_m$. After enlarging the coefficient field $E$ if necessary, we identify $\Lambda_U$ with $\cO_E[[u]]$ in such a way that $\kappa_U(1 + p) = 1 + \varepsilon  u$ where $\varepsilon$ is some element of $\cO_E$ of valuation $\frac{1}{(p-1)p^m}$. Then
   \[ \kappa_U( (1 + p)^x ) = \sum_{n \ge 0} \binom{x}{n} \varepsilon^n u^n,\]
   and we have $\varepsilon^n \binom{x}{m} \in LA_m(\Zp, \Zp)$ for any $n$, by \cite[Theorem 1.29]{colmez-fonctions}.
  \end{proof}

  \begin{remark}
   It is important to use the right topology on $\Lambda_U$, because if one takes $U = \cW_m$ and writes $x \mapsto \kappa_U( 1 + p^{m+1}x )$ as a series $\sum c_n x^n$ with $c_n \in \Lambda_U$, the $c_n$ tend to zero $m_U$-adically (the above argument shows in fact that $c_n \in m_U^n$), but they do \emph{not} tend to zero $p$-adically.
  \end{remark}

 \subsection{The spaces $D_U^\circ(T_0)$ and $D_U^\circ(T_0')$}

  \begin{definition}
   Let $H$ be the group $\Zp^{\oplus 2}$. We define subsets $T_0, T_0' \subset H$ by
   \begin{align*}
    T_0 &\coloneqq \Zp^\times \times \Zp,&
    T_0' &\coloneqq p\Zp \times \Zp^\times.
   \end{align*}
  \end{definition}

  \begin{proposition}
   The subset $T_0$ is preserved by right multiplication by the monoid $\Sigma_0(p) = \tbt{\Zp^\times}{\Zp}{p\Zp}{\Zp} \subset \operatorname{Mat}_{2 \times 2}(\Zp)$, and $T_0'$ by the monoid $\Sigma_0'(p) = \tbt{\Zp}{\Zp}{p\Zp}{\Zp^\times}$. In particular, both $T_0$ and $T_0'$ are preserved by scalar multiplication by $\Zp^\times$.\qed
  \end{proposition}

  \begin{remark}
   The definition of $T_0$ coincides with that used in \cite{andreattaiovitastevens} (and our $\Sigma_0(p)$ is their $\Xi(p)$). The subspace $T_0'$ is the image of $T_0$ under right multiplication by $\stbt 0 {-1} p 0$, and conjugation by this element interchanges $\Sigma_0(p)$ and $\Sigma_0'(p)$.
  \end{remark}

  \begin{definition}
   For $m \ge 0$, we write $A^\circ_{U, m}(T_0)$ for the space of functions
   \[ f: T_0 \to \Lambda_U\]
   which are homogenous of weight $\kappa_U$, i.e.~satisfy
   \[ f(\gamma t) = \kappa_U(\gamma) f(t) \]
   for $\gamma \in \Zp^\times$, $t \in T_0$, and are such that the function $z\mapsto f(1, z)$ lies in $LA_m(\Zp, \Lambda_U)$. We equip this module with the topology defined by the subgroups $m_U^n A^\circ_{U, m}$.

   Similarly, we write $A^\circ_{U, m}(T_0')$ for the space of functions $T_0' \to \Lambda_U$ which are homogenous of weight $\kappa_U$ and are such that $z \mapsto f(pz, 1) \in LA_m(\Zp, \Lambda_U)$, again endowed with the $m_U$-adic topology.
  \end{definition}

  \begin{proposition}
   If $U \subseteq \cW_m$, then the space $A^\circ_{U, m}(T_0)$ is preserved by the left action of $\Sigma_0(p)$ on functions $T_0 \to \Lambda_U$ defined by
   \[ (\gamma f)(t) = f(t\gamma),\]
   and similarly for $A^\circ_{U, m}(T_0')$.
  \end{proposition}

  \begin{proof}
   We give the proof for $T_0'$; the proof for $T_0$ is similar.

   Unravelling the definition of the actions, we must show that if $\gamma = \stbt a b {pc} d \in \Sigma_0'(p)$ and $f \in A^\circ_U(T_0')$, then the function
   \[ z \mapsto \kappa_U(d) \kappa_U(1 + pd^{-1}bz) f\left(p \cdot \frac{c + az}{d + pbz}, 1\right)\]
   is in $LA_m(\Zp, \Lambda_U)$. Since $LA_m(\Zp, \Lambda_U)$ is closed under multiplication, and contains $\Zp$, it suffices to check that $z \mapsto \kappa_U(1 + pd^{-1}bz)$ and $z \mapsto f\left(p \cdot \frac{c + az}{d + pbz}, 1\right)$ are in this space. For the factor $\kappa_U(1 + pd^{-1}bz)$ this follows from Lemma \ref{lemma:convergencekappa}.

   For the factor $f\left(p \cdot \frac{c + az}{d + pbz}, 1\right)$, we note that the map $z \mapsto \frac{c + az}{d + pbz}$ preserves all the rigid-analytic neighbourhoods $X_m$ of $\Zp$, so it preserves the ring of rigid-analytic functions convergent and bounded by 1 on these spaces; thus $z \mapsto g\left(\frac{c + az}{d + pbz}\right)$ is in $LA_m(\Zp, \Lambda_U)$ if $g \in LA_m(\Zp, \Lambda_U)$.
  \end{proof}

  For the rest of this section, let $T$ denote either $T_0$ or $T_0'$, and $\Sigma$ either $\Sigma_0$ or $\Sigma'_0$ respectively.

  Note that as a topological $\Lambda_U$-module, $A^\circ_{U, m}(T)$ is isomorphic to the space of countable sequences $(c_n)_{n=1}^\infty$ with $c_n \in \Lambda_U$ such that $c_n \to 0$ in the $m_U$-adic topology.

  \begin{definition}
   We write
   \[
    D^\circ_{U, m}(T) =
    \operatorname{Hom}_{\Lambda_U}(A^\circ_{U, m}(T), \Lambda_U),
   \]
   and $D_{U, m}(T) = D^\circ_{U, m}(T)[1/p]$.
  \end{definition}

  Note that any linear functional $\mu \in D^\circ_{U, m}(T)$ is necessarily continuous (where we endow both $A^\circ_{U, m}(T)$ and $\Lambda_U$ with their $m_U$-adic topologies). We equip $D^\circ_{U, m}(T)$ with the weak (or more formally weak-star) topology, generated by sets of the form $\{ \mu: \mu(f) \in m_U^n \}$ for $f \in A^\circ_{U, m}(T)$ and $n \ge 0$, i.e.~the weakest topology such that all the evaluation-at-$f$ morphisms are continuous (when the target $\Lambda_U$ is equipped with the $m_U$-adic topology).

  In this topology $D^\circ_{U, m}(T)$ becomes compact; indeed, we have a topological isomorphism $D^\circ_U \to \prod_{n=0}^\infty \Lambda_U$, with the inverse-limit topology.

  \begin{lemma}
   \label{lemma:basechange1}
   The formation of $D^\circ_{U, m}(T)$ commutes with base-change in $U$, in the sense that for $V \subseteq U$ two open discs defined over $E$, we have
   \[ D^\circ_{U, m}(T) \htimes_{\Lambda_U} \Lambda_V = D^\circ_{V, m}(T).\]
  \end{lemma}

  \begin{proof}
   Clear by construction.
  \end{proof}

  \begin{lemma}
   We may write $D^\circ_{U, m}(T)$ as an inverse limit
   \[ D^\circ_{U, m}(T) = \varprojlim_n D^\circ_{U, m}(T)/ \Fil^n, \]
   where each $\Fil^n$ is preserved by the action of $\Sigma$, and the quotient $D^\circ_{U, m}(T)/ \Fil^n$ is finite.
  \end{lemma}

  \begin{proof}
   For $T = T_0$ and $m = 0$ this is \cite[Proposition 3.10]{andreattaiovitastevens}, and the generalisation to $m \ge 1$ is given in \cite[\S 2.1]{Hansen-Iwasawa}. The case of $T = T_0'$ is proved similarly (or, alternatively, follows from the case of $T = T_0$ via conjugation by $\stbt 0 {-1} p 0$).
  \end{proof}

  \begin{proposition}
   Let $D^{\la}(T, E)$ be the algebra of $E$-valued locally analytic distributions on $T$. Then there is an isomorphism
   \[ D^{\la}(T, E) \to \varprojlim_{U, m} D_{U, m}(T),\]
   given by mapping the Dirac distribution $[t]$, for $t \in T$, to the $\Lambda_U$-linear functional on $A^\circ_{U, m}$ given by evaluation at $t$. This map commutes with the action of $\Sigma$ on both sides, and restricts to an isomorphism
   \[ \Lambda_{\cO_E}(T) \to \varprojlim_{U, m} D_{U, m}^\circ(T).\]
  \end{proposition}

  \begin{proof}
   We give the proof for $T_0'$, the proof for $T_0$ being similar. Because of the homogeneity requirement, any function in $A^\circ_{U, m}(T_0')$ is uniquely determined by its restriction to $p\Zp \times 1$, and this gives an isomorphism $D^\circ_{U, m}(T) \cong LA_m(\Zp, \cO_E)^* \htimes_{\cO_E} \Lambda_U$. Both results now follow by passing to the inverse limit.
  \end{proof}

  Now let $k \in \cW$ be an integer weight (i.e.~of the form $z \mapsto z^k$ with $k \ge 0$); any such weight automatically lies in $\cW_0$. As for $U$ above, we may define a space $A^\circ_{k, m}(T)$ of $m$-analytic $\cO_E$-valued functions on $T$ homogenous of weight $k$, and its dual $D^\circ_{k, m}(T)$, for any $m \ge 0$.

  Restriction to $T$ gives a natural embedding $P^\circ_k \into A^\circ_{k, m}(T)$, where $P^\circ_k$ is the space of \emph{polynomial} functions on $\Zp^2$, homogenous of degree $k$, with $\cO_E$ coefficients. Dually, we obtain a canonical, $\Sigma_0(p)$-equivariant projection $\rho_k: D^\circ_{k, m} \to (P^\circ_k)^* = \TSym^k \cO_E^2$.

  \begin{proposition}
   \label{prop:diagram1}
   The following diagram is commutative, for any $U$, any $m$ sufficiently large that $U \subseteq \cW_m$, and any $k \in U$:
   \begin{diagram}
    \Lambda(T)  & \rTo & D^\circ_{U, m}(T) & \rTo & D^\circ_{k, m}(T) \\
    \dInto &&&& \dTo_{\rho_k}\\
    \Lambda(H) & &\rTo^{\mom^k} && \TSym^k H
   \end{diagram}
   Here $\mom^k$ is as defined in \cite{Kings-Eisenstein}, and the left vertical arrow is the natural inclusion $T \into \Zp^{\oplus 2}$.
  \end{proposition}

  \begin{proof}
   This is clear by construction.
  \end{proof}

 \subsection{The Ohta pairing}

  We now define a pairing between distribution modules on $T_0$ and $T_0'$, following \cite[\S 4]{Ohta95}.

  \begin{definition}
   Let $H = \Zp^{\oplus 2}$, as above. We define a bilinear map $\phi: H \times H \to \Zp$ by
   \[ \phi\left( (x_1, y_1), (x_2, y_2)\right) = x_1 y_2 - x_2 y_1. \]
  \end{definition}

  This clearly restricts to a map $T_0 \times T_0' \to \Zp^\times$; so the $\Lambda_U$-valued function $\Phi$ on $T_0 \times T_0'$ given by $\Phi(t, t') = \kappa_U(\phi(t, t'))$ is well-defined, homogenous of weight $\kappa_U$ in either variable, and $m$-analytic whenever $U \subseteq \cW_m$.

  \begin{definition}
   We write
   \[ \{ -, -\} : D^\circ_{U, m}(T_0) \times D^\circ_{U, m}(T_0') \to \Lambda_U \]
   for the bilinear map given by pairing with the function $\Phi \in A^\circ_{U, m}(T_0) \htimes_{\Lambda_U} A^\circ_{U, m}(T_0')$.
  \end{definition}

  This is evidently $\Lambda_U$-bilinear, and it satisfies
  \[ \{ \mu \gamma, \mu' \gamma\} = \kappa_U(\det \gamma) \cdot \{  \mu, \mu'\} \]
  for any $\mu \in D^\circ_{U, m}(T_0)$, $\mu' \in D^\circ_{U, m}(T_0')$, and $\gamma \in U_0(p)$, where $U_0(p) = \Sigma_0(p) \cap \Sigma_0'(p)$ is the Iwahori subgroup of $\GL_2(\Zp)$.

  \begin{remark}
   Let us describe the above map slightly more concretely. We take $m = 0$, for simplicity; then the functions $f_n( (x, y) ) = \kappa_U(x) \cdot (y/x)^n$ are an orthonormal basis of $A^\circ_{U, 0}(T_0)$, so a distribution $\mu \in D^\circ_{U, 0}(T_0)$ is uniquely determined by its moments $\mu_n = \mu(f_n)$, which can be any sequence of elements of $\Lambda_U$. Similarly, the functions $g_n( (px, y) ) = \kappa_U(y) (x/y)^n$ are an orthonormal basis of $A^\circ_{U, 0}(T_0')$ and any $\mu' \in D^\circ_{U, 0}(T_0')$ is uniquely determined by its moments $\mu_n' = \mu'(g_n)$.

   Given such $\mu, \mu'$, we define an element of $\Lambda_U$ as follows: the function $\Phi\left( (1, z), (pw, 1) \right) = \kappa_U(1 - p z w)$ can be written as a power series $\sum a_n (wz)^n$, with $a_n \in \Lambda_U$ such that $a_n \to 0$ in the $m_U$-adic topology, by Lemma \ref{lemma:convergencekappa}; then $\{ \mu, \mu'\}$ is the value of the convergent sum $\sum_{n \ge 0} a_n \mu_n \mu_n'$.
  \end{remark}

 \subsection{Sheaves on modular curves}

  \begin{notation}
   Let $M, N$ be integers $\ge 1$ with $M \mid N$ and $M + N \ge 5$. We write $Y(M, N)$ for the modular curve over $\ZZ[1/N]$ defined in \cite[\S 2.1]{Kato-p-adic}.
  \end{notation}

  We recall the construction of an \'etale sheaf of abelian groups $\sH_{\Zp}$, and the corresponding sheaf of Iwasawa algebras $\Lambda(\sH_{\Zp})$, associated to the universal elliptic curve $\cE$ over $Y(M, N)$; and more generally the sheaf of sets $\sH_{\Zp}\langle D \rangle$ and sheaf of $\Lambda(\sH_{\Zp})$-modules $\Lambda(\sH_{\Zp}\langle D \rangle)$, where $D$ is a subscheme of $\cE$ finite \'etale over $Y(M, N)$. Cf.~\cite[\S 4.1]{KLZ1b}.

  We shall apply this to the curve $Y = Y(1, N(p))$ where $p \nmid N$, parametrising triples $(E, P, C)$ where $E$ is an elliptic curves (over some $\ZZ[1/Np]$-algebra), $P$ is a point of exact order $N$ on $E$, and $C$ is a subgroup of $E$ of order $p$. Let $D = E[p] - C$, which is finite \'etale over $Y$ of degree $p^2 - p$, and $D' = C - \{ 0\}$, which is finite \'etale of degree $p-1$; then the sheaves $\sH_{\Zp}\langle D \rangle$ and $\sH_{\Zp} \langle D' \rangle$ are defined. Since both $D$ and $D'$ are contained in $E[p]$, there is a multiplication-by-$p$ map
  \[ [p]_*: \Lambda(\sH_{\Zp}\langle D \rangle) \to \Lambda(\sH_{\Zp}),\]
  and similarly for $D'$.

  \begin{proposition}
   The pullbacks of the sheaves $\Lambda(\sH_{\Zp})$, and $\Lambda(\sH_{\Zp}\langle D \rangle)$, and $\Lambda(\sH_{\Zp}\langle D' \rangle)$ to the pro-scheme $Y(p^\infty, Np^\infty)$ are isomorphic to the constant sheaves $\Lambda(\Zp^2)$, $\Lambda(T_0)$, and $\Lambda(T_0')$ respectively; and the maps $[p]_*$ are induced by the natural inclusions $T_0 \into \Zp^2$ and $T_0' \into \Zp^2$.
  \end{proposition}

  \begin{proof}
   It suffices to check the corresponding statement for the inverse systems of sheaves of \emph{sets} $\sH_{\Zp}$, $\sH_{\Zp}\langle D \rangle$ and $\sH_{\Zp}\langle D' \rangle$. However, over $Y(p^\infty, Np^\infty)$ we have two sections $e_1, e_2$ of $\sH_{\Zp}$ identifying it with the constant sheaf $\Zp^2$; and since the level $p$ subgroup $C$ is generated by $e_2 \bmod p$, the sheaf $\sH_{\Zp}\langle D \rangle$ is precisely the subset of linear combinations $ae_1 + be_2$ such that $a \ne 0 \bmod p$, which is $T_0$, while $\sH_{\Zp}\langle D' \rangle$ is similarly identified with $T_0'$.
  \end{proof}

  Now let $m \ge 0$, and $U$ a wide open disc contained in $\cW_m$, as before.

  \begin{proposition}
   There are pro-sheaves of $\Lambda_U$-modules $\cD^\circ_{U, m}(\sH_0)$ and $\cD^\circ_{U, m}(\sH_0')$ on $Y$, whose pullbacks to $Y(p^\infty, Np^\infty)$ are the constant pro-sheaves $D^\circ_{U, m}(T_0)$ and $D^\circ_{U, m}(T_0')$ respectively; and the Galois group of $Y(p^\infty, Np^\infty) / Y$ acts on $D^\circ_{U, m}(T_0)$ and $D^\circ_{U, m}(T_0')$ via its natural identification with the Iwahori subgroup of $\GL_2(\Zp)$.
  \end{proposition}

  \begin{proof}
   The above trivialisation of $\sH_{\Zp}$ over $Y(p^\infty, Np^\infty)$ determines a homomorphism from the \'etale fundamental group $\pi_1^{\et}(Y)$ to the Iwahori subgroup $U_0(p) \subseteq \GL_2(\Zp)$. Since $D^\circ_{U, m}(T_0)$ is an inverse limit of finite right modules for $U_0(p)$, and any finite right $\pi_1^{\et}(Y)$-module defines an \'etale sheaf on $Y$, we obtain a pro-sheaf $\cD^\circ_{U, m}(\sH_0)$, and similarly for $D^\circ_{U, m}(T_0')$. These are sheaves of $\Lambda_U$-modules since the action of $U_0(p)$ on the modules $D^\circ_{U, m}(T_0)$ and $D^\circ_{U, m}(T_0')$ is $\Lambda_U$-linear.
  \end{proof}

  \begin{remark}
   Compare \cite[\S 3.3]{andreattaiovitastevens}; the argument is given there for the Kummer \'etale site on a log rigid space over $\Qp$ (with log-structure given by the cusps), but the argument works equally well in the much simpler case of affine modular curves over $\QQ$.
  \end{remark}

  \begin{proposition}
   \label{prop:diagram2}
   For any $k \in U$ we have commutative diagrams of pro-sheaves on $Y$
   \begin{diagram}
    \Lambda(\sH_{\Zp}\langle D \rangle) & \rTo & \cD^\circ_{U, m}(\sH_0) & \rTo & \cD^\circ_{k, m}(\sH_0) \\
    \dTo^{[p]_*} && && \dTo_{\rho_k}\\
    \Lambda(\sH_{\Zp}) & & \rTo^{\mom^k} && \TSym^k\left(\sH\right)
   \end{diagram}
   and
   \begin{diagram}
    \Lambda(\sH_{\Zp}\langle D' \rangle) & \rTo & \cD^\circ_{U, m}(\sH_0') & \rTo & \cD^\circ_{k, m}(\sH_0') \\
    \dTo^{[p]_*} && && \dTo_{\rho_k}\\
    \Lambda(\sH_{\Zp}) & & \rTo^{\mom^k} && \TSym^k\left(\sH\right)
   \end{diagram}
   Here $\mom^k$ is as defined in \cite{Kings-Eisenstein}.
  \end{proposition}

  \begin{proof}
   We have the diagram of proposition \ref{prop:diagram1}, which we may interpret as a diagram of constant pro-sheaves on $Y(p^\infty, Np^\infty)$; and the morphisms in the diagram are all equivariant for the action of the Iwahori subgroup, so they descend to morphisms of sheaves on $Y$.
  \end{proof}

  We can similarly construct $\cD^\circ_{U, m}(\sH_0)$ and $\cD^\circ_{U, m}(\sH_0')$ as sheaves on $Y(U)$, for any sufficiently small open compact subgroup $U \subseteq \GL_2(\widehat \ZZ)$ whose image in $\GL_2(\Zp)$ is contained in the Iwahori subgroup. Moreover, if $g \in \operatorname{GL}_2(\QQ) \cap \Sigma_0(p)$, so there is a natural map
  \[ Y(U) \to Y(g U g^{-1}) \]
  corresponding to $z \mapsto gz$ on the upper half-plane, then the action of $g$ on $D^\circ_{U, m}(\sH_0)$ gives a map of sheaves on $Y$
  \[ \cD^\circ_{U, m}(\sH_0) \to g^*\left(\cD^\circ_{U, m}(\sH_0)\right); \]
  the same holds with $\sH_0'$ and $\Sigma_0'$ in place of $\sH_0$ and $\Sigma_0$.

  \begin{definition}
   We define
   \begin{align*}
    M^\circ_{U, m}(\sH_0) &= H^1_{\et}\left(\overline{Y}, \mathcal{D}_{U, m}^\circ(\sH_0)\right)(-\kappa_U),\\
    M^\circ_{U, m}(\sH_0') &= H^1_{\et}\left(\overline{Y}, \mathcal{D}_{U, m}^\circ(\sH'_0)\right)(1).
   \end{align*}
   We also make the same definitions for compactly-supported and parabolic cohomology, which we write as $M^\circ_{U, m}(\sH_0)_c$, $M^\circ_{U, m}(\sH_0)_{\mathrm{par}}$ (and similarly for $\sH_0'$).
  \end{definition}

  These are profinite topological $\Lambda_U$-modules, equipped with continuous actions of $\Gal(\overline{\QQ} / \QQ)$ unramified outside $N p \infty$. As topological $\Lambda_U$-modules (forgetting the Galois actions) they are isomorphic to more familiar objects:
  \begin{itemize}
   \item The space $M^\circ_{U, m}(\sH_0)$ is isomorphic to the group cohomology $H^1\left(\Gamma, \mathcal{D}_{U, m}^\circ(T_0)\right)$, where $\Gamma = \Gamma_1(N(p)) = \Gamma_1(N) \cap \Gamma_0(p)$ (since $Y_1(N(p))(\CC)$ has contractible universal cover and its fundamental group is $\Gamma_1(N) \cap \Gamma_0(p)$). 
   \item The space $M^\circ_{U, m}(\sH_0)_c$ is isomorphic to the space of \emph{modular symbols}
   \[ \Hom_{\Gamma}\left( \operatorname{Div}^0(\mathbf{P}^1_{\QQ}), \mathcal{D}_{U, m}^\circ(T_0)\right). \]
  \end{itemize}
  The same statements hold with $\sH_0'$ and $T_0'$ in place of $\sH_0$ and $T_0$.

  \begin{notation}
   We shall refer to $M^\circ_{U, m}(\sH_0)$ and $M^\circ_{U, m}(\sH_0')$ as \emph{\'etale overconvergent cohomology} (of weight $U$, tame level $N$, and degree of overconvergence $m$).
  \end{notation}
  
  We now state some properties of these modules:
  
  \begin{proposition}
   \mbox{~}
   \begin{enumerate}
    \item (Compatibility with specialisation) 
    Let $\varpi_k$ be the ideal of $\Lambda_U$ corresponding to the character $z \mapsto z^k$. For any integer $k \ge 0 \in U$, there is an isomorphism
    \[ M^\circ_{U, m}(\sH_0) / \varpi_k \cong M^\circ_{k, m}(\sH_0).\]
    For compactly-supported cohomology this is true for $k \ge 1$, while for $k = 0$ we have an injective map
    \[ M^\circ_{U, m}(\sH_0)_c / \varpi_0 \into M^\circ_{0, m}(\sH_0)_c \]
    whose cokernel has rank 1 over $\cO_E$, with the Hecke operator $U_p$ acting as multiplication by $p$. Similar statements hold for $\sH_0'$ in place of $\sH_0$.
    
    \item (Control theorem) For any integer $k \ge 0$, the map
    \[ M_{k, m}(\sH_0) \rTo^{\rho_k} H^1_{\et}(\overline{Y}, \TSym^k(\sH)(-k))[1/p] \]
    is an isomorphism on the $U_p = \alpha$ eigenspace, for any $\alpha$ of valuation $< k + 1$. The same holds for compactly-supported and parabolic cohomology, and for $\sH_0'$ and $U_p'$ in place of $\sH_0$ and $U_p$.
    
    \item (Duality) There are $\Lambda_U$-bilinear, $G_{\Qp}$-equivariant pairings
    \begin{align*}
     M^\circ_{U, m}(\sH_0)_c \times M^\circ_{U, m}(\sH_0') &\to \Lambda_U, \\
     M^\circ_{U, m}(\sH_0) \times M^\circ_{U, m}(\sH_0')_c &\to \Lambda_U,\\
     M^\circ_{U, m}(\sH_0)_{\mathrm{par}} \times M^\circ_{U, m}(\sH_0')_{\mathrm{par}} &\to \Lambda_U,
    \end{align*}
    which we denote by $\{ - , -\}$. For integers $k \ge 0$ we have
    \[ \operatorname{ev}_k\left( \{ x, x'\} \right) = \{ \rho_k(x), \rho_k(x') \}_k \]
    where $\operatorname{ev}_k$ is evaluation at $k$, and on the right-hand side $\{-, -\}_k$ signifies the Poincar\'e duality pairing.
    
    \item There is an isomorphism $W: M^\circ_{U, m}(\sH_0)_{?} \to M^\circ_{U, m}(\sH_0')_?$ (where $? \in \{ \varnothing, c, \mathrm{par}\}$), intertwining the action of the Hecke operators $T_n$ with the $T_n'$ (including $n = p$); this is compatible via the maps $\rho_k$ with the Atkin--Lehner operator $W_{Np}$ (but \textbf{not} with the Galois action).
   \end{enumerate}
  \end{proposition}
  
  \begin{proof}
   For part (1), see \cite[Lemma 3.18]{andreattaiovitastevens}. For compactly-supported cohomology see \cite[Theorem 3.10]{Bellaiche-critical}. (Bella\"iche works with coefficients in an affinoid disc, rather than a wide-open disc as we do, but the argument is the same.)
  
   Part (2) is the celebrated Stevens control theorem; see \cite[Theorem 3.16]{andreattaiovitastevens} for $H^1$, and \cite[Theorem 1.1]{PollackStevens-overconvergent} for $H^1_c$.
   
   For part (3), if we identify $\Zp(1)$ with $\Zp$ as sheaves on $Y(p^\infty, Np^\infty)$ via the section given by the Weil pairing and our trivialisation of $\sH$, then the Iwahori subgroup $U_0(p)$ acts on $\Zp(1)$ via the determinant character, and hence our pairing of $U_0(p)$-modules
   $D^\circ_{U, m}(\sH_0) \times D^\circ_{U, m}(\sH_0') \to \Lambda_U$ gives a pairing of \'etale pro-sheaves on $Y$
   \[ \cD_{U, m}(\sH_0) \times \cD_{U, m}(\sH_0) \to \Lambda_U(\kappa_U), \]
   where $\kappa_U$ is the composite of the cyclotomic character with the canonical map $\Zp^\times \to \Lambda_U^\times$. Hence we have a cup-product pairing
   \[ H^1_c(\overline{Y}, \cD_{U, m}(\sH_0)(1)) \times H^1(\overline{Y}, \cD_{U, m}(\sH_0')(1)) \to H^2_c(\overline{Y}, \Lambda_U(2 + \kappa_U)), \]
   and since there is a canonical isomorphism $H^2_c(\overline{Y}, \Zp(1)) \cong \Zp$, this gives a pairing into $\Lambda_U(1 + \kappa_U)$ as claimed. It is clear by construction that this is compatible with the Poincar\'e duality pairings with $\TSym^k$ coefficients for each $k \ge 0$.
   
   Part (4) follows from the fact that the action of the matrix $\stbt 0 {-1} {Np} 0$ on $H$ interchanges $T_0$ and $T_0'$.
  \end{proof}
  
  \begin{remark}
   The pairing $\{ - , -\}$ (in any of its various incarnations) is far from perfect (since its specialisation at a classical weight $k \ge 0$ factors through the maps $\rho_k$, so any non-classical eigenclass of weight $k$ must be in its kernel). Nonetheless, we shall see below that it induces a perfect pairing on small slope parts.
  \end{remark}

 \subsection{Slope decompositions}

  As before, let $U$ be a wide open disc contained in $\cW_m$, for some $m$. Let $B_U = \Lambda_U[1/p]$, and let $M$ be one of the $B_U$-modules $M_{U, m}(\sH_0)_?$, for $? \in \{ \varnothing, c, \mathrm{par}\}$, and let $\lambda \in \RR_{\ge 0}$.

  \begin{definition}
   We say $M$ has a \emph{slope $\le \lambda$ decomposition} if we can write it as a direct sum of $B_U$-modules
   \[ M = M^{(\le \lambda)} \oplus M^{(> \lambda)},\]
   where the following conditions are satisfied:
   \begin{itemize}
    \item the action of the Hecke operator $U_p$ preserves the two summands;
    \item the module $M_U^{(\le  \lambda)}$ is finitely-generated over $B_U$;
    \item the restrictions of $U_p$ to $ M_U^{(\le  \lambda)}$ and $M_U^{(> \lambda)}$ have slope $\le  \lambda$ and slope $> \lambda$ respectively.
   \end{itemize}
  \end{definition}
  
  \begin{remark}
    There are several equivalent definitions of \emph{slope $\le \lambda$}, see \cite{andreattaiovitastevens} for further discussion. We shall use the following formulation: the endomorphism $U_p$ of $M_U^{(\le \lambda)}$ is invertible, and the sequence of endomorphisms $\left(p^{\lfloor n\lambda\rfloor} \cdot  (U_p)^{-n}\right)_{n \ge 0}$ is bounded in the operator norm.
  \end{remark}
  Note that the summands $M^{(\le \lambda)}$ and $M^{(> \lambda)}$ must be stable under the actions of the prime-to-$p$ Hecke operators, and of the Galois group $G_{\QQ}$, since these commute with the action of $U_p$.

  \begin{theorem}[{\cite[Theorem 3.17]{andreattaiovitastevens}}]
   Let $k \ge 0$ and $0 \le \lambda < k + 1$. Then there exists an open disc $U \ni k$ in $\cW$, defined over $E$, such that the module $M_{U, 0}(\sH_0)$ has a slope $\le \lambda$ decomposition.
  \end{theorem}

  The same results hold \emph{mutatis mutandis} for $M = M_{U, 0}(\sH_0')$, using the Hecke operator $U_p'$ in place of $U_p$; this follows directly from the previous statement using the isomorphism between the two modules provided by the Atkin--Lehner involution. There are also corresponding statements for compactly-supported and parabolic cohomology.
 
 \subsection{Coleman families}
 
  A considerably finer statement is possible if we restrict to a ``neighbourhood'' of a classical modular form. We make the following definition:
 
  \begin{definition}
   Let $U \subseteq \cW$ be an open disc such that the classical weights $U \cap \ZZ_{\ge 0}$ are dense in $U$. A \emph{Coleman family} $\cF$ over $U$ (of tame level $N$) is a power series
   \[ \cF = \sum_{n \ge 1} a_n(\cF) q^n \in \Lambda_U[[q]], \]
   with $a_1(\cF) = 1$ and $a_p(\cF)$ invertible in $B_U$, such that for all but finitely many classical weights $k \in U \cap \ZZ_{\ge 0}$, the series $\cF_k = \sum_{n \ge 1} a_n(\cF)(k) \in \cO_E[[q]]$ is the $q$-expansion of a classical modular form of weight $k + 2$ and level $\Gamma_1(N) \cap \Gamma_0(p)$ which is a normalised eigenform for the Hecke operators.
  \end{definition}
 
  \begin{remark}
   This definition is somewhat crude, since for a more satisfying theory one should also consider more general classical weights of the form $z \mapsto z^k \chi(z)$ for $\chi$ of finite order, and allow families indexed by a finite flat rigid-analytic cover of $U$ rather than by $U$ itself. This leads to the construction of the eigencurve. However, the above definition will suffice for our purposes, since we are only interested in small neighbourhoods in the eigencurve around a classical point.
  \end{remark}
  
  \begin{definition}
   \label{def:noble}
   A \emph{noble eigenform} of tame level $N$ is a normalised cuspidal Hecke eigenform $f_{\alpha}$ of level $\Gamma_1(N) \cap \Gamma_0(p)$ and some weight $k + 2 \ge 2$, with coefficients in $E$, having $U_p$-eigenvalue $\alpha = a_p(f_\alpha)$, such that:
   \begin{itemize}
    \item $f_\alpha$ is a $p$-stabilisation of a newform $f$ of level $N$ whose Hecke polynomial $X^2 - a_p(f) X + p^{k+1} \varepsilon_f(p)$ has distinct roots (``$p$-regularity'');
    \item if $v_p(\alpha) = k + 1$, then the Galois representation $M_E(f) |_{G_{\Qp}}$ is not a direct sum of two characters (``non-criticality'').
   \end{itemize}
  \end{definition}
  
  \begin{theorem}
   \label{thm:colemanfamily}
   Suppose $f_\alpha$ is a noble eigenform of weight $k_0 + 2$. Then there exists a disc $U \ni k_0$ in $\cW$, and a unique Coleman family $\cF$ over $U$, such that $\cF_{k_0} = f_\alpha$. 
  \end{theorem}
  
  \begin{proof}
   This follows from the fact that the Coleman--Mazur--Buzzard eigencurve $\mathscr{C}(N)$ of tame level $N$ is \'etale over $\cW$ (and, in particular, smooth) at the point corresponding to a noble eigenform $f_\alpha$. See \cite{Bellaiche-critical}.
  \end{proof}
  
  \begin{remark}
   As remarked in \cite{Hansen-Iwasawa}, the condition that the Hecke polynomial of $f$ has distinct roots is conjectured to be redundant, and known to be so when $f$ has weight 2; and it is also conjectured that the only newforms $f$ of weight $\ge 2$ such that $M_E(f)|_{G_{\Qp}}$ splits as a direct sum are those which are of CM type with $p$ split in the CM field.
  \end{remark}
  
  \begin{theorem}
   \label{thm:colemanrep}
   Let $f_\alpha$ be a noble eigenform, and $\cF$ the Coleman family passing through $f_\alpha$. If the disc $U \ni k_0$ is sufficiently small, then:
   \begin{itemize}
    \item The module
    \[ 
     M_U(\cF) \coloneqq M_{U, 0}(\sH_0)
     \Big[T_n = a_n(\cF)\ \forall n \ge 1\Big] 
    \]
    is a direct summand of $M_{U, 0}(\sH_0)$ as a $B_U$-module, free of rank 2 over $B_U$, and lifts canonically to $M_{U, 0}(\sH_0)_c$.
    \item The same is true of the module
    \[ M_U(\cF)^* \coloneqq M_{U, 0}(\sH_0')
       \Big[T_n' = a_n(\cF)\ \forall n \ge 1\Big].
    \]
    \item The pairing $\{-, -\}$ induces an isomorphism of $B_U[G_{\Qp}]$-modules
    \[ M_U(\cF)^* \cong \Hom_{B_U}(M_U(\cF), B_U). \]
    \item For each $k \ge 0 \in U$, the form $\cF_k$ is a classical eigenform, and we have isomorphisms of $E$-linear $G_{\Qp}$-representations
    \[ 
     M_U(\cF) / \varpi_k M_U(\cF) = M_E(\cF_k) \quad\text{and}\quad
     M_U(\cF)^* / \varpi_k M_U(\cF)^* = M_E(\cF_k)^*.
    \]
   \end{itemize}
  \end{theorem}
  
  \begin{proof}
   The finite-slope parts of all the various overconvergent cohomology groups can be glued into coherent sheaves on the eigencurve $\mathscr{C}(N)$. In a neighbourhood of a noble point, the eigencurve is \'etale over weight space and these sheaves are all locally free of rank 2; and the map from $H^1_c$ to $H^1$ is an isomorphism at the noble point, so it must be an isomorphism on some neighbourhood of it. See \cite[Proposition 2.3.5]{Hansen-Iwasawa} for further details.
  \end{proof}
  
 \subsection{Weight one forms}
 
  If $f$ is a cuspidal newform of level $N$ and weight 1, and $f_\alpha$ is a $p$-stabilisation of $f$, then it is \emph{always} the case that $v_p(\alpha) = k_0 + 1 = 0$ and $M_E(f)|_{G_{\Qp}}$ splits as a direct sum (since $M_E(f)$ is an Artin representation). Nonetheless, analogues of Theorem \ref{thm:colemanfamily} and Theorem \ref{thm:colemanrep} do hold for these forms.
  
  \begin{notation}
   We say that $f$ has \emph{real multiplication} by a real quadratic field $K$ if there is a Hecke character $\psi$ of $K$ such that $M_E(f) \cong \operatorname{Ind}_{G_K}^{G_\QQ}(\psi)$.
  \end{notation}
  
  \begin{theorem}
   \label{thm:colemanrepwt1}
   Let $f_\alpha$ be a $p$-stabilisation of a $p$-regular weight 1 eigenform.
   \begin{enumerate}
    \item There is an open disc $U \ni -1$ in $\cW$, a finite flat rigid-analytic covering $\tilde U \rTo^{\kappa} U$ unramified away from $-1$ and totally ramified at $-1$, and a family of eigenforms $\cF \in B_{\tilde U}[[q]]$, whose specialisation at $\kappa^{-1}(-1)$ is $f_\alpha$. We may take $\tilde U = U$ if (and only if) $f$ does not have real multiplication by a quadratic field in which $p$ is split.
    
    \item The module
    \[ M_{\tilde U}(\cF) = \left(\kappa^* M_{U, 0}(\sH_0)\right)\left[ T_n = a_n(\cF) \, \forall n \ge 1\right] \]
    is a direct summand of $\kappa^* M_{U, 0}(\sH_0)$, free of rank 2 as a $B_{\tilde U}$-module, and lifts canonically to $\kappa^* M_{U, 0}(\sH_0)_c$.
    
    \item The same is true of
    \[ M_{\tilde U}(\cF)^* = \left(\kappa^* M_{U, 0}(\sH_0')\right)\left[ T_n' = a_n(\cF) \, \forall n \ge 1\right], \]
    and the pairing $\{-, -\}$ induces an isomorphism $M_{\tilde U}(\cF)^* \cong \Hom_{B_{\tilde U}}(M_{\tilde U}(\cF), B_{\tilde U})$.
   \end{enumerate}
  \end{theorem}
  
  \begin{proof}
   Part (1) is exactly the statement that the eigencurve is smooth at the point corresponding to $f_\alpha$, and is \'etale over weight space except in the real-multiplication setting; see \cite{BellaicheDimitrov}.
   
   Part (2) for compactly supported cohomology is an instance of \cite[Proposition 4.3]{Bellaiche-critical}. However, the kernel and cokernel of the map $ M_{U, 0}(\sH_0)_c \to M_{U, 0}(\sH_0)$ are supported on the Eisenstein component of the eigencurve, and since $f_\alpha$ is a smooth point on the cuspidal eigencurve $\sC^0(N) \subset \sC(N)$, it does not lie on the Eisenstein component. Hence the kernel and cokernel localise to 0 at $f_\alpha$, implying that for small enough $U$ the $\cF$-eigenspaces of $M_U(\sH_0)_c$ and $M_U(\sH_0)$ coincide.
   
   For part (3) we use the fact that the Ohta pairings induce perfect dualities on the ordinary parts of the modules $M_U(\sH_0)_c$ and $M_U(\sH_0')$ (cf.~\cite{Ohta95}).
  \end{proof}
  
  \begin{remark}
   Parts (1) and (2) of Theorem \ref{thm:colemanrepwt1} also hold for non-noble points of weight $\ge 2$ corresponding to the critical $p$-stabilisations of ordinary CM forms, by \cite[Proposition 4.5]{Bellaiche-critical}. However, we do not know if part (3) holds in this situation.
  \end{remark}
  
%%%%%%%%%%%%%%%%%%%%%%%%%%%%%%%%%%%%%

\section{Rankin--Eisenstein classes in Coleman families}

%%%%%%%%%%%%%%%%%%%%%%%%%%%%%%%%%%%%%

 \subsection{Coefficient modules}
  \label{sect:CG}

  Let $H$ be a group isomorphic to $\Zp^2$ (but not necessarily canonically so), for $p$ an odd prime. Then we can regard the modules $\TSym^r H$ as representations of $\operatorname{Aut}(H) \approx \operatorname{GL}_2(\mathbf{Z}_p)$. In this section, we shall show that the Clebsch--Gordan decompositions of the groups $\TSym^r H \otimes \TSym^s H$ can themselves be interpolated as $r$ varies (for fixed $s$), after passing to a suitable completion.

  In this section we shall refer to morphisms as \emph{natural} if they are functorial with respect to automorphisms of $H$.

  \begin{proposition}
   \label{prop:SES1}
   For $A$ an open compact subset of $H$ such that $A \cap pH = \varnothing$, and any $r \ge 1$, there is a short exact sequence
   \[ 0 \rTo C(A) \otimes \Sym^{j-1}(H^\vee) \otimes \wedge^2(H^\vee) \rTo^{\alpha} C(A) \otimes \Sym^j H^\vee \rTo^\beta C(A) \rTo 0\]
   where $C(A)$ is the space of continuous $\cO_E$-valued functions on $A$. This short exact sequence is natural, and split (but not naturally split).
  \end{proposition}

  \begin{proof}
   Let us begin by defining the maps. The map $\beta$, which is the simpler of the two, is given by interpreting $\Sym^j H^\vee$ as a subspace of $C(A)$ (consisting of functions which are the restrictions to $A$ of homogenous polynomial functions on $H$ of degree $j$) and composing with the multiplication map $C(A) \otimes C(A) \to C(A)$.

   The map $\alpha$ is more intricate: it is given by including $\bigwedge^2(H^\vee)$ in $H^\vee \otimes H^\vee$, and grouping the terms as
   \[  (C(A) \otimes H^\vee) \otimes (\Sym^{j-1}(H^\vee) \otimes H^\vee). \]
   As above, we have a canonical multiplication map $C(A) \otimes H^\vee \to H^\vee$, and multiplication in the symmetric algebra $\Sym^\bullet(H^\vee)$ gives a map $\Sym^{j-1}(H^\vee) \otimes H^\vee \to \Sym^j H^\vee$, and this gives the first map in the sequence. The composite $\beta \circ \alpha$ is clearly 0, since it factors through the map $\wedge^2 H^\vee \to \Sym^2 H^\vee$.

   Having defined the maps intrinsically, we may check the exactness of the sequence after fixing a basis of $H$. Let $x$, $y$ be the corresponding coordinate functions, so that $x^j, x^{j-1} y, \dots, y^j$ is a basis of $\Sym^j H^\vee$ and $x \otimes y - y \otimes x$ is a basis of $\wedge^2 H^\vee$. With these identifications we can write the sequence as
   \[ 0 \rTo C(A)^{\oplus j} \rTo C(A)^{\oplus (j+1)} \rTo C(A) \rTo 0\]
   with the maps being $(f_0, \dots, f_{j-1}) \mapsto (-y f_0, x f_0 - y f_1, \dots, x f_{j-1})$ and $(f_0, \dots, f_j) \mapsto x^j f_0 + \dots + y^j f_j$.
   The injectivity of $\alpha$ is now clear, since multiplication by $x$ (or by $y$) is injective in $C(A)$.

   To show that the map $\beta$ is surjective, we write down a (non-canonical) section. We can decompose $A$ as a union $A_1 \sqcup A_2$ where $x$ is invertible on $A_1$ and $y$ is invertible on $A_2$. We define $\delta(f) = (x^{-j} f, 0, \dots, 0)$ on $C(A_1)$ and $\delta(f) = (0, \dots, 0, y^{-j} f)$ on the $C(A_2)$ factor; then $\beta \circ \delta$ is clearly the identity, so $\beta$ is surjective.

   Finally, let $(f_0, \dots, f_j) \in \ker(\beta)$. Choosing $A = A_1 \sqcup A_2$ as before, we may assume either $x$ or $y$ is invertible on $A$. We treat the first case, the second being similar. We define $\gamma(f_1, \dots, f_j) = (g_0, \dots, g_{j-1})$ where $g_{j-1} = x^{-1} f_j$, $g_{j-2} = x^{-2}(x f_{j-1} + y f_j)$ etc, down to $g_0 = x^{-j}(x^{j-1} f_{1} + \dots + y^{j-1} f_j)$. But then $(\alpha \circ \gamma) + (\beta \circ \delta) = \id$, so we have exactness at the middle term.
  \end{proof}

   Now let $C^{\la}(A)$ denote the space of locally analytic $E$-valued functions on $A$; exactly the same argument shows that we have an exact sequence analogous to \eqref{prop:SES1},
   \[  0 \rTo C^{\la}(A) \otimes \Sym^{j-1}(H^\vee) \otimes \wedge^2(H^\vee) \rTo^{\alpha} C^{\la}(A) \otimes \Sym^j H^\vee \rTo^\beta C^{\la}(A) \rTo 0.\]

  \begin{proposition}
   Let $\delta: C^{\la}(A) \to C^{\la}(A) \otimes \Sym^j H^\vee$ be the morphism defined in a basis by
   \begin{equation}
    \label{eq:defdelta}
    \delta(f) = \frac{1}{j!} \sum_{s + t = j} \binom{j}{s} \frac{\partial^j f}{\partial x^s \partial y^t} \otimes x^{s}y^t.
   \end{equation}
   Then $\delta$ is natural, and the composite $\beta \circ \delta$ is the endomorphism of $C^{\la}(A)$ given by $\frac{1}{j!} \prod_{i = 0}^{j-1}(\nabla - i)$, where $\nabla$ is given by
   \[ (\nabla f)(h) = \left.\frac{\mathrm{d}}{\mathrm{d}t} f(t h) \right|_{t = 1}.\]
  \end{proposition}

  \begin{proof}
   The morphism $\delta$ is simply $\tfrac{1}{j!}$ times the $j$-th power of the total derivative map $C^{\la}(A) \to C^{\la}(A) \otimes \operatorname{Tan}(A)^*$, combined with the identification $\operatorname{Tan}(A) \cong \operatorname{Tan}(H) \cong H$. From this description the naturality is clear, and a computation shows that it agrees with the more concrete description above. The identity for $\beta \circ \delta$ is easily seen by induction on $j$.
  \end{proof}

  It will be convenient to adopt the notation $\binom{\nabla}{j}$ for the endomorphism $\tfrac{1}{j!} \prod_{i=0}^{j-1}(\nabla-i)$. We may regard this as an element of the space $D^{\la}(\Zp^\times)$ of locally analytic distributions on $\Zp^\times$.
  
  \begin{proposition}
   For any $k \ge j$, the restriction of $\delta$ to the space $\Sym^k H^\vee$ of homogenous polynomials of degree $k$ lands in the subspace
   \[ \Sym^{k-j} H^\vee \otimes \Sym^j H^\vee \subset C^{\la}(A) \otimes \Sym^j H^\vee,\]
   and the resulting map $\Sym^k H^\vee \to \Sym^{k-j} H^\vee \otimes \Sym^j H^\vee$ is the dual of the symmetrised tensor product map $\TSym^{k-j} H \otimes \TSym^j H \to \TSym^k H$.

   If $k < j$ then the restriction of $\delta$ to $\Sym^k H^\vee$ is the zero map.
  \end{proposition}

  \begin{proof}
   It is obvious that $\Sym^k H^\vee$ embeds naturally into $C^{\la}(A)$, and its image under $\delta$ is contained in $\Sym^{k-j} H^\vee \otimes \Sym^j H^\vee$. A straightforward computation in coordinates shows that this map sends $x^a y^b$ to $\sum_{s + t = j} \binom{a}{s} \binom{b}{t} \left( x^{a-s} y^{b-t} \otimes x^s y^t\right)$, which coincides with the dual of the symmetrised tensor product.

   On the other hand it is obvious from equation \eqref{eq:defdelta} that $\delta$ vanishes on any polynomial of total degree $< j$.
  \end{proof}

  \begin{corollary}
   \label{cor:deltastar}
   There are natural maps
   \[ \delta^*: D^{\la}(A) \otimes \TSym^j(H) \to D^{\la}(A) \]
   and
   \[ \beta^*: D^{\la}(A) \to  D^{\la}(A) \otimes \TSym^j(H), \]
   where $\beta^*$ is given on group elements by $[h] \mapsto [h] \otimes h^{[j]}$, and $\delta^*$ satisfies
   \[ \delta^* \circ \beta^* = \binom{\nabla}{j}.\]
   Moreover, for any $k \ge 0$ we have
   \[
    \mom^k{} \circ \delta^* =
    \begin{cases}
     0 & \text{if $k < j$}, \\
     \mom^{k-j} \cdot 1 & \text{if $k \ge j$},
    \end{cases}
   \]
   where  $\mom^{k-j} \cdot 1$ denotes the composition
   \[ D^{\la}(A) \otimes \TSym^j(H) \rTo^{\mom^{k-j} \otimes 1} \TSym^{k-j} H \otimes \TSym^j H \rTo \TSym^k H\]
   (where the second map is the symmetrized tensor product).
  \end{corollary}

  \begin{proof}
   This follows by dualizing the previous proposition.
  \end{proof}

  We now consider varying $j$, for which it is convenient to re-label the maps $\beta^*, \delta^*$ above as $\beta_j^*$ and $\delta_j^*$.

  \begin{lemma}
   \label{lemma:composedeltas}
   Let $h \ge j \ge 0$. Then the composition
   \begin{align*}
    D^{\la}(A) \otimes \TSym^j(H) \rTo^{\beta_{h-j}^* \otimes \id} & D^{\la}(A) \otimes \TSym^{h-j} H \otimes \TSym^j H \\
     \rTo & D^{\la}(A) \otimes \TSym^h H\\
     \rTo^{\delta_h^*} & D^{\la}(A),
   \end{align*}
   where the unlabelled arrow is given by the symmetrised tensor product,
   is given by
   \[ \binom{\nabla-j}{h-j}\, \delta_j^*.\]
  \end{lemma}

  \begin{proof}
   Explicit computation.
  \end{proof}
  
 \subsection{Nearly-overconvergent \'etale cohomology}

  We also have an analogue of the Clebsch--Gordan map for the distribution spaces $D^\circ_{U, m}(T_0')$ introduced above, which are completions of $D^{\la}(T_0')$. The rigid space $\cW$ has a group structure, so we can make sense of $U - j$ for any integer $j$.

  \begin{proposition}
   \label{prop:AISCG}
   There are natural maps
   \[ \beta^*_j: D^\circ_{U, m}(T_0') \to D^\circ_{U - j, m}(T_0') \otimes \TSym^j H
   \]
   and
   \[ \delta^*_j: D^\circ_{U - j, m}(T_0') \otimes \TSym^j H \to D_{U, m}(T_0'), \]
   commuting with the action of $\Sigma_0(p)$, such that $\delta_j^* \circ \beta_j^*$ is multiplication by $\binom{\nabla}{j} \in \Lambda_U[1/p]$.
  \end{proposition}

  \begin{proof}
   We simply transport the constructions of \S \ref{sect:CG} to the present setting (taking $A = T_0'$). The naturality of these constructions precisely translates into the assertion that the resulting maps commute with the $\Sigma_0(p)$-action. Since the functions in $A_{U, m}$ are homogenous of weight $\kappa_U$ (the canonical character $\Zp^\times \to \Lambda_U^\times$), we have $\frac{\mathrm{d}}{\mathrm{d}t} f(th)|_{t = 1} = \nabla \cdot f(h)$ for all $f \in A_{U, m}$, where on the right-hand side $\nabla$ is regarded as an element of $\Lambda_U[1/p]$; that is, the two actions of $\nabla$ on $A_{U, m}$, as a differential operator and as an element of the coefficient ring, coincide.
  \end{proof}

  \begin{remark}
   Note that $\delta^*_j$ takes values in $D_{U, m} = D^\circ_{U, m}[1/p]$, not in $D^\circ_{U, m}$ itself; the denominator arises from the fact that the map $\delta_j$ on $A^\circ_{U, m}$ does not preserve the $\Lambda_U$-lattice $A^\circ_{U, m}$, but rather maps $A^\circ_{U, m}$ to $\frac{1}{j! p^{1+m}} A^\circ_{U, m}$. Note also that if $U \subset \cW_0$ and $U$ contains none of the integers $\{0, \dots, j-1\}$, then $\binom{\nabla}{j}$ is invertible in $\Lambda_U[1/p]$.
  \end{remark}

  The maps of spaces $\beta_j^*$ and $\delta_j^*$ induce maps of \'etale sheaves on $Y = Y_1(N(p))$ (for any $N$), $\cD^{\circ}_{U, m}(\sH_0') \to \cD^\circ_{U-j, m}(\sH_0') \otimes \TSym^j \sH$ and $\cD^\circ_{U-j, m}(\sH_0') \otimes \TSym^j \sH \to \cD_{U, m}(\sH_0')$, which we denote by the same symbols.

  \begin{definition}
   We shall refer to the cohomology groups $H^*_{\et}(\overline{Y}, \cD_{U-j, m}(\sH_0') \otimes \TSym^j \sH)$ as \emph{nearly-overconvergent \'etale cohomology}, and the map
   \[ \delta_j^*: H^*_{\et}(\overline{Y}, \cD_{U-j, m}(\sH_0') \otimes \TSym^j \sH) \to H^*_{\et}(\overline{Y}, \cD_{U, m}(\sH_0')) \]
   as the \emph{overconvergent projector}.
  \end{definition}

  \begin{remark}
   The motivation for this terminology is that the sheaves $\cD_{U-j, m}(\sH_0') \otimes \TSym^j \sH$, and the maps $\beta_j^*$ and $\delta_j^*$ relating them to the overconvergent cohomology sheaves $\cD_{U, m}(\sH_0')$, are an \'etale analogue of the coherent sheaves appearing in the theory of nearly-overconvergent $p$-adic modular forms (see \cite{Urban-nearly-overconvergent}).
  \end{remark}
  
  Recall from Corollary \ref{cor:deltastar} that the composite of $\delta_j^*$ with the moment map $\rho_k$ is zero if $0 \le k < j$, which is somewhat undesirable. We can rectify this issue as follows. Recall that we have defined $M_U(\sH_0') = H^1_{\et}(\overline{Y}, \cD_{U, m}(\sH_0')(1))$. 
  
  \begin{proposition}
   \label{prop:defprj}
   Let $U$ be an open disc contained in $\cW_0$, and $\cF$ a Coleman family defined over $U$. Suppose the following condition is satisfied: for any integer weight $k \ge 0$ in $U$, the projection map $M_k(\sH'_0) \to M_k(\cF)^*$ factors through $\rho_k$.
   
   Then, for any $j \ge 0$, the composite map
   \[ H^1(\overline{Y}, \cD_{U-j}(\sH_0') \otimes \TSym^j(\sH)(1)) \rTo^{\delta_j^*} M_U(\sH_0') \rTo^{\pr_{\cF}} M_U(\cF)^* \]
   takes values in $\nabla (\nabla - 1) \dots (\nabla - j + 1) M_U(\cF)^*$, and hence the map
   \[ \pr_\cF^{[j]} = \frac{1}{\binom{\nabla}{j}} \pr_{\cF} \circ \mathop{\delta_j^*}: H^1(\overline{Y}, \cD_{U-j}(\sH_0') \otimes \TSym^j(\sH)(1)) \to M_U(\cF)^*\]
   is well-defined.
  \end{proposition}
  
  \begin{proof}
   Note that $\nabla$, regarded as a rigid-analytic function on $\cW$, takes the value $k$ at an integer weight $k$. So the only points in $\cW_0$ at which $\nabla (\nabla - 1) \dots (\nabla - j + 1)$ fails to be invertible are the positive integers $\{0, \dots, j-1\}$, and it has simple zeroes at all of these points.
   
   If $k$ is one of these integers, then we have $M_U(\sH_0') / (\nabla - k) M_U(\sH_0') = M_k(\sH_0')$. Hence it suffices to show that $\pr_{\cF} \circ \delta_j^*$ is zero on $M_k(\sH_0')$; but this is immediate since the specialisation of $\pr_{\cF}$ at $k$ factors through $\rho_k$, and $\rho_k \circ \delta_j^*$ is zero for $0 \le k < j$. 
   
   This shows that $\pr_\cF \circ \delta_j^*$ lands in the stated submodule. Since $M_U(\cF)^*$ is a free $\Lambda_U[1/p]$-module (and $\Lambda_U[1/p]$ is an integral domain), the map $\pr_\cF^{[j]}$ is therefore well-defined.
  \end{proof}
  
  \begin{remark}
   This proposition can be interpreted as follows: we can renormalise $\delta_j^*$ to be an inverse to $\beta_j^*$, as long as we avoid points on the eigencurve which are non-classical but have classical weights.
  \end{remark}
  
  By construction, the map $\pr^{[j]}_{\cF}$ has the property that the following diagram commutes:
  \begin{diagram}
   H^1(\overline{Y}, \cD_{U}(\sH_0')(1))\\
   \dTo^{\beta_j^*} & \rdTo^{\pr_\cF} \\
   H^1(\overline{Y}, \cD_{U-j}(\sH_0') \otimes \TSym^j(\sH)(1)) & 
   \rTo_{\pr_\cF^{[j]}} & M_U(\cF)^*.
  \end{diagram}
  More generally, if $0 \le j \le h$, then (as in Lemma \ref{lemma:composedeltas}) we can consider $\beta_{h - j}^* \cdot \mathrm{id}$ as a map
  \[ \cD_{U-j}(\sH_0') \otimes \TSym^{[j]}(\sH) \to \cD_{U-h}(\sH_0') \otimes \TSym^{[h]}(\sH), \]
  and from Lemma \ref{lemma:composedeltas} one computes that
  \begin{equation}
   \label{eq:composedeltas}
   \pr_{\cF}^{[h]}\ \circ\ (\beta_{h - j}^* \cdot \mathrm{id}) = \binom{h}{j} \pr_{\cF}^{[j]}.
  \end{equation}

%%%%%%%%%%%%%%%%%%%%%%%%%%%%%%%%%%%%%

 \subsection{Two-parameter families of Beilinson--Flach elements}

  Let $N_1, N_2$ be integers such that $p \nmid N_i$ and $p N_1, p N_2 \ge 4$. We also choose two wide open discs $U_1$ and $U_2$ in $\cW_0$, and consider the sheaf
  \[
   \mathcal{D}^\circ_{[U_1, U_2]} \coloneqq \mathcal{D}^\circ_{U_1}(\sH_0') \boxtimes \mathcal{D}^\circ_{U_2}(\sH_0')
  \]
  on the affine surface $Y_1(N_1(p)) \times Y_1(N_2(p))$.

  \begin{definition}
   Let $N$ be any integer divisible by $N_1$ and $N_2$ and with the same prime factors as $N_1 N_2$. For any $j \ge 0$ and $m \ge 1$, we define the element
   \[ \cBF^{[U_1, U_2, j]}_{m, N_1, N_2, a} \in H^3_{\et}\left(Y_1(N_1(p)) \times Y_1(N_2(p)) \times \mu_m^\circ, \mathcal{D}^\circ_{[U_1, U_2]}(2-j)\right)[1/p]\]
   as the image of the class
   \[ \cBF^{[j]}_{m, Np, a} \in H^3_{\et}\left(Y_1(Np)^2 \times \mu_m^\circ, (\Lambda(\sH\langle C \rangle) \otimes \TSym^j \sH)^{\boxtimes 2}(2-j)\right),\]
   under pushforward along $Y_1(Np)^2 \to Y_1(N_1(p)) \times Y_1(N_2(p))$, composed with the map induced by the morphisms of sheaves
   \begin{align*}
    \Lambda(\sH\langle C \rangle) \otimes \TSym^j \sH &\rTo \mathcal{D}^\circ_{U_i-j}(\sH) \otimes  \TSym^j \sH\\
    &\rTo^{\delta^*_j} \mathcal{D}^\circ_{U_i}(\sH)[1/p]
   \end{align*}
   for $i = 1, 2$. Here, the first map is given by the natural maps $\Lambda(\sH\langle C \rangle) \to \cD_{U}^\circ$, for $U = U_i - j$, and the second map is the overconvergent projector $\delta^*_j$ of Proposition \ref{prop:AISCG}.
  \end{definition}

  \begin{remark}
   We are using implicitly here the fact that the Beilinson--Flach elements can be lifted canonically to classes with coefficients in the sheaves $\Lambda(\sH_{\Zp}\langle D' \rangle)$. Cf.~Remark \ref{remark:weasel} above.
  \end{remark}

  The Hochschild--Serre spectral sequence and the K\"unneth formula give a canonical surjection
  \[  H^3_{\et}\left(Y_1(N_1(p)) \times Y_1(N_2(p))\times \mu_m^\circ, \mathcal{D}^\circ_{[U_1, U_2]}(2-j)\right)[1/p] \to H^1_{\et}\left(\ZZ[1/mN_1N_2p, \mu_m], M_{U_1} \htimes M_{U_2}(-j)\right), \]
  and we (abusively) denote the image of $\cBF^{[U_1, U_2, j]}_{m, N_1, N_2, a}$ under this map by the same symbol.

  \begin{proposition}
   For any integer weights $k_1 \in U_1$ and $k_2 \in U_2$ with $\min(k_1, k_2) \ge j$, the map
   \[ \rho_{k_1} \boxtimes \rho_{k_2}: \mathcal{D}_{[U_1, U_2]} \to \TSym^{k_1} \sH \otimes \TSym^{k_2} \sH\]
   sends $\cBF^{[U_1, U_2, j]}_{m, N_1, N_2, a}$ to the pushforward of $\cBF^{[k, k', j]}_{m, Np, a}$. On the other hand, if $0 \le k_1 < j$ or $0 \le k_2 < j$, then the image of $\cBF^{[U_1, U_2, j]}_{m, N_1, N_2, a}$ under $\rho_{k_1} \boxtimes \rho_{k_2}$ is zero.
  \end{proposition}

  \begin{proof}
   This follows from the last statement of \ref{cor:deltastar}, since $\cBF^{[k, k', j]}_{m, Np, a}$ is by definition the image of $\cBF^{[j]}_{m, Np, a}$ under the map $(\mom^{k-j} \cdot \id) \boxtimes (\mom^{k'-j} \cdot \id)$.
  \end{proof}

  Now let us choose newforms $f, g$, of levels $N_1, N_2$ and weights $k_1 + 2, k_2 + 2 \ge 2$, and roots $\alpha_1, \alpha_2$ of their Hecke polynomials, such that the $p$-stabilisations $f_{i, \alpha_i}$ both satisfy the hypotheses of Theorem \ref{thm:colemanrep}. The theorem then gives us families of overconvergent eigenforms $\cF_1$, $\cF_2$ passing through the $p$-stabilisations of $f$ and $g$, defined over some discs $U_1 \ni k_1, U_2 \ni k_2$.
  
  \begin{proposition}
   \label{prop:BFtwovar-coleman}
   If the discs $U_i$ are sufficiently small, then there exist classes
   \[ \cBF^{[\cF, \cG, j]}_{m, a} \in H^1\left(\ZZ\left[\mu_m, \tfrac{1}{mpN_1N_2}\right], 
     M_{U_1}(\cF)^* \htimes M_{U_2}(\cG)^*(-j)\right) 
   \]
   such that
   \[ (\pr_{\cF} \times \pr_{\cG})\left(\cBF^{[U_1, U_2, j]}_{m, a}\right) = \binom{\nabla_1}{j} \binom{\nabla_2}{j} \cBF^{[\cF, \cG, j]}_{m, a},
   \]
   where $\nabla_i$ denotes the image of $\nabla$ in $\Lambda_{U_i}[1/p]$.
  \end{proposition}
  
  \begin{proof}
   After shrinking the discs $U_i$ if necessary so that all integer-weight specialisations of $\cF$ and $\cG$ are classical, so that Proposition \ref{prop:defprj} applies, we can simply define $\cBF^{[\cF, \cG, j]}_{m, a}$ as the image of $\cBF^{[j]}_{m, a}$ under $\pr_{\cF}^{[j]} \times \pr_{\cG}^{[j]}$.
  \end{proof}
  
 \subsection{Interpolation in $j$}
  \label{sect:interpj}
  
  Now let $\cF, \cG$ be Coleman families over open discs $U_1, U_2$, satisfying the conditions of Proposition \ref{prop:BFtwovar-coleman}.
  
  \begin{proposition}
   \label{prop:cyclogrowth}
   For any $h \ge 0$, and any $a$, there is a constant $C$ independent of $r$ such that the elements $\cBF^{[\cF, \cG, j]}_{mp^r, a}$, for $0 \le j \le h$ and $r \ge 1$, satisfy the following norm bound:
   \[ \left\| \sum_{j = 0}^h (-1)^j \binom{h}{j} \Res_{p^r}^{p^{\infty}}\left( \frac{1}{(-a)^j j!} \cdot \cBF^{[\cF, \cG, j]}_{mp^r,a}\right) \right\| \le Cp^{-hr}. \]
  \end{proposition}

  \begin{proof}
   We shall deduce this from Theorem \ref{thm:congruencesfinal} (and Remark \ref{remark:weasel}). This theorem gives a bound for the classes
   \[ \sum_{j = 0}^h a^{h-j}(h-j)! (1 \otimes \mom^{h-j})^{\boxtimes 2}\Res_{p^r}^{p^{\infty}}\left( \cBF^{[j]}_{mp^r, Np, a}\right).\]
   We apply to this the map $\operatorname{pr}_{\cF}^{[h]} \boxtimes \operatorname{pr}_{\cG}^{[h]}$. This maps $(1 \otimes \mom^{h-j})^{\boxtimes 2} \cBF^{[j]}$ to $\binom{h}{j}^2 \cBF^{[\cF, \cG, j]}_{mp^r, a}$, by \eqref{eq:composedeltas}. So the image of the expression of Theorem \ref{thm:congruencesfinal} is
   \[ \sum_{j = 0}^h a^{h-j}(h-j)! \binom{h}{j}^2 \Res_{p^r}^{p^{\infty}}\left(\cBF^{[\cF, \cG, j]}_{mp^r, a}\right),\]
   which is exactly $a^h h!$ times the quantity in the proposition. We may ignore the factor $a^h h!$, since it is nonzero and independent of $r$.
  \end{proof}

  We now choose \emph{affinoid} discs $V_i$ contained in the $U_i$ (so the $M_{V_i}(\cF_i)^*$ become Banach spaces).

  \begin{theorem}
   \label{thm:3varelt}
   There is a element
   \[\cBF^{[\cF, \cG]}_{m, a} \in H^1\left(\QQ(\mu_m), D_{\lambda_1 + \lambda_2}(\Gamma,M_{V_1}(\cF)^* \htimes M_{V_2}(\cG)^*)\right)\]
   which enjoys the following interpolating property: for any integers $(k_1, k_2, j)$ with $k_i \in V_i$ and $0 \le j \le \min(k, k')$, the image of $\cBF^{[\cF, \cG]}_{m, a}$ at $(k_1, k_2, j)$ is
   \[ \left(1 - \frac{p^j}{a_p(\cF_{k_1})a_p(\cG_{k_2})}\right) \frac{\cBF^{[\cF_{k_1}, \cG_{k_2}, j]}_{m, a}}{(-a)^{j} j! \binom{k_1}{j}\binom{k_2}{j}}.\]
  \end{theorem}

  \begin{proof}
   We choose an integer $h \ge \lfloor \lambda_1 + \lambda_2 \rfloor$, and apply Proposition \ref{prop:unbounded-iwasawa} with $K = \QQ(\mu_m)$, $S$ the set of primes dividing $p m N_1 N_2$, $A = \cO(V_1 \times V_2)$, $M = M_{V_1}(\cF)^* \htimes M_{V_2}(\cG)^*$, $\lambda = \lambda_1 + \lambda_2$, and
   \[ x_{n, j} = (a_p(\cF) a_p(\cG))^{-n}  \cdot \frac{\cBF^{[\cF, \cG, j]}_{mp^n, a}}{(-a)^{j} j! } \]
   for $0 \le j \le h$ and $n \ge 1$. These $x_{n, j}$ are norm-compatible for $n \ge 1$, and we obtain norm-compatible elements for all $n \ge 0$ by defining
   \[ x_{0, j}\coloneqq \operatorname{cores}_{\QQ(\mu_p) / \QQ}\left( x_{1, j}\right) = \left(1 - \frac{p^j}{a_p(\cF) a_p(\cG)}\right) \frac{\cBF^{[\cF, \cG, j]}_{m, a}}{(-a)^{j} j!}.\]
   Moreover, the bound we have just established in Proposition \ref{prop:cyclogrowth} shows that $\left\|p^{-nh}\sum_{j = 0}^h (-1)^j \binom{h}{j}x_{n, j}\right\| \le Cp^{\lambda n}$, which is exactly the growth bound required for Proposition \ref{prop:unbounded-iwasawa}. It is not difficult to see that $H^0(\QQ_\infty, M_{V_1}(\cF)^* \htimes M_{V_2}(\cG)^*) = 0$, so we obtain a class 
   \[ x[h] \in H^1\left(\QQ(\mu_m), D_{\lambda_1 + \lambda_2}(\Gamma,M_{V_1}(\cF)^* \htimes M_{V_2}(\cG)^*)\right)\]
   interpolating the classes $x_{n, j}$ for all $n \ge 0$ and all $j \in \{0, \dots, h\}$. However, if we have two integers $h' \ge h \ge \lfloor \lambda_1 + \lambda_2\rfloor$, then the element $x[h']$ satisfies an interpolating property strictly stronger than that of $x[h]$, so we deduce that $x[h]$ is in fact independent of $h$ and interpolates $x_{n, j}$ for all $j \ge 0$. We define $\cBF^{[\cF, \cG]}_{m, a}$ to be this element. The interpolating property is now immediate from the interpolating property of the 2-variable classes $\cBF^{[\cF, \cG, j]}_{m, a}$ at integers $k_1, k_2 \ge j$.
  \end{proof}

%%%%%%%%%%%%%%%%%%%%%%%%%%%%%%%%%%%%%%%%%%%%%%%%%%%%%%
%%%%%%%%%%%%%%%%%%%%%%%%%%%%%%%%%%%%%%%%%%%%%%%%%%%%%%

\section{Phi-Gamma modules and triangulations}

%%%%%%%%%%%%%%%%%%%%%%%%%%%%%%%%%%%%%%%%%%%%%%%%%%%%%%

 \subsection{Phi-Gamma modules in families}

  Let $\sR$ denote\footnote{The rings $\sR$ and $\sR^+$ are often also denoted by $\mathbf{B}^{\dag}_{\mathrm{rig}, \Qp}$ and $\mathbf{B}^+_{\rig, \Qp}$ respectively; this notation is used in several earlier works of the present authors.} the Robba ring (of $\mathbf{Q}_p$), which is the ring of formal Laurent series over $\Qp$ in a variable $\pi$, convergent on some annulus of the form $\{x: 0 < v_p(x) < \varepsilon\} \subseteq \AA^1_{\mathrm{rig}}$; and let $\sR^+ \subseteq \Qp[[\pi]]$ be its subring of elements that are analytic on the whole disc $\{ x: v_p(x) > 0\}$. We endow these with their usual actions of Frobenius $\varphi$ and the group $\Gamma \cong \Zp^\times$. We define a left inverse $\psi$ of $\varphi$ by putting
  \[ \varphi\circ\psi f(\pi)=\frac{1}{p}\sum_{\zeta^p=1}f(\zeta(\pi+1)-1)\]
  for any $f(\pi)\in \sR^+$. 

  As is well known, there is a functor $\DD^{\dag}_{\rig}$ mapping $p$-adic representations of $G_{\Qp}$ to $(\varphi, \Gamma)$-modules over $\sR$ (finitely-generated free $\sR$-modules with commuting $\sR$-semilinear operators $\varphi$ and $\Gamma$), and this is a fully faithful functor whose essential image is the subcategory of $(\varphi, \Gamma)$-modules of slope 0.

  \begin{remark}
   Strictly speaking, the definition of the functor $\DD^{\dag}_{\rig}$ depends on the auxilliary choice of a compatible system of $p$-power roots of unity $(\zeta_{p^n})_{n \ge 0}$ in $\overline{\QQ}_p$. We shall fix, once and for all, such a choice, and in applications to global problems we shall often assume that $\zeta_{p^n}$ corresponds to $e^{2\pi i / p^n} \in \CC$.
  \end{remark}

  Now let $A$ be a reduced affinoid algebra over $\Qp$, and write $\sR_A = \sR \htimes A$ and similarly for $\sR_A^+$. We define an \emph{$A$-representation} of $G_{\Qp}$ to be a finitely-generated locally free $A$-module endowed with an $A$-linear action of $G_{\Qp}$ (continuous with respect to the canonical Banach topology of $M$).

  \begin{theorem}[{Berger--Colmez, \cite{BergerColmez}}]
   For any $A$-representation $M$ of $G_{\Qp}$, we may define a finite locally-free $\sR_A$-module $\DD^{\dag}_{\rig}(V)$, endowed with semilinar continuous actions of $\varphi$ and $\Gamma$, such that
    \[ \DD^{\dag}_{\rig}(V_x) \cong \DD^{\dag}_{\rig}(V) / m_{x} \]
  for every $x \in \operatorname{Max}(A)$.
  \end{theorem}
  
  \begin{definition}
   If $D$ is a $(\varphi, \Gamma)$-module over $\sR_A$, we define cohomology groups $H^i(\Qp, D)$ as the cohomology of the ``Herr complex''
   \[ C_{(\varphi, \gamma)}(D) \coloneqq D \rTo^{\left(\begin{smallmatrix}\varphi - 1\\ \gamma - 1 \end{smallmatrix}\right)} D \oplus D \rTo^{(1-\varphi, \gamma - 1)} D\]
   and \emph{analytic Iwasawa cohomology} $H^i_{\Iw}(\QQ_{p, \infty}, D)$ as the cohomology of the complex
   \[ C_{\psi}(D) \coloneqq 0 \rTo D \rTo^{\psi - 1} D. \]
  \end{definition}
  
  These groups are compatible with the usual Galois cohomology in the sense that if $M$ is a $A$-representation of $G_{\Qp}$, then we have $H^i(\Qp, \DD^\dag_{\rig}(M)) = H^i(\Qp, V)$ and
  \[
   H^i_{\Iw}(\QQ_{p, \infty}, \DD^\dag_{\rig}(M)) = D^{\la}(\Gamma, \Qp) \htimes_{D_0(\Gamma, \Qp)}
   \left( \varprojlim_n H^i(\QQ_{p,n}, T) \right),
  \]
  where $T$ is the unit ball for any $G_{\Qp}$-invariant Banach-module norm on $M$, by 
  by \cite[Corollary 4.4.11]{KPX}.
  
  \begin{corollary}
   For $M$ an $A$-representation of $G_{\Qp}$, we have a canonical isomorphism
   \[ H^1_{\Iw}(\QQ_{p, \infty}, \DD^\dag_{\rig}(M)) \cong H^1(\Qp, D^{\la}(\Gamma, M)).\]
   In particular there is a canonical map
   \[  H^1(\Qp, D_{\lambda}(\Gamma, M)) \to  H^1_{\Iw}(D) \]
   compatible with the natural maps to $H^1(\Qp, M(\eta))$ for every character $\eta$ of $\Gamma$.
  \end{corollary}
  
  \begin{proof}
   Let us choose an increasing sequence of affinoid discs $X_n \subseteq \cW$ whose union is $\cW$. Since we have $D^{\la}(\Gamma, \Qp) = \cO(\cW) = \varprojlim_n \cO(X_n)$, we can regard $D^{\la}(\Gamma, M)$ as a locally free sheaf of $G_{\Qp}$-representations on $\cW \times \operatorname{Max} A$, and we deduce that
   \[ H^1(\Qp, D^{\la}(\Gamma, M)) = \varprojlim_n H^1(\Qp, \cO(X_n) \htimes M), \]
   by \cite[Theorem 1.7]{Pottharst-analytic}. For each $n$, $X_n \times \operatorname{Max} A$ is an affinoid space, so we obtain
   \[ H^1(\Qp, \cO(X_n) \htimes M) = H^1(\Qp, \DD^\dag_{\rig}( \cO(X_n)\htimes M )), \]
   by \cite[Proposition 2.7]{Pottharst-analytic}. Finally, the inverse limit of the modules $\DD^\dag_{\rig}( \cO(X_n)\htimes M )$ is the module $\mathbf{Dfm}(\DD^\dag_{\rig}(M))$ considered in \cite[Theorem 4.4.8]{KPX}, where it is shown that
   \[ H^1(\Qp, \mathbf{Dfm}(D)) = H^1_{\Iw}(\QQ_{p, \infty},D) \]
   for any $(\varphi, \Gamma)$-module $D$ over $\sR_A$.
  \end{proof}
  
  Finally, if the base $A$ is a finite field extension of $\Qp$, then the functors $\DD_{\cris}(-)$ and $\DD_{\dR}(-)$ can be extended from $A$-linear representations of $G_{\Qp}$ to the larger category of $(\varphi, \Gamma)$-modules over $\sR_A$, and one has the following fact:
  
  \begin{theorem}[{Nakamura, see \cite{Nakamura-Iwasawa}}]
   If $A$ is a finite extension of $\Qp$, there exist Bloch--Kato exponential and dual-exponential maps
   \[ \exp_{\Qp, D}: \DD_{\dR}(D) \to H^1(\Qp, D)\]
   and
   \[ \exp^*_{\Qp, D^*(1)}: H^1(\Qp, D) \to \DD_{\dR}(D)\]
   for de Rham $(\varphi, \Gamma)$-modules $D$ over $\sR_A$, which are functorial in $D$ and are compatible with the usual definitions when $D = \DD^{\dag}_{\rig}(V)$ for a de Rham representation $V$.
  \end{theorem}
  
%%%%%%%%%%%%%%%%%%%%%%%%%%%%%%%%%%%%%%%%%%%%%%%%%%%%%  
  
 \subsection{Perrin-Riou logarithms in families}

  Throughout this section, $A$ denotes a reduced affinoid algebra, with supremum norm $\| \cdot \|$, and $\alpha \in A^\times$.

  \begin{definition}
   We write $\sR_A(\alpha^{-1})$ for the free rank 1 $(\varphi, \Gamma)$-module over $\sR_A$ with basis vector $e$ such that $\varphi(e) = \alpha^{-1} e$ and $\gamma e = e$ for all $\gamma \in \Gamma$. We write $\sR^+_A(\alpha^{-1})$ for the submodule $\sR^+_A \cdot e$ of $\sR_A(\alpha^{-1})$.
   \end{definition}

   \begin{lemma}
    Suppose $\|\alpha\| \le 1$ and $\alpha - 1$ is not a zero-divisor in $A$. Then
    \[ \sR_A(\alpha^{-1})^{\psi = 1} \subseteq \sR^+_A(\alpha^{-1}). \]
   \end{lemma}

   \begin{proof}
    This follows from Lemma 1.11 of \cite{Colmez-serie-principale}. Cf.~\cite[\S 4.1]{Hansen-Iwasawa}.
   \end{proof}

   We use this lemma to define a Perrin-Riou big logarithm map for  $\sR_A(\alpha^{-1})$ when $\alpha-1$ is not a zero-divisor, following closely the construction in \cite[\S 4.2]{Hansen-Iwasawa}, as the composition
   \begin{equation}
    \label{eq:BigLog}
    \begin{aligned}
    \cL_{\sR_A(\alpha^{-1})} : H^1_{\Iw}(\QQ_{p, \infty}, \sR_A(\alpha^{-1})) = \sR_A(\alpha^{-1})^{\psi = 1} \rTo^\cong& \sR_A^+(\alpha^{-1})^{\psi = 1}\\
     \rTo^{1 - \varphi}& \sR_A^+(\alpha^{-1})^{\psi = 0}\\
    \rTo^{\mathfrak{M}}& A \htimes \cO(\cW)
   \end{aligned}
   \end{equation}
   where the third arrow is the base-extension to $A$ of the Mellin transform (and $\cW$ is weight space). Note that our assumption that $\alpha - 1$ is not a zero-divisor in $A$ implies that $\sR_A(\alpha^{-1})^{\varphi = 1} = 0$, and hence that $\cL_{\sR_A(\alpha^{-1})}$ is injective.
   
 \subsection{Triangulations}

  \begin{definition}
   Let $D$ be a $(\varphi, \Gamma)$-module over $\sR \htimes A$ which is locally free of rank 2. A \emph{triangulation} of $D$ is a short exact sequence of $(\varphi, \Gamma)$-modules over $\sR \htimes A$,
   \[ 0 \to \sF^+ D \to D \to \sF^- D \to 0, \]
   where the modules $\sF^\pm D$ are locally free of rank 1 over $\sR \htimes A$.
  \end{definition}

  \begin{theorem}[Ruochuan Liu, \cite{Liu-triangulation}]
   \label{thm:triangulation}
   Let $(f, \alpha)$ be as in Theorem \ref{thm:colemanrep}. Then one can find an affinoid disc $V \subset \cW$ containing $k$ such that the $(\varphi, \Gamma)$-module
   \[ D_V(\cF)^* \coloneqq \DD^{\dag}_{\rig}(M_V(\cF)^*)\]
   over $\cO(V)$ admits a canonical triangulation, with $\sF^- D_V(\cF)^* \cong \sR_A(\alpha_{\cF}^{-1})$ and $\sF^+ D_V(\cF)^* \cong \sR_A(\alpha_{\cF} \cdot \varepsilon_{\cF}(p)^{-1})(1 + \kappa_V)$.
  \end{theorem}

 \subsection{Eichler--Shimura isomorphisms}
 
  The last technical ingredient needed to proceed to the proof of our explicit reciprocity law is the following:
  
  \begin{theorem}[Eichler--Shimura relation in families]
   \label{thm:eichlershimura}
   In the setting of Theorem \ref{thm:triangulation}, after possibly shrinking $V$, there is a canonical $\cO(V)$-basis vector
   \[ \omega_{\cF} \in \left(\sF^- D_V(\cF)(1+\kappa_V + \varepsilon_\cF^{(p)})\right)^{\Gamma = 1}\]
   such that for every integer weight $t \ge 0$ in $V$, the specialisation of $\omega_{\cF}$ at $t$ coincides with the image of the differential form $\omega_{f_t}$ attached to the normalised eigenform $f_t$.
  \end{theorem}
  
  This is a minor modification of results of Ruochuan Liu (in preparation); we outline the proof below. The starting point is the following theorem:
  
  \begin{theorem}[Andreatta--Iovita--Stevens, \cite{andreattaiovitastevens}]
   For any integer $k_0 \ge 0$, and real $\lambda < k_0 + 1$, we can find an open disc $V \subset \cW$ containing $k_0$ and a Hecke-equivariant isomorphism
   \[ H^0(X(w), \omega^{\dagger, \kappa_V + 2}_V)^{\le \lambda}\rTo_{\cong}^{\operatorname{comp}_V}\left(H^1_{\et}(Y, \cD^\circ_{V, m}(1))^{\le h} \htimes \CC_p  \right)^{G_{\Qp}}
   \]
   interpolating Faltings' Hodge--Tate comparison isomorphisms for each $k \in V$. Here $X(w)$ is a rigid-analytic neighbourhood of the component of $\infty$ in the ordinary locus of the compactification $X$ of $Y$; and $\omega^{\dagger, \kappa_V + 2}_V$ is a certain sheaf of $\cO(V)$-modules on $X(w)$, whose specialisation at any integer $k \ge 0 \in V$ is the $(k+2)$-th power of the Hodge bundle for every $k \in V$.
  \end{theorem}
  
  \begin{proof}[Proof of Theorem \ref{thm:eichlershimura}]
   We translate the statement of the above theorem into the language of $(\varphi, \Gamma)$-modules. For any family of $G_{\Qp}$-representations $M$ over an affinoid algebra $A$, we have a canonical isomorphism
   \[ \left(M \htimes_{\Qp} \CC_p\right)^{G_{\Qp}} \cong \DD_{\mathrm{Sen}}(M)^{\Gamma}, \]
   where $\DD_{\mathrm{Sen}}(M)$ is defined in terms of the $(\varphi, \Gamma)$-module $\DD^{\dag}_{\mathrm{rig}}(M)$. Moreover, $\DD_{\mathrm{Sen}}\left(\sF^+ D_V(\cF)(1 + \kappa_V)\right)^{\Gamma}$ is zero. Hence, by composing $\operatorname{comp}_V$ with the projection to $\sF^-$, we have an isomorphism
   \[ H^0(X(w), \omega_V^{\kappa_V + 2})[\cF] \rTo^\cong \DD_{\mathrm{Sen}}\left(\sF^- D_V(\cF)(1 + \kappa_V)\right)^{\Gamma}. \]
   The left-hand side is free of rank 1, spanned by $\tau \cdot \cF$ where $\tau$ is the Gauss sum of $\varepsilon_\cF^{(p)}$. On the other hand, since the $(\varphi, \Gamma)$-module $D^- = \sF^- D_V(\cF)(1 + \kappa_V)$ is unramified, we have $\DD_{\mathrm{Sen}}(D^-)^\Gamma = (D^-)^{\Gamma}$.
  \end{proof}
  
  \begin{corollary}
   Under the same hypotheses as Theorem \ref{thm:eichlershimura}, possibly after shrinking $V$ further, there is a $\cO(V)$-basis vector 
   \[ \eta_{\cF} \in \left(\sF^+ D_V(\cF)\right)^{\Gamma = 1}\]
   with the property that for every classical specialisation $\cF_t$ of $\cF$, the specialisation of $\eta_{\cF}$ at $t$ is the unique vector whose cup product with the differential $\omega_{\overline{\cF_t}}$ attached to the complex conjugate $\overline{\cF}_t$ of $\cF_t$ is given by
   \[ \frac{1}{\lambda_N(\cF_t) \cdot \left(1 - \frac{\beta}{\alpha}\right)\left(1 - \frac{\beta}{p\alpha}\right)},\]
   where $\alpha$ and $\beta$ are the roots of the Hecke polynomial of $\cF_t$, and $\lambda_N(\cF_t)$ is its Atkin--Lehner pseudo--eigenvalue.
  \end{corollary}
  
  \begin{proof}
   This follows by dualising $\omega_\cF$ using the Ohta pairing $\{-,-\}$; the computations are exactly the same as in the ordinary case, for which see \cite[Proposition 10.1.2]{KLZ1b}.
  \end{proof}
  
%%%%%%%%%%%%%%%%%%%%%%%%%%%%%%%%%%%%%%%%%%%
%%%%%%%%%%%%%%%%%%%%%%%%%%%%%%%%%%%%%%%%%%%

\section{The explicit reciprocity law}

%%%%%%%%%%%%%%%%%%%%%%%%%%%%%%%%%%%%%%%%%%%%

 \subsection{Regulator maps for Rankin convolutions}
  \label{sect:explicitrecip}
  
  Now let us choose two newforms $f, g$ and $p$-stabilisations $(\alpha_f, \alpha_g)$ satisfying the hypotheses of Theorem \ref{thm:colemanfamily}. 
  
  \begin{notation}
   We write 
   \[ \sF^{--} D_{V_1 \times V_2}(\cF \otimes \cG)^*=\sF^- D_{V_1}(\cF)^*\hat\otimes \sF^- D_{V_2}(\cG)^*,\]
   and similarly for $\sF^{-+}$, $\sF^{+-}$ and $\sF^{++}$. We also define $\sF^{-\circ}  D_{V_1 \times V_2}(\cF \otimes \cG)^*=\sF^- D_{V_1}(\cF)^*\hat\otimes D_{V_2}(\cG)^*$.
  \end{notation}

  \begin{theorem}
   \label{thm:BFeltisSelmer}
   If $V_1$ and $V_2$ are sufficiently small, then (for any $m$ coprime to $p$) the image of $\cBF^{[\cF, \cG]}_{m, a}$ under projection to the module $H^1_{\Iw}(\QQ(\mu_m) \otimes \QQ_{p, \infty}, \sF^{--} D_{V_1 \times V_2}(\cF \otimes \cG)^*)$ is zero.
  \end{theorem}
  
  \begin{proof}
    By taking the $V_i$ sufficiently small, we may assume that $\sF^{--} D_{V_1 \times V_2}(\cF \otimes \cG)^*$ is actually isomorphic to $\sR_A(\alpha^{-1})$, where $\alpha = \alpha_{\cF} \alpha_{\cG}$ and $A = \cO(V_1 \times V_2)$, and that $\|\alpha^{-1}\| < p^{1 + h}$ and $\alpha - 1$ is not a zero-divisor. It suffices, therefore, to show that $\cL_{\sR_A(\alpha^{-1})}$ maps the image of $\cBF^{[\cF, \cG]}_{m, a}$ to zero.
    
    However, for each pair of integers $(\ell, \ell') \in V_1 \times V_2$ with $\ell, \ell' \ge 1 + 2h$ and such that $\cF_\ell$ and $\cG_{\ell'}$ are not twists of each other, we know that the image of $\cL_{\sR_A(\alpha^{-1})}(\cBF^{[\cF, \cG]}_{m, a})$ vanishes when restricted to $(\ell, \ell') \times \cW \subseteq \operatorname{Max}(A) \times \cW$, by Proposition \ref{prop:vanishing}. Since such pairs $(\ell, \ell')$ are Zariski-dense in $\operatorname{Max}(A)$, the result follows.
  \end{proof}
  
  \begin{remark}
   Cf.~\cite[Lemma 8.1.5]{KLZ1b}, which is an analogous (but rather stronger) statement in the ordinary case.
  \end{remark}
  
  Hence the projection of $\cBF^{[\cF, \cG]}_{m, a}$ to $\sF^{- \circ}$ is in the image of the injection
  \[ H^1_{\Iw}(\QQ_{p, \infty}, \sF^{-+} D_{V_1 \times V_2}(\cF \otimes \cG)^*) \to  H^1_{\Iw}(\QQ_{p, \infty}, \sF^{-\circ} D_{V_1 \times V_2}(\cF \otimes \cG)^*).\]
  Since $\sF^+ D_{V_2}(\cG)^*$ is isomorphic to an unramified module twisted by an $A^\times$-valued character of the cyclotomic Galois group $\Gamma$, we may define a Perrin-Riou logarithm map for $\sF^{-+} D_{V_1 \times V_2}(\cF \otimes \cG)^*$ by reparametrising the corresponding map for its unramified twist, exactly as in Theorem 8.2.8 of \cite{KLZ1b}. That is, if we define
  \[ \DD(\sF^{-+} M(\cF \otimes \cG)^*) = \left(\sF^{-+} D(\cF \otimes \cG)^*(-1-\kappa_{V_2})\right)^{\Gamma = 1},\]
  which is free of rank 1 over $\cO(V_1 \times V_2)$, then we obtain the following theorem:
  
  \begin{theorem}
   There is an injective morphism of $\cO(V_1 \times V_2 \times W)$-modules
   \[
    \cL: H^1_{\Iw}(\QQ_{p, \infty}, \sF^{-+} D_{V_1 \times V_2}(\cF \otimes \cG)^*)
    \to \DD(\sF^{-+} M(\cF \otimes \cG)^*) \htimes \cO(\cW),
   \]
   with the following property: for all classical specialisations $f, g$ of $\cF,\cG$, and all characters of $\Gamma$ of the form $\tau = j + \eta$ with $\eta$ of finite order and $j \in \ZZ$, we have a commutative diagram
   \begin{diagram}
    H^1_{\Iw}\left(\QQ_{p, \infty}, \sF^{-+} D_{V_1 \times V_2}(\cF \otimes \cG)^*\right) & \rTo^{\cL} & \DD(\sF^{-+} M(\cF \otimes \cG)^*) \htimes \cO(\cW) \\
    \dTo & & \dTo\\
    H^1(\Qp, \sF^{-+} D(f \otimes g)^*(-j-\eta)) & \rTo & \sF^{-+}\DD_{\mathrm{cris}}( M(f \otimes g)^*(-\varepsilon_{g, p}))
   \end{diagram}
   in which the bottom horizontal map is given by
   \[
    \left.\begin{cases}
    \left(1 - \frac{p^j}{\alpha_f \beta_g}\right)\left(1 - \frac{\alpha_f \beta_g}{p^{1 + j}}\right)^{-1} & \text{if $r = 0$} \\
    \left(\frac{p^{1 + j}}{\alpha_f \beta_g}\right)^{r} G(\varepsilon)^{-1} & \text{if $r > 0$}
   \end{cases}\right\}
   \cdot
   \begin{cases}
    \tfrac{(-1)^{k'-j}}{(k'-j)!} \log & \text{if $j \le k'$,} \\ (j-k'-1)! \exp^* & \text{if $j > k'$,}
    \end{cases}
   \]
   where $\exp^*$ and $\log$ are the Bloch--Kato dual-exponential and logarithm maps, $\varepsilon$ is the finite-order character $\varepsilon_{g, p} \cdot \eta^{-1}$ of $\Gamma$, $r \ge 0$ is the conductor of $\varepsilon$, and $G(\varepsilon) = \sum_{a \in (\ZZ / p^r \ZZ)^\times} \varepsilon(a) \zeta_{p^r}^a$ is the Gauss sum.
  \end{theorem}
  
  \begin{proof}
   The construction of the map $\cL$ is immediate from \eqref{eq:BigLog}. The content of the theorem is that the map $\cL$ recovers the maps $\exp^*$ and $\log$ for the specialisations of $\cF$ and $\cG$; this follows from Nakamura's construction of $\exp^*$ and $\log$ for $(\varphi, \Gamma)$-modules.
  \end{proof}

  \begin{theorem}[Explicit reciprocity law]
   \label{thm:explicitrecip}
   If the $V_i$ are sufficiently small, then we have
   \[ \left\langle \cL\left(\cBF^{[\cF, \cG]}_{1, 1}\right), \eta_{\cF} \otimes \omega_{\cG} \right\rangle = (c^2 - c^{-(\mathbf{k} + \mathbf{k}' - 2\mathbf{j})} \varepsilon_\cF(c)^{-1} \varepsilon_\cG(c)^{-1})(-1)^{1+\mathbf{j}} \lambda_N(\cF)^{-1} L_p(\cF, \cG, 1 + \mathbf{j}).\]
   Here, $L_p(\cF,\cG,1 + \mathbf{j})$ denotes Urban's $3$-variable $p$-adic $L$-function as constructed in \cite{Urban-nearly-overconvergent}, and $\varepsilon_\cF$ and $\varepsilon_\cG$ are the characters by which the prime-to-$p$ diamond operators act on $\cF$ and $\cG$.
  \end{theorem}
  
  \begin{proof}
   The two sides of the desired formula agree at every $(k, k', j)$ with $k \in V_1$, $k' \in V_2$ and $0 \le j \le \min(k, k')$, by \cite[Theorem 6.5.9]{KLZ1a}. These points are manifestly Zariski-dense, and the result follows.
  \end{proof}
  
  \begin{remark}
   The construction of $\omega_{\cG}$, and the proof of the explicit reciprocity law, are also valid if $\cG$ is a Coleman family passing through a $p$-stabilisation $g_\alpha$ of a $p$-regular weight 1 form, as in Theorem \ref{thm:colemanrepwt1}; the only difference is that one may need to replace $V_2$ with a finite flat covering $\tilde V_2$. In this setting, $g_\alpha$ is automatically ordinary, so $\cG$ is in fact a Hida family, and one can use the construction of $\omega_{\cG}$ given in \cite[Proposition 10.12.2]{KLZ1b}.
  \end{remark}

%%%%%%%%%%%%%%%%%%%%%%%%%%%%%%%%%%%%%%%%%%%%%%%%%%%%%%
%%%%%%%%%%%%%%%%%%%%%%%%%%%%%%%%%%%%%%%%%%%%%%%%%%%%%%

\section{Bounding Selmer groups}

%%%%%%%%%%%%%%%%%%%%%%%%%%%%%%%%%%%%%%%%%%%%%%%%%%%%%%

 \subsection{Notation and hypotheses}
 
  Let $f, g$ be cuspidal modular newforms of weights $k + 2, k' + 2$ respectively, and levels $N_f, N_g$ prime to $p$. We \emph{do} permit here the case $k' = -1$. We suppose, however, that $k > k'$, so in particular $k \ge 0$; and we choose an integer $j$ such that $k' + 1 \le j \le k$. If $j = \frac{k + k'}{2} + 1$, then we assume that $\varepsilon_f \varepsilon_g$ is not trivial, where $\varepsilon_f$ and $\varepsilon_g$ are the characters of $f$ and $g$.
  
  As usual, we let $E$ be a finite extension of $\Qp$ with ring of integers $\cO$, containing the coefficients of $f$ and $g$. Our goal will be to bound the Selmer group associated to the Galois representation $M_\cO(f \otimes g)(1 + j)$, in terms of the $L$-value $L(f, g, 1 + j)$; our hypotheses on $(k, k', j)$ are precisely those required to ensure that this $L$-value is a \emph{critical} value.
  
  It will be convenient to impose the following local assumptions at $p$:
  
  \begin{itemize}
   \item ($p$-regularity) We have $\alpha_f \ne \beta_f$ and $\alpha_g \ne \beta_g$, where $\alpha_f, \beta_f$ are the roots of the Hecke polynomial of $f$ at $p$, and similarly for $g$.
   
   \item (no local zero) None of the pairwise products
   \[ 
    \{ \alpha_f \alpha_g, \alpha_f \beta_g, \beta_f \alpha_g, \beta_f \beta_g\}
   \]
   is equal to $p^j$ or $p^{1 + j}$, so the Euler factor of $L(f, g, s)$ at $p$ does not vanish at $s = j$ or $s = 1 + j$.
   
   \item (nobility of $f_\alpha$) If $f$ is ordinary, then either $\alpha_f$ is the unit root of the Hecke polynomial, or $M_E(f) |_{G_{\Qp}}$ is not the direct sum of two characters (so the eigenform $f_\alpha$ is noble in the sense of \ref{def:noble}).
   
   \item (nobility of $g_\alpha$ and $g_\beta$) If $k' \ge 0$, then $M_E(g)|_{G_{\Qp}}$ does not split as a direct sum of characters, so both $p$-stabilisations $g_\alpha$ and $g_\beta$ are noble.
  \end{itemize}
  
  \begin{remark}\mbox{~}
   \begin{enumerate}
    \item In our arguments we will use both $p$-stabilisations $g_\alpha$ and $g_\beta$ of $g$, but only the one $p$-stabilisation $f_\alpha$ of $f$; in particular, we do not require that the other $p$-stabilisation $f_\beta$ be noble.
    
    \item Note that the ``no local zero'' hypothesis is automatic, for weight reasons, unless $k + k'$ is even and $j = \frac{k + k'}{2}$ or $j = \frac{k + k'}{2} + 1$ (so the  $L$-value $L(f, g, 1 + j)$ is a ``near-central'' value).
   \end{enumerate}
  \end{remark} 
  
  The $p$-regularity hypothesis implies that we have direct sum decompositions
  \[ \DD_\cris(M_E(f)^*) = \DD_{\cris}(M_E(f)^*)^{\alpha_f} \oplus \DD_{\cris}(M_E(f)^*)^{\beta_f}\]
  where $\varphi$ acts on the two direct summands as multiplication by $\alpha_f^{-1}$, $\beta_f^{-1}$ respectively, and similarly for $g$. This induces a decomposition of $\DD_\cris(M_E(f \otimes g)^*)$ into four direct summands $\DD_{\cris}(M_E(f \otimes g)^*)^{\alpha_f \alpha_g}$ etc.
  
  \begin{definition}
   We write
   \begin{align*}
    \DD_{\cris}(M_E(f \otimes g)^*)^{\alpha_f \circ} &= \DD_{\cris}(M_E(f \otimes g)^*)^{\alpha_f \alpha_g} \oplus \DD_{\cris}(M_E(f \otimes g)^*)^{\alpha_f \beta_g} \\&= \DD_{\cris}(M_E(f)^*)^{\alpha_f} \otimes_E \DD_{\cris}(M_E(g)^*).
   \end{align*}
   We write $\pr_{\alpha_f}$ for the projection
   \[ \DD_{\cris}(M_E(f \otimes g)^*) \to \DD_{\cris}(M_E(f \otimes g)^*)^{\alpha_f \circ}\]
   with $\DD_{\cris}(M_E(f \otimes g)^*)^{\beta_f \circ}$ as kernel.
  \end{definition}
  
  \begin{proposition}
   If $W$ denotes the Galois representation $M_E(f \otimes g)^*(-j)$, then:
   \begin{itemize}
    \item $H^1(\Qp, W)$ is 4-dimensional (as an $E$-vector space), and $H^0(\Qp, W) = H^2(\Qp, W) = 0$;
    \item we have
    \[ H^1_\mathrm{e}(\Qp, W) = H^1_\mathrm{f}(\Qp, W) = H^1_\mathrm{g}(\Qp, W),\]
    and this space has dimension 2;
    \item the dual exponential map gives an isomorphism
    \[ \frac{H^1(\Qp, W)}{H^1_\mathrm{f}(\Qp, W)} \rTo^\cong \Fil^0 \DD_\cris(W);\]
    \item the projection
    \[ \Fil^0 \DD_\cris(W) \rTo^{\pr_{\alpha_f}}  \DD_{\cris}(W)^{\alpha_f \circ}\]
    is an isomorphism.
   \end{itemize}
  \end{proposition}
  
  \begin{proof}
   This is an elementary exercise using local Tate duality, Tate's local Euler characteristic formula, and the ``no local zero'' hypothesis.
  \end{proof}

  \begin{theorem}
   \label{thm:ESexists}
   Fix some $c > 1$ coprime to $6p N_f N_g$. For each $m \ge 1$ coprime to $pc$, we obtain two classes
   \[ c^{\alpha_f \alpha_g}_m, c^{\alpha_f \beta_g}_m \in H^1(\QQ(\mu_m), M_E(f \otimes g)^*(-j)), \]
   with the following properties:
   \begin{enumerate}[(i)]
    \item for every prime $v \nmid p$ of $\QQ(\mu_m)$, we have
    \[ \operatorname{loc}_v\left(c^{\alpha_f \alpha_g}_m\right) \in H^1_\mathrm{f}(\QQ(\mu_m)_v, M_E(f \otimes g)^*(-j));\]
    \item there is a constant $R$ (independent of $m$) such that 
    \[ R c^{\alpha_f \alpha_g}_m, Rc^{\alpha_f \beta_g}_m \in H^1(\QQ(\mu_m), M_{\cO_E}(f \otimes g)^*(-j)) / \{ \text{torsion}\},\]
    where $M_{\cO_E}(f \otimes g)^*$ is the lattice in $M_E(f \otimes g)^*$ which is the image of the \'etale cohomology with $\cO_E$-coefficients;
    \item for $\ell \nmid m N_f N_g$, we have
    \[ \norm_{m}^{\ell m} \left(c^{\alpha_f \alpha_g}_{\ell m}\right) = P_\ell(\ell^{-1-j} \sigma_\ell^{-1}) \cdot c^{\alpha_f \alpha_g}_{m},\]
    where $P_\ell(X)$ is the local Euler factor of $L(f, g, s)$ at $\ell$, and similarly for $c^{\alpha_f \beta_g}_{\ell m}$;
    \item the images of $c^{\alpha_f \alpha_g}_m$ and $c^{\alpha_f \beta_g}_m$ under the map
    \begin{align*}
     H^1(\QQ(\mu_m) \otimes \Qp, M_E(f \otimes g)^*(-j)) \rTo^{\exp^*}& \QQ(\mu_m) \otimes_{\QQ} \Fil^0 \DD_{\cris}(M_{E}(f \otimes g)^*(-j))\\
     \rTo^{\pr_{\alpha_f}}& \QQ(\mu_m) \otimes_{\QQ} \DD_{\cris}(M_E(f \otimes g)^*)^{\alpha_f \circ}
    \end{align*}
    lie in the subspaces $\DD_{\cris}(M_{E}(f)^*)^{\alpha_f \beta_g}$ and $\DD_{\cris}(M_{E}(f)^*)^{\alpha_f \alpha_g}$ respectively;
   
    \item for $m = 1$, the projections $\pr_{\alpha}\left( \exp^* c^{\alpha_f \alpha_g}_1\right)$ and $\pr_{\alpha}\left( \exp^* c^{\alpha_f \beta_g}_1\right)$ are non-zero (for some suitable choice of $c$) if and only if $L(f \otimes g, 1 + j) \ne 0$.
   \end{enumerate}
  \end{theorem}
  
  \begin{proof}
   We define the class $c^{\alpha_f \alpha_g}$ as follows. Using the $p$-stabilisations $f_\alpha$ of $f$ and $g_\alpha$ of $g$, Theorem \ref{thm:3varelt} gives rise to elements 
   \[ 
    \cBF^{[\cF, \cG]}_{m, 1} \in H^1(\QQ(\mu_m), D_\lambda(\Gamma, M_{V_1 \times V_2}(\cF \otimes \cG)^*)
   \]
   where $\cF$ and $\cG$ are Coleman families through $f_\alpha$ and $g_\alpha$ (which exist, since $f_\alpha$ is noble, and $g_\alpha$ is either noble of weight $\ge 2$ or $p$-regular of weight 1). Specialising these at $(f_\alpha, g_\alpha, j)$, and identifying $M_E(f_\alpha \otimes g_\alpha)^*$ with $M_E(f \otimes g)^*$ via the maps $\Pr^{\alpha_f}$ and $\Pr^{\alpha_g}$, we obtain classes $z_m^{\alpha_f \alpha_g} \in H^1(\QQ(\mu_m), M_{E}(f \otimes g)^*(-j))$.
   
   These classes satisfy (i), by Proposition \ref{prop:iwacoho-unramified}. They also satisfy (ii), by Proposition \ref{prop:ming-lun-lemma} (using the fact that $f$ and $g$ have differing weights, by hypothesis, so we have $H^0(\QQ^{\mathrm{ab}}, M_E(f_\alpha \otimes g_\alpha)^*) = 0$).
   
   The classes  $z_m^{\alpha_f \alpha_g}$ do \emph{not} satisfy (iii); instead, they satisfy the a slightly more complicated norm-compatibility relation $\norm_{m}^{\ell m} \left(c^{\alpha_f \alpha_g}_{\ell m}\right) = Q_\ell(\ell^{-1-j} \sigma_\ell^{-1}) c^{\alpha_f \alpha_g}_{m}$ where $Q_\ell(X) \in X^{-1} \cO_L[X]$ is a polynomial congruent to $-X^{-1} P_\ell(X)$ modulo $\ell - 1$. However, the ``correct'' Euler system relation can be obtained by modifying each class $z^{\alpha_f \alpha_g}_m$ by an appropriate element of $\cO_L[(\ZZ / m\ZZ)^\times]$, as in \cite[\S 7.3]{LLZ14}. This gives classes $c_m^{\alpha_f \alpha_g}$ satisfying (i)--(iv).
   
   It remains to verify (iv) and (v). It suffices to prove these for the un-modified classes $z_m^{\alpha_f \alpha_g}$. For (iv), let $K$ denote the completion of $\QQ(\mu_m)$ at a prime above $p$, and $K_\infty = K(\mu_{p^\infty})$. Then we have a diagram 
   \begin{diagram}
    D^{\la}(\Gamma, \Qp) \otimes_{D_0(\Gamma, \Qp)} H^1_{\Iw}(K_\infty, V) & \rTo & H^1_{\Iw}(K_\infty, \sF^{--} D )\\
    \dTo & & \dTo \\
    H^1(K, W) & \rTo & H^1(K, \sF^{--} D)\\
    \dTo^{\exp^*_{K, W^*(1)}} & & \dTo_{\exp^*_{K, (\sF^{--} D)^*(1)}} \\
    \DD_{\cris, K}(W) & \rTo & \DD_{\cris, K}(\sF^{--} D).
   \end{diagram}
   Here $W$ denotes the Galois representation $M_E(f \otimes g)^*(-j)$, as above, $D$ denotes $\DD^\dag_{\rig}(V)$, and $\sF^{--} D$ is the quotient of $D$ (in the category of $(\varphi, \Gamma)$-modules) determined by the triangulations of $M_E(f_\alpha)^*$ and $M_E(f_\beta)^*$. Note that this quotient depends on the choice of $\alpha_f$ and $\alpha_g$, although the Galois representation $W$ does not.
   
   The horizontal arrows in the diagram are induced by the morphism of $(\varphi, \Gamma)$-modules $D \to \sF^{--} D$. We know that the image of $\cBF^{[f_\alpha, g_\alpha]}_{m, 1}$ in $H^1_{\Iw}(K_\infty, \sF^{--} D )$ is zero, by Theorem \ref{thm:BFeltisSelmer}; so its image in the bottom right-hand corner is zero. However, the projection $\DD_\cris(V) \to \DD_\cris(\sF^{--} D)$ factors through projection to the eigenspace $\DD_\cris(V)^{\alpha_f \alpha_g}$, and is an isomorphism on this eigenspace; so we recover the statement that $\exp^*(c^{\alpha_f \alpha_g}_m)$ projects to zero in $\DD_\cris(V)^{\alpha_f \alpha_g}$, as required.
   
   Finally, we prove (v). For this, we use an analogous commutative diagram with $\sF^{-\circ}$ in place of $\sF^{--}$:
   \begin{diagram}
    D^{\la}(\Gamma, \Qp) \otimes_{D_0(\Gamma, \Qp)} H^1_{\Iw}(\QQ_{p,\infty}, V) & \rTo & H^1_{\Iw}(\QQ_\infty, \sF^{-\circ} D ) & \lInto & H^1_{\Iw}(\QQ_{p,\infty}, \sF^{-+} D )\\
    \dTo & & \dTo & & \dTo \\
    H^1(\Qp, V) & \rTo & H^1(\Qp, \sF^{-\circ} D) & \lTo & H^1(\Qp, \sF^{-+} D)\\
    \dTo^{\exp^*_{\Qp, V^*(1)}} & & \dTo^{\exp^*_{\Qp, (\sF^{-\circ} D)^*(1)}}& & \dTo_{\exp^*_{\Qp, (\sF^{-+} D)^*(1)}} \\
    \DD_{\cris}(V) & \rTo & \DD_{\cris}(\sF^{-\circ} D) & \lInto & \DD_{\cris}(\sF^{-+} D).
   \end{diagram}
   The projection $\DD_\cris(V) \to \DD_{\cris}(\sF^{-\circ}D)$ induces an isomorphism
   \[ \DD_\cris(V)^{\alpha_f \circ} \rTo^\cong \DD_{\cris}(\sF^{-\circ}D). \]
   Theorem \ref{thm:BFeltisSelmer} implies that the image of $\cBF^{[f_\alpha, g_\alpha]}_1$ in $H^1_{\Iw}(\QQ_{p,\infty}, V)$ lies in the image of $H^1_{\Iw}(\QQ_{p,\infty}, \sF^{-+} D )$; and the explicit reciprocity law shows that the image of this class in $\DD_{\cris}(\sF^{-+} D)$ is non-zero if and only if $(c^2 - c^{2j-k-k'} \varepsilon_f(c) \varepsilon_g(c)) L(f, g, 1 + j) \ne 0$. Our hypothesis that $ \varepsilon_f(c) \varepsilon_g(c)$ be nontrivial if $j = \frac{k + k'}{2} + 1$ shows that we can choose $c$ such that the first factor is nonzero. So, for a suitable choice of $c$, the projection of $z_1^{\alpha_f \alpha_g}$ to $\DD_{\cris}(V)^{\alpha_f \beta_g}$ is non-zero if and only if $L(f, g, 1 + j) \ne 0$.
   
   This completes the construction of classes $c^{\alpha_f \alpha_g}_m$ with the required properties. The construction of $c^{\alpha_f \beta_g}_m$ is identical, using the $p$-stabilisation $g_\beta$ in place of $g_\alpha$.
  \end{proof}
  
 \subsection{Bounding the Bloch--Kato Selmer group}
  
  Recall that if $V$ is a geometric $p$-adic representation of $\Gal(\overline{\QQ} / \QQ)$, then we define
  \[ 
   H^1_{\mathrm{f}}(\QQ, V) = \{ x \in H^1(\QQ, V) : \operatorname{loc}_\ell(x) \in H^1_{\mathrm{f}}(\QQ_\ell, V) \text{ for all finite primes $\ell$} \}.
  \]
 
  \begin{theorem}
   \label{thm:BKSel}
   Suppose the assumptions of Theorem \ref{thm:ESexists} are satisfied, and in addition the following hypothesis is satisfied:
   \begin{itemize}
    \item (big image) There exists an element $\tau \in \Gal(\overline{\QQ} / \QQ(\mu_{p^\infty}))$ such that $V / (\tau - 1) V$ is 1-dimensional, where $V = M_E(f \otimes g)(1 + j)$.
   \end{itemize}
   If $L(f, g, 1 + j) \ne 0$, then the Bloch--Kato Selmer group $H^1_{\mathrm{f}}(\QQ, V)$ is zero.
  \end{theorem}
  
  \begin{remark}
   It is shown in \cite{Loeffler-big-image} that, under fairly mild hypotheses on $f$ and $g$, the the ``big image'' hypothesis is satisfied for all but finitely many primes $\frP$ of the coefficient field.
  \end{remark}
  
  \begin{proof}
   Let $H^1_{\mathrm{strict}}(\QQ, V)$ denote the \emph{strict Selmer group}, which is the kernel of the localisation map
   \[ \loc_p: H^1_{\mathrm{f}}(\QQ, V) \to H^1_{\mathrm{f}}(\Qp, V).\]
   
   Let $T$ be a lattice in $V$, and let $A = V / T$. By Theorem \ref{thm:ESexists}, for some nonzero scalar $R$, the classes $R \cdot c_m^{\alpha_f \alpha_g}$ form a non-zero Euler system for $T^*(1)$ in the sense of \cite[Definition 2.1.1]{Rubin-Euler-systems}, if we replace condition (ii) in the definition by the alternative condition (ii')(b) of \S 9.1 of \emph{op.cit.}.
   
   By \cite[Theorem 2.2.3]{Rubin-Euler-systems}, the existence of any non-zero Euler system for $V^*(1)$, together with the ``big image'' hypothesis, implies that the $p$-torsion Selmer group
   \[ H^1_{\mathrm{strict}}(\QQ, A) \coloneqq \operatorname{ker}\left( H^1(\QQ, A) \to \bigoplus_\ell \frac{H^1(\QQ_\ell, A)}{H^1_{\mathrm{strict}}(\QQ_\ell, A)}\right) \]
   is finite, where $H^1_{\mathrm{strict}}(\QQ_\ell, A)$ is defined as the image of the map
   \[ H^1_{\mathrm{f}}(\QQ_\ell, V) \to H^1(\QQ_\ell, A)\]
   for $\ell \ne p$, and $H^1_{\mathrm{strict}}(\Qp, A) = 0$. However, the image of $H^1_{\mathrm{strict}}(\QQ, V)$ in $H^1(\QQ, A)$ is clearly contained in $H^1_{\mathrm{strict}}(\QQ, A)$; so we conclude that $H^1_{\mathrm{strict}}(\QQ, V)$ is zero.
   
   In order to refine this, we use Poitou--Tate duality. Let $H^1_{\mathrm{relaxed}}(\QQ, V^*(1))$ (the ``relaxed Selmer group'') denote the classes in $H^1(\QQ, V^*(1))$ whose localisation lies in $H^1_\mathrm{f}$ for all $\ell \ne p$ (but may be arbitrary at $p$). Then we have two exact sequences
   \[ 0 \rTo H^1_{\mathrm{strict}}(\Qp, V) \rTo H^1_{\mathrm{f}}(\Qp, V)\rTo^{\loc_p^{\mathrm{f}}} H^1_{\mathrm{f}}(\Qp, V) \]
   and
   \[ 0 \rTo H^1_{\mathrm{f}}(\Qp, V^*(1)) \rTo H^1_{\mathrm{relaxed}}(\Qp, V^*(1)) \rTo^{\loc_p^{\mathrm{s}}} H^1_{\mathrm{s}}(\Qp, V^*(1)),\]
   where $H^1_{\mathrm{s}}(\Qp, V^*(1)) = \frac{H^1(\Qp, V^*(1))}{H^1_{\mathrm{f}}(\Qp, V^*(1))}$ (the ``singular quotient''). 
   Local Tate duality identifies $H^1_{\mathrm{s}}(\Qp, V^*(1))$ with the dual of $H^1_{\mathrm{f}}(\Qp, V)$; and the Poitou--Tate global duality exact sequence implies that the images of $\loc^{\mathrm{f}}_p$ and $\loc^{\mathrm{s}}_p$ are orthogonal complements of each other; compare \cite[Theorem 1.7.3]{Rubin-Euler-systems}.
   
   We have constructed two classes in $H^1_{\mathrm{relaxed}}(\Qp, V^*(1))$, namely $c_1^{\alpha_f \alpha_g}$ and $c_1^{\alpha_f \beta_g}$, whose images in $\frac{H^1(\Qp, V^*(1))}{H^1_{\mathrm{f}}(\Qp, V^*(1))}$ are linearly independent (since their images under $\exp^*$ span distinct eigenspaces). So $\loc_p^{\mathrm{s}}$ is surjective, and consequently $\loc^{\mathrm{f}}_p$ is the zero map. As we have already shown that $H^1_{\mathrm{strict}}(\Qp, V) = 0$, this shows that $H^1_{\mathrm{f}}(\Qp, V)$ is zero.
  \end{proof}
  
  \begin{remark}
   The above argument is an adaptation of the ideas of \cite[\S 6.2]{DR-diagonal-cycles-II}, in which Poitou--Tate duality is used to bound the image of the map $\loc^\mathrm{f}_p$ for a Galois representation arising from the product of three cusp forms. In our setting, since we have a full Euler system rather than just the two classes $c_1^{\alpha_f \alpha_g}$ and $c_1^{\alpha_f \beta_g}$, we can also bound the kernel of this map.
  \end{remark}

 \subsection{Corollaries}
 
  From Theorem \ref{thm:BKSel} one obtains a rather precise description of the global cohomology groups. We continue to write $V = M_E(f \otimes g)(1 + j)$. 
  
  Let $S$ be any finite set of places of $\QQ$, containing $\infty$ and all primes dividing $p N_f N_g$. Then the action of $\Gal(\overline{\QQ} / \QQ)$ on $V$ factors through $\Gal(\QsQ)$, the Galois group of the maximal extension of $\QQ$ unramified outside $S$. Since the Bloch--Kato local condition coincides with the unramified condition for $\ell \notin S$, we have
  \[ H^1_\mathrm{f}(\QQ, V) = \{ x \in H^1(\QsQ, V) : \loc_\ell(x) \in H^1_{\mathrm{f}}(\QQ_\ell, V) \text{ for all $\ell \in S$}\}.\]
  
  \begin{remark}
   Since $\Gal(\QsQ)$ is the \'etale fundamental group of $\ZZ[1/S]$, we may interpret any continuous $\Qp$-linear representation of $\Gal(\QsQ)$ as a $p$-adic \'etale sheaf on $\operatorname{Spec} \ZZ[1/S]$, and the continuous cohomology groups $H^i(\QsQ, -)$ coincide with the \'etale cohomology groups $H^i_{\et}(\ZZ[1/S], -)$. The latter language is used in \cite{KLZ1b} for instance; but in the present work we have found it easier to use the language of group cohomology, since this makes the arguments of \S \ref{sect:analyticprelim} easier to state.
  \end{remark}
     
  \begin{corollary}
   \label{cor:ptduality}
   If the hypotheses of Theorem \ref{thm:BKSel} hold, then:
   \begin{enumerate}
    \item The localisation maps
    \begin{align*}
      H^2(\QsQ, V) &\to \bigoplus_{\ell \in S} H^2(\QQ_\ell, V)\quad \text{and}\\
      H^2(\QsQ, V^*(1)) &\to \bigoplus_{\ell \in S} H^2(\QQ_\ell, V^*(1))
    \end{align*}      
    are isomorphisms.
    \item The space $H^1_{\mathrm{f}}(\QQ, V^*(1))$ is zero.
    \item The space $H^1_{\mathrm{relaxed}}(\QQ, V^*(1))$ is 2-dimensional, and $c_1^{\alpha_f \alpha_g}$ and $c_1^{\alpha_f \beta_g}$ are a basis.
   \end{enumerate}
  \end{corollary}
  
  \begin{proof}
   Again by Poitou--Tate global duality, we have an exact sequence
   \begin{multline*}
    0 \rTo H^1_{\mathrm{f}}(\QQ, V^*(1)) \rTo H^1(\QsQ, V^*(1)) \rTo \bigoplus_{\ell \in S} H^1_{\mathrm{s}}(\QQ_\ell, V^*(1)) \\
    \rTo H^1_{\mathrm{f}}(\QQ, V)^* \rTo  H^2(\QsQ, V^*(1)) \rTo \bigoplus_{\ell \in S} H^2(\QQ, V^*(1)) \rTo 0.
   \end{multline*}
   In the situation of the theorem, we have $H^1_{\mathrm{f}}(\QQ, V) = 0$, so the localisation map for $H^2(\QsQ, V^*(1))$ is an isomorphism.
   
   Now let $c_\ell = \dim H^2(\QQ_\ell, V^*(1))$. Using Tate's local Euler characteristic formula, for any $\ell \in S \setminus \{p\}$ we have $\dim H^1(\QQ_\ell, V^*(1)) = c_\ell$; while for $\ell = p$ we have $c_p = 0$ and $\dim H^1_{\mathrm{s}}(\Qp, V^*(1)) = 2$. Thus $\dim \bigoplus_{\ell \in S} H^1_{\mathrm{s}}(\QQ_\ell, V^*(1)) = 2 + \sum c_\ell = 2 + \dim H^2(\QsQ, V^*(1))$. However, Tate's global Euler characteristic formula gives $\dim H^1(\QsQ, V^*(1)) = 2 +\dim H^2(\QsQ, V^*(1))$. 
   
   Thus the map $H^1(\QsQ, V^*(1)) \rTo \bigoplus_{\ell \in S} H^1_{\mathrm{s}}(\QQ, V^*(1))$ is a surjection between finite-dimensional vector spaces of the same dimension, so it is injective and we conclude that $H^1_{\mathrm{f}}(\QQ, V^*(1)) = 0$. Repeating the duality argument with $V^*(1)$ in place of $V$ we now deduce that the localisation map for $H^2(\QsQ, V)$ is an isomorphism.
   
   Finally, since $H^1_{\mathrm{f}}(\QQ, V^*(1)) = 0$, we deduce that $H^1_{\mathrm{relaxed}}(\QQ, V^*(1))$ maps isomorphically to its image in $H^1_{\mathrm{s}}(\Qp, V^*(1))$; but the images of $c_1^{\alpha_f \alpha_g}$ and $c_1^{\alpha_f \beta_g}$ are a basis of $H^1_{\mathrm{s}}(\Qp, V^*(1))$, so these two classes must be a basis of $H^1_{\mathrm{relaxed}}(\QQ, V^*(1))$.
  \end{proof}
  
  \begin{corollary}
   Let $L_S(f, g, s) = \prod_{\ell \notin S} P_\ell(\ell^{-s})^{-1}$ be the $L$-function without its local factors at places in $S$. If the hypotheses of Theorem \ref{thm:BKSel} are satisfied and $L_S(f, g, 1 + j) \ne 0$, then $H^2(\QsQ, M_E(f \otimes g)^*(-j)) = 0$.
  \end{corollary}
  
  \begin{proof}
   For primes $\ell \in S$, $\ell \ne p$, let us set
   \[ P_\ell(X) = \det\left( 1 - X \operatorname{Frob}_\ell^{-1} : M_E(f \otimes g)^{I_\ell}\right), \]
   and
   \[ P_\ell^0(X) = \det\left( 1 - X \operatorname{Frob}_\ell^{-1} : M_E(f)^{I_\ell} \otimes M_E(g)^{I_\ell}\right).\]
   We define $P_p(X) = P_p^0(X) = \det\left( 1 - X \varphi : \DD_{\cris}(M_E(f \otimes g))\right)$. Then we have 
   \begin{align*}
    L_S(f, g, s) &= L(\pi_f \otimes \pi_g, s) \prod_{\ell \in S} P_\ell(\ell^{-s})\\
    &=  L(f, g, s) \prod_{\ell \in S} P_\ell^0(\ell^{-s})
   \end{align*}
   where $L(\pi_f \otimes \pi_g, s)$ (the ``primitive'' Rankin--Selberg $L$-function) and $L(f, g, s)$ (the ``imprimitive'' Rankin--Selberg $L$-function) are both holomorphic on the whole complex plane. So if $L_S(f, g, 1 + j)$ is non-zero, then we must have $P_\ell(\ell^{-1-j}) \ne 0$ for all $\ell \in S$, and $L(f, g, 1 + j) \ne 0$. 
   
   From the definition of $P_\ell(X)$, the fact that $P_\ell(\ell^{-1-j}) \ne 0$ implies that  $H^0(\QQ_\ell, M_E(f \otimes g)(1+j)) = 0$ for all $\ell \in S$. Thus $H^2(\QQ_\ell, M_E(f \otimes g)^*(-j)) = 0$ for all $\ell \in S$, and since the global $H^2$ injects into the direct product of these groups, it must also vanish.
  \end{proof}
  
  \begin{remark}
   One can check that the only values of $s$ at which the Euler factors $P_\ell(\ell^{-s})$ may vanish for some $\ell \in S$ are 
   \[ s \in \left\{ \frac{k + k'}{2}, \frac{k + k' + 1}{2}, \frac{k + k' + 2}{2}\right\}.\]
   Note that the centre of the functional equation, with our normalisations, is at $s = \frac{ k + k' + 3}{2}$.
  \end{remark}

 \subsection{Application to elliptic curves}
  
  Theorem \ref{thm:BKSel} above allows us to strengthen one of the results of \cite{KLZ1b} to cover elliptic curves which are not necessarily ordinary at $p$:
 
  \begin{theorem}
   \label{thm:sha}
   Let $E / \QQ$ be an elliptic curve without complex multiplication, and $\rho$ a 2-dimensional odd irreducible Artin representation of $G_{\QQ}$ (with values in some finite extension $L/\QQ$). Let $p$ be a prime. Suppose that the following hypotheses are satisfied:
   \begin{enumerate}[(i)]
     \item The conductors $N_E$ and $N_\rho$ are coprime;
     \item $p \ge 5$;
     \item $p \nmid N_E N_\rho$;
     \item the map $G_\QQ \to \operatorname{Aut}_{\Zp}(T_p E)$ is surjective;
     \item $\rho(\operatorname{Frob}_p)$ has distinct eigenvalues.
    \end{enumerate}
    If $L(E, \rho, 1) \ne 0$, then the group
    \[ \Hom_{\Zp[\Gal(F / \QQ)]}(\rho, \operatorname{Sel}_{p^\infty}(E / F))\]
    (where $F$ is the splitting field of $\rho$) is finite. In particular,
    \[ \Hom_{\Zp[\Gal(F / \QQ)]}(\rho, \Sha_{p^\infty}(E / F))\]
    is finite.
  \end{theorem}
  
  \begin{proof}
   This is exactly Theorem \ref{thm:BKSel} applied with $f = f_E$, the weight 2 form attached to $E$, and $g = g_{\rho}$, the weight 1 form attached to $\rho$. Compare Theorem 11.7.4 of \cite{KLZ1b}, which is exactly the same theorem under the additional hypotheses that $E$ is ordinary at $p$ and $\rho(\operatorname{Frob}_p)$ has distinct eigenvalues modulo a prime of $L$ above $p$.
  \end{proof}

 \section{Addendum: Remarks on the proof of the reciprocity law}
  \label{sect:appendix}
   
  In order to formulate the explicit reciprocity law of Theorem \ref{thm:explicitrecip}, one needs to invoke the main theorem of \cite{Urban-nearly-overconvergent}: the construction of a 3-variable $p$-adic Rankin--Selberg $L$-function as a rigid-analytic function on $V_1 \times V_2 \times \cW$, where $V_i$ are small discs in the Coleman--Mazur eigencurve surrounding classical $p$-stabilised eigenforms, and $\cW$ is weight space.
  
  Unfortunately, since the present paper was submitted, it has emerged that there are some unresolved technical issues in the paper \cite{Urban-nearly-overconvergent}, so the existence of this $p$-adic $L$-function is not at present on a firm footing. We hope that this issue will be resolved in the near future; but as a temporary expedient we explain here an unconditional proof of a weaker form of explicit reciprocity law which suffices for the arithmetic applications in the present paper.

  \subsection{A three-variable geometric $p$-adic $L$-function}
  
   We place ourselves in the situation of \S \ref{sect:explicitrecip}, so $f_\alpha, g_\alpha$ are noble eigenforms, obtained as $p$-stabilisations of newforms $f, g$ of weights $k_0 + 2, k_0' + 2$ and levels prime to $p$; and $V_1, V_2$ are small enough affinoid discs in weight space around $k_0$ and $k_0'$, over which there are Coleman families $\cF, \cG$ passing through $f_\alpha$, $g_\alpha$. We also allow the possibility that $k'_0 = -1$, $g$ is a $p$-regular weight 1 newform, and $g$ does not have real multiplication by a field in which $p$ splits. (The exceptional real-multiplication case can be handled similarly by replacing $V_2$ with a ramified covering; we leave the details to the reader.)
   
   For notational simplicity, we shall suppose that $\varepsilon_\cF \varepsilon_{\cG}$ is nontrivial, and is not of $p$-power order. Thus there is a $c > 1$ coprime to $6p N_f N_g$ for which the factor $c^2 - c^{2\mathbf{j} - \mathbf{k}-\mathbf{k}'} \varepsilon_\cF(c)^{-1} \varepsilon_{\cG}(c)^{-1}$ is a unit in $\cO(V_1 \times V_2 \times \cW)$; and we may define $\BF^{[\cF, \cG]}_{1, 1}$ (without $c$) by dividing out by this factor.
   
   We shall begin by turning Theorem C on its head, and \emph{defining} a $p$-adic $L$-function to be the output of this theorem:
   
   \begin{definition}
    We define $L_p^{\mathrm{geom}}(\cF, \cG) \in \cO(V_1 \times V_2 \times \cW)$ by
    \[ L_p^{\mathrm{geom}}(\cF, \cG) \coloneqq (-1)^{1 + \mathbf{j}} \lambda_N(\cF) \left\langle \mathcal{L}\left(\BF^{[\cF, \cG]}_{1, 1}\right), \eta_\cF \otimes \omega_{\cG}\right\rangle.\]
   \end{definition}
   
   Our goal is now to show that this geometrically-defined $p$-adic $L$-function is related to critical values of complex $L$-functions.
   
  \subsection{Values in the geometric range}
  
   By construction, for integer points of $V_1 \times V_2 \times \cW$ in the ``geometric range'' -- that is, the points $(k, k', j)$ with $0 \le j \le \min(k, k')$ -- the geometric $p$-adic $L$-function interpolates the syntomic regulators of the Rankin--Eisenstein classes. From the computations of \cite{KLZ1a}, we have the following explicit formula for these syntomic regulators. 
   
   Let $f_{k, \alpha}$ be the $p$-stabilised eigenform that is the specialisation of $\cF$ in weight $k+2$, and let $\lambda_{f_{k, \alpha}}$ be the unique linear functional on the space $S_{k + 2}^\mathrm{oc}(N_f, E)$ of overconvergent cusp forms that factors through projection to the $f_{k, \alpha}$-isotypical subspace and sends $f_{k, \alpha}$ to 1. We view $\lambda_{f_{k, \alpha}}$ as a linear functional on $S_{k + 2}^\mathrm{oc}(N, E)$, where $N = \operatorname{lcm}(N_f, N_g)$, by composing with the trace map from level $N$ to level $N_f$.
   
   \begin{theorem}[{\cite[Theorem 6.5.9]{KLZ1a}}]
    For $(k, k', j)$ in the geometric range, with $j > \tfrac{k}{2} -1$, we have
    \[ L_p^{\mathrm{geom}}(\cF, \cG)(k, k', j) = N^{k + k' - 2j} \lambda_{f_{k, \alpha}}\left[ \Pi^{\mathrm{oc}}\left( g_{k', \alpha} \cdot F^{[p]}_{k - k', k' - j + 1}\right)\right].\]
   \end{theorem}
   
   Here $F^{[p]}_{k - k', k' - j + 1}$ is a nearly-overconvergent $p$-adic Eisenstein series of weight $k-k'$ and degree of near-overconvergence $\le k-j$, whose $p$-adic $q$-expansion (image under the unit-root splitting) is given by
   \[ 
    \sum_{p \nmid n} q^n \sum_{d \mid n} d^{k-j} (d')^{j-1-k'}\left( \zeta_N^{d'} + (-1)^{k-k'} \zeta_N^{-d'}\right).
   \]
   Note that we have
   \[ F^{[p]}_{k-k', k'-j+1} = \theta^{k-j} \left(E^{[p]}_{2j-k-k'}\right), \]
   where $\theta = q \frac{\mathrm{d}}{\mathrm{d}q}$ and $E^{[p]}_\kappa$, for $\kappa \in \cW$, denotes the weight $\kappa$ overconvergent Eisenstein series
   \[ \sum_{p \nmid n} q^n \sum_{d \mid n} d^{\kappa-1}\left( \zeta_N^{d} + (-1)^\kappa \zeta_N^{-d}\right).\]
   Since $E^{[p]}_{r}$ is overconvergent of weight $r$, it follows that $g_{k', \alpha} \cdot \theta^{k-j}\left(E^{[p]}_{2j-k-k'}\right)$ lies in the space $S_{k + 2}^{\mathrm{n-oc}, k-j}(N)$ of nearly-overconvergent cusp forms of weight $k + 2$ and degree of near-overconvergence $k-j$. The condition $j > \tfrac{k}{2}-1$ implies that $k + 2 > 2(k-j)$, so Urban's overconvergent projector $\Pi^{\mathrm{oc}}$ is defined on $S_{k + 2}^{\mathrm{n-oc}, k-j}(N)$ \cite[\S 3.3.3]{Urban-nearly-overconvergent}. Thus the right-hand side of the formula in the theorem is defined.
  
  \subsection{Two-variable analytic $L$-functions}
   
   Let us now pick an integer $t \ge 0$, and set $j = k - t$ in the above formulae. Then, for varying $k$ and $k'$ (but $t$ fixed), the forms $g_{k', \alpha} \cdot \theta^t\left(E^{[p]}_{k-k'-2t}\right)$ interpolate to a 2-parameter family of nearly-overconvergent cusp forms over $V_1 \times V_2$ (of weight $\mathbf{k} + 2$ and degree $t$, where $\mathbf{k}$ is the universal weight of $V_1$). Hence we may make sense of
   \[ L_p^{(t)}(\cF, \cG) = N^{2t - \mathbf{k} + \mathbf{k}'}\lambda_{\cF}\left[ \Pi^{\mathrm{oc}}\left( \cG \cdot \theta^{t}\left(E^{[p]}_{ \mathbf{k}- \mathbf{k}'-2t}\right)\right)\right]\]
   as a meromorphic rigid-analytic function on $V_1 \times V_2$, analytic except possibly for simple poles along $V_1 \cap \{ 0, \dots, 2t-2\}$ \cite[\S 3.3.4]{Urban-nearly-overconvergent}.
      
   \begin{remark}
    The important point here is that the power of the differential operator appearing is \emph{constant} in the family; this circumvents the technical issues in \cite{Urban-nearly-overconvergent}, which concern interpolation of families where the degree of near-overconvergence is unbounded.
   \end{remark}
   
   We have the following special sets of integer points $(k, k') \in V_1 \times V_2$:
   
   \begin{enumerate}[(i)]
    \item If $k \ge \max(t, 2t-1)$ and $k' \ge k-t$, then the ``geometric'' interpolating property above applies, showing that for these values of $(k, k')$ we have
    \[ L_p^{(t)}(\cF, \cG)(k, k') = L_p^{\mathrm{geom}}(\cF, \cG)(k, k', k-t). \]
    Since such $(k, k')$ are manifestly Zariski-dense in $V_1 \times V_2$, this relation must in fact hold for all points $(\kappa, \kappa') \in V_1 \times V_2$.
    
    \item If $k' \ge 0$ and $k - k' \ge 2t + 1$, then both $g_{k', \alpha}$ and $E^{[p]}_{k-k'-2t}$ are classical modular forms (since, after possibly shrinking $V_2$, we may arrange that the specialisations of the family $\cG$ at classical weights are classical). Thus the product $g_{k', \alpha} \cdot \theta^{t}\left(E^{[p]}_{k-k'-2t}\right)$ is a classical nearly-holomorphic form, and on such forms Urban's overconvergent projector coincides with the holomorphic projector $\Pi^{\mathrm{hol}}$. This shows that the values of $L_p^{(t)}(\cF, \cG)(k, k')$ for $(k,k')$ in this range are algebraic, and they compute the values of the Rankin--Selberg $L$-function in the usual way. This also holds for $k' = -1$, as long as we assume that the weight 1 specialisation $g_{k', \alpha}$ is classical (which is no longer automatic).
   \end{enumerate}
   
   Combining these two statements, we deduce the following version of an explicit reciprocity law:
   
   \begin{theorem}
    Let $(k, k', j)$ be an integer point of $V_1 \times V_2 \times \cW$ with $k \ge 0$, $k' \ge -1$, and $\frac{k + k' + 1}{2} \le j \le k$; and suppose $f_{k, \alpha}$ and $g_{k', \alpha}$ are $p$-stabilisations of classical forms $f_k, g_{k'}$. Then we have
    \[ L_p^{\mathrm{geom}}(\cF, \cG)(k, k', j) = 
    \frac{\cE(f_k, g_{k'}, 1 + j)}{\cE(f_k) \cE^*(f_k)} \cdot \frac{j! (j- k' -1)!}{\pi^{2j-k'+1} (-1)^{k-k'} 2^{2j + 2 + k - k'} \langle f_k, f_k \rangle_{N_f}} \cdot L(f_k, g_{k'}, 1 + j),\]
    where the local Euler factors are given by
    \begin{gather*}
     \mathcal{E}(f) = \left( 1 - \frac{\beta_f}{p \alpha_f}\right), \qquad \mathcal{E}^*(f) = \left( 1 - \frac{\beta_f}{\alpha_f}\right),\\
     \mathcal{E}(f, g, 1+j) =
     \left( 1 - \frac{p^{j}}{\alpha_f \alpha_g}\right) \left( 1 - \frac{p^{j}}{\alpha_f \beta_g}\right) \left( 1 - \frac{\beta_f \alpha_g}{p^{1+j}}\right) \left( 1 - \frac{\beta_f \beta_g}{p^{1+j}}\right).
    \end{gather*}
   \end{theorem}
    
   This suffices to prove Theorem C of the introduction when $j \ge \frac{k + k' + 1}{2}$. The remaining cases of Theorem C, when $k' + 1 \le j < \frac{k + k' + 1}{2}$, are easily reduced to these cases using the functional equation.
   
   \begin{remark}
    It is important to be clear about what this argument does \emph{not} prove: we obtain no information at all about the values of the geometric $p$-adic $L$-function at points of the form $(k, k', j + \chi)$ for a nontrivial finite-order character $\chi$. In particular, we cannot determine by this method whether the specialisation of our 3-variable geometric $L$-function to $\{k_0\} \times \{k_0'\} \times \cW$ coincides with other existing constructions of a single-variable $p$-adic Rankin--Selberg $L$-function (cf.~\cite{My91}).
   \end{remark}

\newcommand{\noopsort}[1]{}
\providecommand{\bysame}{\leavevmode\hbox to3em{\hrulefill}\thinspace}
\providecommand{\MR}[1]{\relax}
\renewcommand{\MR}[1]{%
 MR \href{http://www.ams.org/mathscinet-getitem?mr=#1}{#1}.
}
\providecommand{\href}[2]{#2}
\newcommand{\articlehref}[2]{#2 (\href{#1}{link})}

\end{document}